\documentclass[a4paper,11pt]{article}

\usepackage{caption,subcaption}
\usepackage{amsfonts}
\usepackage{amssymb}
\usepackage{amsthm}
\usepackage{amsmath}
\usepackage{a4wide}
\usepackage{mathrsfs}
\usepackage{epsfig}
\usepackage{xcolor}
\usepackage{tikz,fp,ifthen}
\usepackage{float}
\usepackage[all]{xy}
\usetikzlibrary{decorations.pathreplacing}
\usetikzlibrary{shapes,arrows,chains,calc,arrows,quotes}
\usepackage{wasysym}
\mathchardef\comma=\mathcode`,

\newcommand{\pdfgraphics}{\ifpdf\DeclareGraphicsExtensions{.pdf,.jpg}\else\fi}
\usepackage{graphicx}

\DeclareMathOperator*{\argmin}{arg\,min}

\newcommand{\infconv}{\triangle}

\definecolor{darkblue}{rgb}{0.9,0.9,1}
\usepackage{color}
\definecolor{hanblue}{rgb}{0.27, 0.42, 0.81}
\definecolor{red}{rgb}{1.0, 0.0, 0.0}
\definecolor{darkred}{rgb}{0.55, 0.0, 0.0}
\definecolor{darkgreen}{rgb}{0.0,0.55, 0.0}
\definecolor{lightgray}{rgb}{0.95,0.95,0.95}
\usepackage[colorlinks, citecolor=hanblue,linkcolor=red]{hyperref}

\newcommand{\rg}{\mathrm{rg}}
\newcommand{\domain}{\mathrm{dom}}
\newcommand{\weakstar}{\stackrel{*}{\rightharpoonup}}

\newcommand{\st}{\,|\,}
\newcommand{\bigst}{\,\big |\,}
\newcommand{\Bigst}{\,\Big |\,}
\newcommand{\Biggst}{\,\Bigg |\,}
\newcommand{\Sym}{{\rm Sym}}
\newcommand{\mI}{\mathcal{I}}
\newcommand{\mM}{\mathcal{M}}

\newcommand{\height}{\mathcal{H}}

\newcommand{\symgrad}{\mathcal{E}} %
\DeclareMathOperator{\Id}{Id} %

\newcommand{\emb}{I}
\newcommand{\Ic}{\mathcal{I}}

\newcommand{\BV}{\text{BV}}
\DeclareMathOperator{\TV}{TV}
\DeclareMathOperator{\TGV}{TGV}
\DeclareMathOperator{\BD}{BD}

\usepackage{mathtools}
\usepackage{hyperref}

\renewcommand{\div}{{\rm div}\,}

\newcommand{\R}{\mathbb{R}}
\newcommand{\Z}{\mathbb{Z}}

\newcommand{\N}{\mathbb{N}}
\newcommand{\M}{\mathcal{M}}

\newcommand{\Gc}{\mathcal{G}}

\newcommand\restr[2]{{%
  \left.\kern-\nulldelimiterspace %
  #1 %
  \vphantom{\big|} %
  \right|_{#2} %
  }}

\usepackage{enumitem}

\numberwithin{equation}{section}

\theoremstyle{plain}

\newtheorem{teo}{Theorem}[section]

\newtheorem{lemma}[teo]{Lemma}
\newtheorem{prop}[teo]{Proposition}
\newtheorem{cor}[teo]{Corollary}

\newtheorem{dfnz}[teo]{Definition}

\theoremstyle{remark}
\newtheorem{rem}[teo]{Remark}

\usepackage{scalerel}
\DeclareMathOperator*{\bigplus}{\scalerel*{+}{\sum}}

\def\nsz{0.4} %
\def\xsz{4} %
\def\ysz{2} %

\tikzset{
  rgraph/.style={ %
    xscale=\xsz, %
    yscale=\ysz
    },
  edge/.style={ %
    ->,
    >=triangle 60, %
    shorten <= \nsz cm,
    shorten >= \nsz cm,
    },
  dir/.style={ %
    shorten <= 0,
    shorten >= 0,
    >=angle 60
    },
  pics/edge/.style={code={ %
	\foreach \X [count=\i from -1] in {#1}
	{
		\ifthenelse{-1=\i}{ %
			\draw[fill=lightgray,draw=none] (-2*\nsz,-1.6*\nsz) rectangle (2*\nsz,1.6*\nsz);
		    \draw[fill=white] (-1.9*\nsz,-0.6*\nsz) rectangle (1.9*\nsz,0.6*\nsz);
		    \draw[->,dir] (1.9*\nsz,0.6*\nsz) -- (2.6*\nsz,0.6*\nsz);
		    \draw[<-,dir] (-2.6*\nsz,-0.6*\nsz) -- (-1.9*\nsz,-0.6*\nsz);
		}{}
		\node[fill=none,inner sep=0cm] at (0*\nsz*\i,1.1*\nsz*\i) {\tiny \X};
	}}},
  pics/node/.style={code={ %
  	\coordinate (-n) at (0,\nsz);
    \coordinate (-ne) at (45:\nsz);
    \coordinate (-e) at (\nsz,0);
    \coordinate (-se) at (-45:\nsz);    
    \coordinate (-s) at (0,-\nsz);
    \coordinate (-sw) at (-135:\nsz);    
    \coordinate (-w) at (-\nsz,0);
    \coordinate (-nw) at (135:\nsz);    
    \coordinate () at (135:0);    
	\draw[pic actions,thick] (0,0) circle (\nsz);
	\foreach \X [count=\i from -1] in {#1}
	{
		\ifthenelse{-1 = \i}{
			\node at (0,-1.5*\nsz*\i) {\tiny \X};			
			\draw[thick] (-\nsz,0) -- (\nsz,0);
			\draw[thick] (0,-\nsz) -- (0,\nsz);
			}{
			\draw[fill=white,draw=none] (0,0) circle (0.98*\nsz); 
			\node at (0,0) {\tiny \X};			
			}
	}}}
}

\tikzset{
  rgraphtgv/.style={ %
    xscale=\xsz, %
    yscale=\ysz
    },
  edgetgv/.style={ %
    ->,
    >=triangle 60, %
    shorten <= \nsz cm,
    shorten >= \nsz cm,
    },
  dirtgv/.style={ %
    shorten <= 0,
    shorten >= 0,
    >=angle 60
    },
  pics/edgetgv/.style={code={ %
	\foreach \X [count=\i from -1] in {#1}
	{
		\ifthenelse{-1=\i}{ %
			\draw[fill=lightgray,draw=none] (-1.9*\nsz,-1.6*\nsz) rectangle (1.9*\nsz,1.6*\nsz);
		    \draw[fill=white] (-2.8*\nsz,-0.6*\nsz) rectangle (2.8*\nsz,0.6*\nsz);
		    \draw[->,dirtgv] (2.8*\nsz,0.6*\nsz) -- (3.4*\nsz,0.6*\nsz);
		    \draw[<-,dirtgv] (-3.4*\nsz,-0.6*\nsz) -- (-2.8*\nsz,-0.6*\nsz);
		}{}
		\node[fill=none,inner sep=0cm] at (0*\nsz*\i,1.1*\nsz*\i) {\tiny \X};
	}}},
  pics/node/.style={code={ %
  	\coordinate (-n) at (0,\nsz);
    \coordinate (-ne) at (45:\nsz);
    \coordinate (-e) at (\nsz,0);
    \coordinate (-se) at (-45:\nsz);    
    \coordinate (-s) at (0,-\nsz);
    \coordinate (-sw) at (-135:\nsz);    
    \coordinate (-w) at (-\nsz,0);
    \coordinate (-nw) at (135:\nsz);    
    \coordinate () at (135:0);    
	\draw[pic actions,thick] (0,0) circle (\nsz);
	\foreach \X [count=\i from -1] in {#1}
	{
		\ifthenelse{-1 = \i}{
			\node at (0,-1.5*\nsz*\i) {\tiny \X};			
			\draw[thick] (-\nsz,0) -- (\nsz,0);
			\draw[thick] (0,-\nsz) -- (0,\nsz);
			}{
			\draw[fill=white,draw=none] (0,0) circle (0.98*\nsz); 
			\node at (0,0) {\tiny \X};			
			}
	}}}
}

\tikzset{
  predualrgraph/.style={ %
    xscale=\xsz, %
    yscale=\ysz
    },
  predge/.style={ %
    ->,
    >=triangle 60, %
    shorten <= \nsz cm,
    shorten >= \nsz cm,
    },
  predir/.style={ %
    shorten <= 0,
    shorten >= 0,
    >=angle 60
    },
  pics/predge/.style={code={ %
	\foreach \X [count=\i from -1] in {#1}
	{
		\ifthenelse{-1=\i}{ %
			\draw[fill=lightgray,draw=none] (-2*\nsz,-1.6*\nsz) rectangle (2*\nsz,1.6*\nsz);
		    \draw[fill=white] (-1.9*\nsz,-0.6*\nsz) rectangle (1.9*\nsz,0.6*\nsz);
		    \draw[<-,predir] (1.9*\nsz,0.6*\nsz) -- (2.6*\nsz,0.6*\nsz);
		    \draw[->,predir] (-2.6*\nsz,-0.6*\nsz) -- (-1.9*\nsz,-0.6*\nsz);
		}{}
		\node[fill=none,inner sep=0cm] at (0*\nsz*\i,1.1*\nsz*\i) {\tiny \X};
	}}},
  pics/node/.style={code={ %
  	\coordinate (-n) at (0,\nsz);
    \coordinate (-ne) at (45:\nsz);
    \coordinate (-e) at (\nsz,0);
    \coordinate (-se) at (-45:\nsz);    
    \coordinate (-s) at (0,-\nsz);
    \coordinate (-sw) at (-135:\nsz);    
    \coordinate (-w) at (-\nsz,0);
    \coordinate (-nw) at (135:\nsz);    
    \coordinate () at (135:0);    
	\draw[pic actions,thick] (0,0) circle (\nsz);
	\foreach \X [count=\i from -1] in {#1}
	{
		\ifthenelse{-1 = \i}{
			\node at (0,-1.5*\nsz*\i) {\tiny \X};			
			\draw[thick] (-\nsz,0) -- (\nsz,0);
			\draw[thick] (0,-\nsz) -- (0,\nsz);
			}{
			\draw[fill=white,draw=none] (0,0) circle (0.98*\nsz); 
			\node at (0,0) {\tiny \X};			
			}
	}}}
}

\begin{document}
\pdfgraphics 

\pdfgraphics %

\title{Regularization Graphs --- A unified framework for variational regularization of inverse problems}
\author{Kristian Bredies\footnote{Kristian Bredies, Institute of Mathematics and Scientific Computing,
    University of Graz, Heinrichstra\ss{}e 36, A-8010 Graz,
    Austria. Email: \texttt{kristian.bredies@uni-graz.at}} \and Marcello
  Carioni\footnote{Marcello Carioni, University of Cambridge, Department of Applied Mathematics and Theoretical Physics, Wilberforce Road, Cambridge
CB3 0WA, UK Email: \texttt{mc2250@maths.cam.ac.uk}} \and    
    Martin Holler \footnote{Martin Holler, Institute of Mathematics and Scientific Computing,  University of Graz, Heinrichstra\ss{}e 36, A-8010 Graz, Austria. Email: \texttt{martin.holler@uni-graz.at}
    }}

\date{}

\maketitle

\begin{abstract}
\sloppy We introduce and study a mathematical framework for a broad class of regularization functionals for ill-posed inverse problems: \emph{Regularization Graphs}. Regularization graphs allow to construct functionals using as building blocks linear operators and convex functionals, assembled by means of operators that can be seen as generalizations of classical infimal convolution operators.
This class of functionals exhaustively covers existing regularization approaches and it is flexible enough to craft new ones in a simple and constructive way.
We provide well-posedness and convergence results with the proposed class of functionals in a general setting. Further, we consider a bilevel optimization approach to learn optimal weights for such regularization graphs from training data. We demonstrate that this approach is capable of optimizing the structure and the complexity of a \emph{regularization graph}, allowing, for example, to automatically select a combination of regularizers that is optimal for given training data.
\vskip .3truecm \noindent Key words: Inverse problem, regularization functional, graph data structure, bilevel optimization.

  \vskip.1truecm \noindent 2020 Mathematics Subject Classification:
65J20, %
65K10, %
49J45.	%

\end{abstract}

\section{Introduction}

In the last decades, a significant part of inverse problems theory has revolved around constructing suitable regularization approaches that allow for a reliable solution of ill-posed inverse problems. 
Among those, energy-based methods such as Tikhonov regularization \cite{Tikhonov1} have been successful both with respect to mathematical guarantees, e.g., on well-posedness and stability, and with respect to practical performance in applications.
An important cornerstone of energy-based methods are regularization functionals, which are responsible for stabilizing the ill-posed inversion of the forward model and for incorporating prior knowledge, such as smoothness, on the sought solution. 
The later is relevant in particular when dealing with highly structured data such as image data, where a suitable inclusion of prior knowledge makes a significant difference regarding the overall performance of the resulting method.

In this context, non-smooth sparsity-based methods building on measures, measure-valued differential operators, basis transforms or frames have become very popular. 
Besides the celebrated total variation (TV) functional \cite{rudinosherfatemi}, those include methods building on higher-order derivatives such as second-order TV \cite{hinterberger2006boundedhessian_mh}, infimal-convolution-based approaches \cite{chambollelions} or the total generalized variation (TGV) functional \cite{tgvbredieskunischpock}, see \cite{Holler19_ip_review_paper} for a recent review. Transform-based methods include wavelet-, curvelet- or shearlet transforms \cite{Kutyniok12_shearlet_book_mh, mallat2009wavelettour_mh,candes2000curvelets_mh} as well as learned dictionaries \cite{elad2010sparse}. %

Also more specific approaches tailored, for instance, to model certain oscillations \cite{estellers2015, bredies2018oscitgv,hollerkunisch2014, Lefkimmiatis2015StructureTT,   meyeroscillating2001,  Parisotto_Lellmann_Masnou_Schonlieb_2018, Parisotto2018HigherOrderTD} or texture \cite{hollerlearningconvolution2019}, as well as different combinations of existing methods exist, such as TV and second-order TV \cite{papafitsoros2014firstandsecondorder_mh}, higher-order regularizers \cite{brinkmann2019unified, chan2010, papafitsoros2014firstandsecondorder_mh}, TV-type functionals with curvelets or shearlets \cite{Scherzer11_tv_plus_curvelet_mh,Guo14tgvshearlet,Gao19infimal}, a combination of different transform-based approaches \cite{Kutyniok13_mh} or the infimal convolution of TV with $L^p$-norms \cite{burger2015infimal_infty,burger2016infimal_finite_p}. We refer to \cite{Holler19_ip_review_paper,aubert2006mathematical_mh,bredies2018mathematical, Lebrun12denoising_review_mh,scherzer2009variationalmethods_mh} for a review of a subset of the plethora of existing methods.

While all these approaches share the goal of providing a model-based regularization for inverse problems, the way and extent to which they are developed and analyzed is rather different and often application-specific. 
Moreover, the choice of any of such methods is mostly done manually. A systematic approach for the analysis and the automatic, data-based design of regularization functionals that covers a broad class of existing methods does not exist to date.

With introducing the framework of regularization graphs, we aim to provide a step in this direction. A regularization graph can be described as a weighted, directed graph together with a collection of functionals and operators associated with the nodes and the edges of the graph, respectively.
Such structure allows to define regularization functionals via a rather arbitrary combination of linear operators and functionals, e.g., via variable splittings or summations. In particular, both the sum and a (generalized) infimal convolution of the functionals associated with two regularization graphs can be 
formulated as a regularization graph functional, where the underlying  graph is obtained by properly combining the two original ones.

This yields a flexible framework for designing new regularization functionals or combining existing ones, e.g., via infimal convolution. Moreover, by associating weights to the edges of such graphs, a learning 
of both the parameters associated with such functionals as well as the structure of the underlying graph is possible. The latter in particular allows to automatically select optimal regularization functionals from a set of possible choices within a bilevel approach.  

A prototypical example of a regularization graph with nodes $V = \{1,2,3,4\} $, directed edges $E = \{ (1,2),(2,3),(2,4)\}$ and weights $(\alpha_e)_{e \in E}$ is provided in Figure \ref{fig:graph_example_intro}.
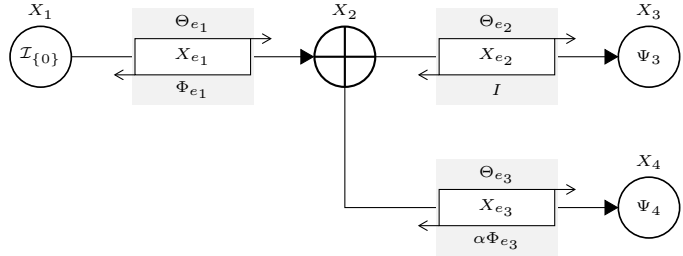
\begin{figure}[t] 
\centering
\begin{tikzpicture}[rgraph]

\draw (0,0) pic (p1) {node={$X_1$,$\mI_{\{0\}}$}};
\draw (1,0) pic (p2) {node={$X_2$}};

\draw (2,0) pic (p3) {node={$X_3$,$\Psi_3$}};
\draw (2,-1) pic (p4) {node={$X_4$,$\Psi_4$}};

\draw[edge] (p1) -- pic{edge={$\Phi_{e_1}$,$X_{e_1}$,$\Theta_{e_1}$}} (p2);
\draw[edge] (p2) -- pic{edge={$\emb$,$X_{e_2}$,$\Theta_{e_2}$}} (p3);
\draw[edge] (p2) -- (1,-1) -- pic{edge={$\alpha \Phi_{e_3}$,$X_{e_3}$,$\Theta_{e_3}$}} (p4);

\end{tikzpicture}
\caption{Example of a \emph{regularization graph} with nodes $V = \{1,2,3,4\}$,  edges $E = \{ (1,2),(2,3),(2,4) \} = \{ e_1,e_2,e_3\}$ and weights $(\alpha_{(1,2)},\alpha_{(2,3)},\alpha_{(2,4)}) = (1,1,\alpha)$. See Remark \ref{rem:graphical} for a detailed interpretation.} \label{fig:graph_example_intro}
\end{figure}
 Here, the operators and the functionals associated to the nodes and the edges of the graph are defined as follows. For $n \in V$ and $e \in E$, $X_n$ and $X_e$ are suitable Banach spaces, $\Phi_e$ are bounded linear operators, $\Theta_e$ are (possibly unbounded) closed range operators and $\Psi_n$ are convex functionals. The spaces $X_n$ and $X_e$ are called \emph{node spaces} and \emph{edge spaces}, respectively, while the $\Psi_n$ are called \emph{node functionals}. Further, we call $\Theta_e$ \emph{forward operators} as they map from the edge space $X_e$ to the direct successor node space  $X_n$. Similarly, we call $\Phi_e$  \emph{backward operators} as they map from the edge space to the direct predecessor node space. Variables $\{ w_{e_1},w_{e_2},w_{e_3}\}$ associated with the edges of the graph, on which both the forward and backward operators are evaluated, are called \emph{edge variables}. 
 
Notice that in our example the root node $1$ and the splitting node $2$ correspond to the functional $\mI_{\{0\}}$, i.e., $\Psi_1 = \Psi_2 = \mI_{\{0\}}$ and $\Phi_{e_2} = \emb$ denotes a continuous embedding of $X_{e_2}$ into $X_2$; see Remark \ref{rem:graphical} below for details. Also note that the weights $(\alpha_e)_{e \in E}$ associated to the graph are depicted in Figure \ref{fig:graph_example_intro} as scalar factors in front of the backward operators $(\Phi_e)_{e \in E}$, where we use the convention that fixed, trivial weights $\alpha_e=1$ are not depicted explicitly. We also remark that, besides the notation $\alpha_e$ for $e \in E$, for specific regularization graphs the non-trivial weights will be often numbered independently of the edge they are associated with; see Figure \ref{fig:ex_reg_graph}.

The regularization functional $R_\alpha : X_1 \rightarrow [0,+\infty]$ corresponding to such a regularization graph is given via minimizing over the involved edge variables $\{ w_{e_1},w_{e_2},w_{e_3}\}$ as
\begin{align*}
R_\alpha(u)  =  & \mathop{\inf_{(w_{e_i})_{i}}}_{} 
  \mI_{\{0\}}(u - \Phi_{e_1}w_{e_1}) + \mI_{\{0\}}(\Theta_{e_1} w_{e_1} -  w_{e_2}- \alpha \Phi_{e_3}w_{e_3}) + \Psi_3(\Theta_{e_2} w_{e_2}) + \Psi_4 (\Theta_{e_3} w_{e_3}) \\
 =  & \mathop{\inf_{w_{e_1},w_{e_3}}}_{}  \Psi_3(\Theta _{e_2} ( \Theta_{e_1} w_{e_1} - \alpha \Phi _{e_3}w_{e_3})) + \Psi_4(\Theta_{e_3}w_{e_3}) \quad \text{s.t. } u = \Phi_{e_1} w_{e_1}.
\end{align*}
The structure defined in this example is a regularization graph under mild additional conditions, most importantly weak* lower semicontinuity and coercivity of $\Psi_3$ and $\Psi_4$, and closedness of the range of each $\Theta_e$, which, for instance, still allows the $\Theta_e$ to be densely defined differential operators and the $\Phi_e$ to be synthesis operators for a given dictionary or frame. The non-trivial weight $\alpha$ allows to adapt the structure of the graph by removing edges, as with $\alpha=0$ and supposing for example that $\Psi_4$ vanishes in zero, we obtain $R_0 (u) = \inf_{ \{\Phi_{e_1} w_{e_1} = u \}}  \Psi_3(\Theta _{e_2}  \Theta_{e_1} w_{e_1})$.

The general structure of a regularization graph is defined in Section \ref{sec:notation} and examples of existing regularization approaches that are included in this setting are provided in Section \ref{sec:examples} and listed in the Appendix. Here, the main conditions on the involved functionals and operators are that the forward operators $\Theta_e$ have closed range (i.e., satisfy a Poincaré-type estimate), that the backward operators $\Phi_e$ are continuous and that the involved node functionals $\Psi_n$ are coercive. 

Under these conditions, we prove well-posedness, stability and convergence results for the application of regularization graphs in a general inverse problem setting. Moreover, we develop a bilevel approach that allows to learn the structure of an optimal graph for a given set of training data and show well-posedness of the resulting non-convex optimization problem.
\medskip

\noindent \textbf{Contribution of the paper in relation to the state of the art.} In a rather abstract setting, general conditions on regularization functionals that allow to guarantee well-posedness, stability and convergence are of course well-known, see for instance \cite{grasmair2011residual, hofmann}. Those, however, are conditions on the overall functionals rather than their building blocks and their verification is often at the same level of difficulty than the results themselves. Furthermore, they do not allow to easily combine different approaches without re-checking the underlying conditions. More specific results also exist, but deal with particular settings such as higher-order regularization \cite{brediesholler2014,brinkmann2019unified}.

More related to the aim of this paper are some works on bilevel optimization, see for instance \cite{calatronibilevel2017} for a review. 
In the probably most closely related work \cite{delosreyes2017}, the authors consider a general bilevel framework that includes TV, the infimal convolution of first and second order TV functionals as well as the TGV functional as particular cases. In contrast to \cite{delosreyes2017}, however, where essentially well-posed linear inverse problems are considered, i.e., those with closed range forward operator, our work is generally applicable to any bounded forward operator. In particular, we do not require closed range and allow for genuinely ill-posed inverse problems, a generalization that is the main source of difficulty for the analysis in this context.

A second, closely related work is the preprint \cite{davoli2019adaptive}. There, the authors consider a bilevel scheme for learning parameters and operators in a TGV-like functional. They provide conditions on the involved operators under which they show well-posedness for a bilevel approach in image denoising. As application they consider an interpolation between a symmetrized and a non-symmetrized differential operator in the second order TGV functional.
Besides being applicable to inverse problems beyond denoising, our work is different to \cite{davoli2019adaptive} in allowing a more flexible combination of linear operators and functionals, far beyond the cascadic structure of TGV. Further, our framework allows for an automatic selection from different choices of existing regularization functionals but also, for instance, to select an optimal order in TGV regularization.

\medskip

\noindent \textbf{Organization of the paper.}  The paper is organized as follows. In Section \ref{sec:notation} we give the precise definition of \emph{regularization graphs} clarifying the main assumptions on the linear operators, the functionals and the involved Banach spaces that yield the results of our work. Also, we provide several examples of existing regularization approaches that can be constructed using a suitable \emph{regularization graph}. In Section \ref{sec:algebraic} we provide basic algebraic properties of regularization graphs, in particular a recursive representation that will be quite useful later on. In Section \ref{sec:analytic_properties} we provide the main analytic properties of functionals associated with regularization graphs that will be the basis for subsequent results on the regularization of inverse problems and bilevel optimization. In particular, we show that any such functional is weak* lower semi-continuous and coercive up to a finite dimensional space. In Section \ref{sec:predual} we provide an equivalent predual formulation of regularization graphs.  Also, the connection to well-known predual representations of existing regularization approaches is made. While the results of this section will not be needed in the subsequent theory, they are nevertheless of interest on their own, in particular in view of optimality conditions and duality-based algorithms.

Section \ref{sec:regularizationinverse} then provides well-posedess and convergence results for the application of regularization graphs to the regularization of linear inverse problems. We focus on linear inverse problems since this allows for a compact presentation of the results without any additional assumptions on the forward model except for continuity. Nevertheless, the analytic results of Section \ref{sec:analytic_properties} also allow to show well-posedness for non-linear inverse problems under standard assumptions on the forward model such as in \cite{hofmann}.
In Section \ref{sec:bilevel}, we develop and analyze a bilevel framework for learning the weights of regularization graphs. In particular, we show well-posedness and an example for a bilevel approach that allows to select optimal regularizers from a set of possible choices by learning zero-weights in the graph. An appendix further provides a list that shows how a selection of existing regularization functionals can be represented by regularization graphs.

\section*{Acknowledgements}

KB, MC and MH gratefully acknowledge support by the Austrian Science Fund (FWF) through the project P 29192 ``Regularization graphs for variational imaging''. MC is supported by the Royal Society (Newton International Fellowship NIF\textbackslash R1\textbackslash 192048 ``Minimal partitions as a robustness boost for neural network classifiers''). The Institute of Mathematics and Scientific Computing in Graz, to which KB and MH are affiliated, is a member of NAWI Graz (\texttt{https://www.nawigraz.at/en/}). KB and MH are further members of/associated with BioTechMed Graz (\texttt{https://biotechmedgraz.at/en/}).

\section{Notation and assumptions}\label{sec:notation}

In this section we define the underlying setting and assumptions used in the paper.
The structure of a general regularization functional will be represented by a directed graph $G=(V,E)$, where $V$ is a non-empty finite set of nodes not containing $0$ and $E \subset (V \times V) \setminus \{(n,n) : n \in V\}$ are the edges. We assume that $G$ has a tree structure and that a root node $ \hat{n} \in V$ exists, i.e., we assume that $G$ contains no cycles and that for each $n \in V$ there exist edges $((n_{i-1},n_i))_{i=1}^M$ in $E$ such that $n_M = n$ and $n_0 = \hat{n}$.

We call a set $F \subset E$ a chain (of length $M>0$ with root $n_0$) if $F = \{ (n_{i-1},n_i) \st i=1,\ldots ,M, \, n_i \neq n_j \text{ for } i \neq j\}$.
Further, for $n \in V$, we denote by $n^-$ the node such that $(n^-,n) \in E$ if $n$ is not the root node of the graph and $n^-=0$ otherwise, noting that $n^-$ is well defined due to the tree structure of $G$.

To any graph $G=(V,E)$ we associate a family of Banach spaces spaces $(X_n)_{n \in V}$ with the nodes and a family of Banach spaces $(X_e)_{e \in E}$ with the edges. Further, we associate the following functionals and operators with $G$.

\begin{itemize}
\item A convex functional $\Psi_n : X_n \rightarrow [0,
\infty]$ for every $n\in V$.  
\item A linear \emph{forward} operator $
\Theta_{(n,m)} :\domain(\Theta_{(n,m)}) \subset X_{(n,m)} \rightarrow X_m$ for every $ (n,m)\in E$.
\item A linear \emph{backward} operator $
\Phi_{(n,m)} :  X_{(n,m)} \rightarrow X_n$ for every $ (n,m)\in E$. %
\end{itemize}

We suppose that each $X_n$, $n\in V$ and each $X_e$, $e\in E$ admits a predual space denoted by $X^\#_n$ and $X^\#_e$, respectively, and
make  the following assumptions on $(\Psi_n)_n$, $(\Theta_e)_e$ and $(\Phi_e)_e$:
\begin{enumerate}[label=(H\arabic*)]
\item \label{ass:lsc_functional} $\Psi_n$ is weak* lower-semicontinuous for every $n\in V$.
\item \label{ass:coerc_functional} For every $n\in V$, $\Psi_n$ is coercive, i.e., for any sequence $(v^k)_k$ in $X_n$ it holds
\[
\|v^k\|_{X_n} \rightarrow + \infty \quad \Rightarrow  \quad  \Psi_n(v^k) \rightarrow +\infty \qquad \text{as} \qquad k \rightarrow +\infty.
\]
\item \label{ass:psi_0_is_0} $\Psi_n(0) = 0$ for every $n \in V$.
\item \label{ass:closed_fwd_operator} $\Theta_{e}$ is weak* closed
for every $e\in E$.
\item \label{ass:ker_fwd_operator}$\ker(\Theta_e)$ is finite dimensional for every $e\in E$.
\item \label{ass:coerc_fwd_operator} For every $e=(n,m) \in E$ there exists $C>0$ and a continuous, linear projection $P_{\ker(\Theta_e)}:X_e \rightarrow \ker(\Theta_e)$ such that 
\begin{equation}\label{eq:poincinassum}
\| w - P_{\ker(\Theta_e)}w \|_{X_e} \leq C \|\Theta_e w\|_{X_m}
\end{equation}
for every $w \in \domain(\Theta_e)$.
\item\label{ass:cont_bkw_operator} $\Phi_e$ is weak* to weak* continuous for every $e \in E$.
\item\label{ass:weak_star_compactness} Bounded sequences in $X_e$ and $X_n$ admit weak* convergent subsequences for every $e \in E$, $n \in V$.
\end{enumerate}

\begin{rem}\label{rem:initialremark} We can observe the following details in the above assumptions:
\begin{itemize}
\item Hypothesis \ref{ass:ker_fwd_operator} implies the existence of a linear and continuous projection on $\ker(\Theta_e)$. 
\item Under Hypothesis \ref{ass:closed_fwd_operator},  \ref{ass:ker_fwd_operator} and \ref{ass:weak_star_compactness}, Hypothesis \ref{ass:coerc_fwd_operator} is equivalent to $\Theta_e$ having closed range, see Lemma \ref{lem:equivalence_coercivity_closed_range} in the Appendix for a proof.
\item Hypothesis \ref{ass:cont_bkw_operator} implies the existence of a continuous predual operator for $\Phi_e$ for each $e \in E$. Consequently, each $\Phi_e$ is continuous as well (see for instance \cite[Remark 3.2]{pikka}).
\item Hypothesis \ref{ass:weak_star_compactness} holds whenever $X_e$ and $X_n$ are reflexive or dual spaces of separable spaces. In case of reflexivity, the notion of weak* convergence can be replaced by weak convergence in all assumptions.
\item Note that, since the $\Psi_n$ are convex, assumption \ref{ass:psi_0_is_0} implies that $\Psi_n(\lambda v) \leq \lambda \Psi_n(v)$ for any $v \in X_n$, $\lambda \in (0,1]$ and $n \in V$. This consequence of assumption \ref{ass:psi_0_is_0} will be needed in the context of varying the weights of a regularization graph. For well-posedness results such as existence and stability as presented in this paper, however, assumption \ref{ass:psi_0_is_0} is not necessary and could be dropped.
\end{itemize}
\end{rem}

We also note that Hypothesis \ref{ass:coerc_functional} implies a coercivity estimate as follows.

\begin{rem}\label{rem:equivcoercivity}
Hypothesis \ref{ass:coerc_functional} holds if and only if there exists $C>0$ and $D \in \R$ such that $\|v\|_{X_n} \leq C\Psi_n(v)+D$ for every $v\in X_n$.
A proof for this can be found for example in \cite[Fact 4.4.8]{borwein}.
\end{rem}

We are now in a position to define the main objects of interest in this paper: Regularization graphs and associated regularization functionals. To this aim, we allow for weights of the form $(\alpha_e)_{e \in E}$ with $\alpha_e \in [0,\infty)$ for all $e \in E$.

\begin{dfnz}[Regularization graph and associated regularization functional]\label{def:reg_graph}
Given $G=(V,E)$ a directed graph with tree structure and root node $\hat{n}$, and the associated spaces, functionals and operators as in Section \ref{sec:notation} such that the hypotheses \ref{ass:lsc_functional} to \ref{ass:weak_star_compactness} hold, the \emph{structure of a regularization graph} is defined as  the tuple $\Gc = (G, (\Psi_n)_{n\in V}, (\Theta_e)_{e\in E}, (\Phi_e)_{e\in E})$. Together with a family of weights $\alpha = (\alpha_e)_{e\in E}$, a \emph{regularization graph} is then defined as the tuple $\Gc_\alpha = (\Gc,\alpha)$.

For any such regularization graph $\Gc_\alpha$, the associated regularization functional $R_\alpha = R(\Gc_\alpha): X_{\hat{n}} \rightarrow [0,\infty]$ (called \emph{regularization graph functional}) is defined as
\begin{equation}\label{defregularizer}
\begin{aligned}
R_\alpha(u)  
& = \inf_{  \substack{ (w_e)_{e \in E} \\ w_e \in \domain(\Theta_e) }} \sum_{n \in V} \Psi_{n} 
\Big( \Theta_{(n^-,n)} w_{(n^-,n)} - \sum_{(n,m) \in E} \alpha_{(n,m)} \Phi_{(n,m)} w_{(n,m)}
\Big) \\
 & = \inf  \Bigg\{ \sum_{n \in V} \Psi_n(v_n)  
 \  \nonumber \Bigst \text{for all }  n\in V, \ e \in E \text{ there exist } w_e \in \domain(\Theta_e) : \\ 
&   \quad \qquad \qquad\qquad\qquad  v_n = \Theta_{(n^-,n)} w_{(n^-,n)} - \sum_{(n,m) \in E} \alpha_{(n,m)}\Phi_{(n,m)} w_{(n,m)} \  \Bigg\}
\end{aligned}
\end{equation}
where we set $\Theta_{(\hat{n}^-,\hat{n})}= \Id$ and $w_{(\hat{n}^-,\hat{n})} = u$. Note that, in case of a trivial regularization graph, i.e., $V = \{\hat n\}$, $E = \emptyset$, we set $R_\alpha(u) = \Psi_{\hat n}(u)$.
\end{dfnz}
\begin{rem}[Weights] Generically, to each edge $e$ within the graph structure of a regularization graph is associated a weight $\alpha_e$. In many cases, e.g., when node functionals only take values in $\{0,\infty\}$, this leads to an overparametrization of the associated regularization functional. To avoid this, we often fix a subset of weights to be equal to $1$ already when defining a regularization graph. Such weights are called \emph{trivial weights}, and the other, non-trivial weights that might still vary are often numbered independently of the edge they are associated with.
\end{rem}
\begin{rem}[Graphical representation of regularization graphs] \label{rem:graphical}
Let us revisit the prototypical graphical representation of a regularization graph in Figure \ref{fig:graph_example_intro}. There, the circles represent nodes, with the node space shown above the circle and the functional $\Psi_n$ inside. A splitting node is represented by a $\oplus$ and is associated with the functional $\mI_{\{0\}}$. The rectangles denote the edges, with the edge space shown in the center, the forward operator $\Theta_e$ shown at the top and the backward operator $\Phi_e$ at the bottom. The weights $(\alpha_e)_{e \in E}$ are depicted as scalar factors in front of the backward operators $(\Phi_e)_{e \in E}$ (with arbitrary numbering independent of their position in the graph), and we use the convention that omitted weights at an edge $e$ correspond to trivial weights $\alpha_e = 1$.

The arrows connect the nodes. At each node $n$, the node functional $\Psi_n$ is evaluated at $\Theta_e$ of the variable from the incoming edge $e$ minus the sum of all $\Phi_e$ applied to the variables $w_e$ from the outgoing edges $\{e = (n,m) \in E: m \in V\}$. The regularization graph functional is given by minimizing this construction over all edge variables in the domain of the corresponding operators $(\Theta_e)_{e \in E}$.
\end{rem}

We can also obtain a more compact representation of $R_\alpha $ as follows. Define the spaces
 \begin{equation*}
 X_V = \bigtimes_{n\in V} X_n  \quad \mbox{and} \quad  X_E = \bigtimes_{e\in E} X_e
 \end{equation*}
 equipped with the product norm, 
and the operator $\Lambda_{\alpha}: \domain(\Lambda_{\alpha}) \subset X_E \rightarrow X_V$ as
\begin{equation}
\label{eq:vectorized_operator_graph}
(\Lambda_{\alpha} w)_n :=
\left\{
\begin{array}{ll}
\displaystyle\Theta_{(n^-,n)} w_{(n^-,n)} - \sum_{(n,m) \in E} \alpha_{(n,m)} \Phi_{(n,m)} w_{(n,m)} & \text{for } n \in V\setminus \{\hat{n}\},\\
\displaystyle - \sum_{(\hat{n},m) \in E} \alpha_{(\hat{n},m)} \Phi_{(\hat{n},m)} w_{(\hat{n},m)} &  \text{for } n = \hat{n}
\end{array}
\right.
\end{equation}
for every $n\in V$ and $w = (w_e)_{e \in E} \in \domain(\Lambda_\alpha)$ where $\domain(\Lambda_\alpha) = \bigtimes_{e\in E} \domain (\Theta_e)$.
Then we can write the functional $R_\alpha $ associated with the regularization graph $\Gc_\alpha$ as
\begin{align*}
R_{\alpha}(u) & = \inf \Big\{ \Psi_{\hat{n}}(u + (\Lambda_{\alpha} w)_{\hat{n}}) +  \sum_{n\in V \setminus \{ \hat{n}\}}  \Psi_n((\Lambda_{\alpha} w)_n) \bigst  w \in \domain(\Lambda_{\alpha})
 \Big\}.
\end{align*}
For notational convenience we also define the functional $\Psi_u : X_V \rightarrow [0,+\infty]$ for $u \in X_{\hat{n}}$ as 
\begin{equation} \label{eq:vectorized psi_reduced}
\Psi_u(v) := \Psi_{\hat{n}}(u + v_{\hat{n}}) + \sum_{n\in V\setminus \{\hat{n}\}} \Psi_n(v_n)  
\end{equation}
such that 
\begin{equation} \label{eq:reg_graph_compact}
R_\alpha(u) = \inf \left\{ \Psi_u(v) \st v \in \rg(\Lambda_\alpha) \right\}.
\end{equation}
\begin{prop}\label{prop:convexineq}
Every regularization graph functional $R_\alpha : X_{\hat n} \rightarrow [0,+\infty]$ is convex, $R_\alpha(0) = 0$ and $R_\alpha(\lambda u) \leq \lambda R_\alpha (u)$ for all $u \in X_{\hat{n}}$, $\lambda \in (0,1]$. Further, in case each $\Psi_n$ for $n \in V$ is positively one homogeneoous, also $R_\alpha$ is positively one homogeneoous.
\end{prop}
The statement follows easily from the representation in \eqref{eq:reg_graph_compact} together with Assumption \ref{ass:psi_0_is_0}.

\subsection{Examples}\label{sec:examples}

In this section, we provide some concrete examples of regularization graphs to which our general assumptions apply. Here, for $d \in \N$, $d \geq 1$, we always denote by $\Omega \subset \R^d$  a bounded Lipschitz domain. Moreover, we denote by $\emb$ the embedding of a Banach space into another one. We remark that domain and codomain of the embeddings change for different examples. However, they can easily be deduced from the context.

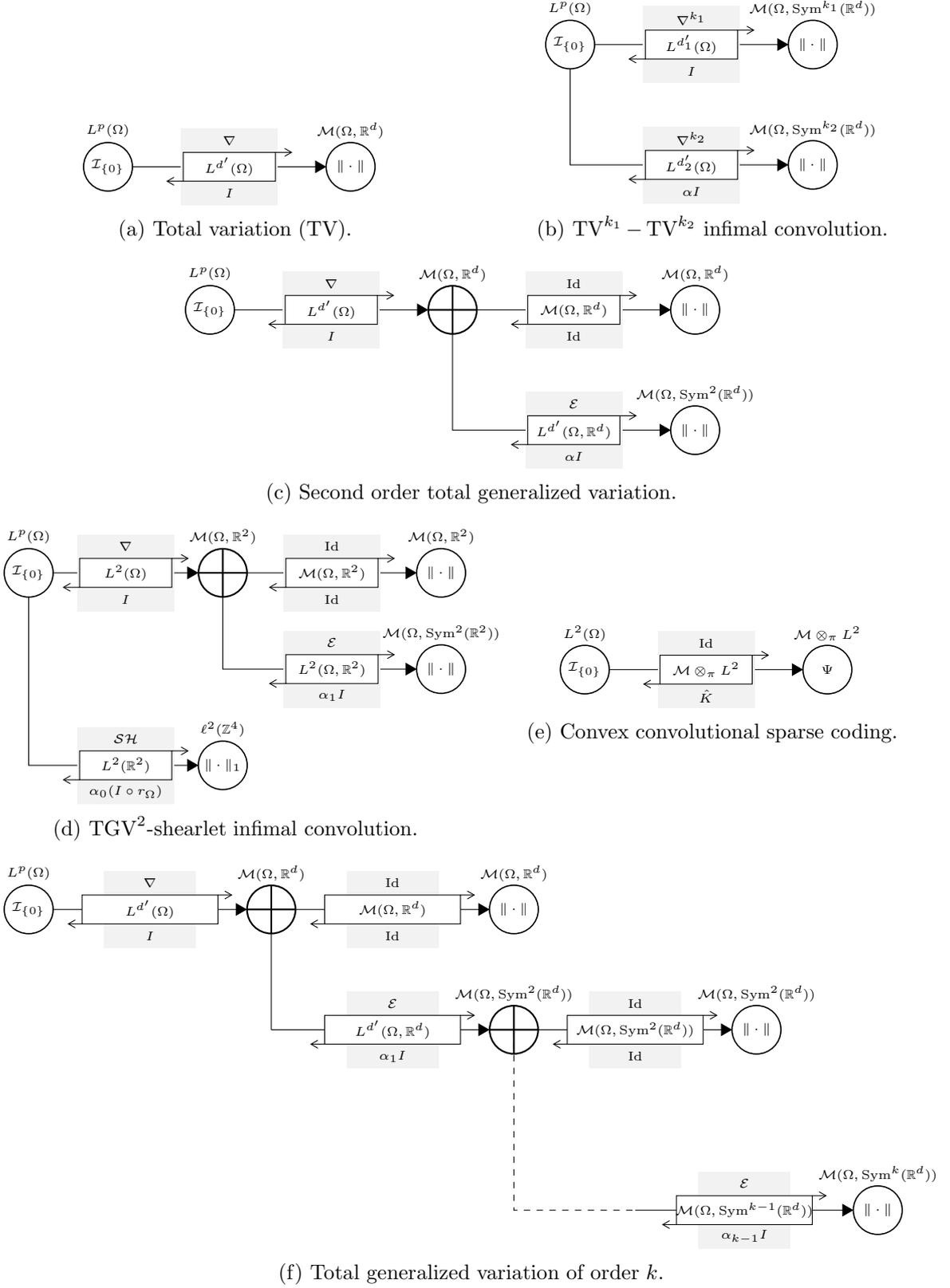
\begin{figure}
\begin{subfigure}[b]{0.5\textwidth}
\centering
\begin{tikzpicture}[rgraph]

\draw (0,0) pic (p1) {node={$L^p(\Omega)$,$\mI_{\{ 0\}}$}};
\draw (1,0) pic (p2) {node={$\mM(\Omega\comma\R^d)$,$\|\cdot \|$}};

\draw[edge] (p1) -- pic{edge={$I$,$L^{d'}(\Omega)$,$\nabla$}} (p2);

\end{tikzpicture}
\subcaption{Total variation (TV).\label{fig:graph_tv}}
\end{subfigure}
\begin{subfigure}[b]{0.5\textwidth}
\centering
\begin{tikzpicture}[rgraph]

\draw (0,0) pic (p1) {node={$L^p(\Omega)$,$\mI_{\{ 0\}}$}};
\draw (1,0) pic (p2) {node={$\mM(\Omega\comma\Sym^{k_1}(\R^d))$,$\|\cdot \|$}};

\draw (1,-1) pic (p3) {node={$\mM(\Omega\comma\Sym^{k_2}(\R^d))$,$\|\cdot \|$}};

\draw[edge] (p1) -- pic{edge={$\emb$,$L^{d_1'}(\Omega)$,$\nabla^{k_1}$}} (p2);
\draw[edge] (p1) -- (0,-1) --  pic{edge={$\alpha I$,$L^{d_2'}(\Omega)$,$\nabla^{k_2}$}} (p3);

\end{tikzpicture}
\subcaption{$\TV^{k_1}-\TV^{k_2}$ infimal convolution.\label{fig:graph_tvk1_tvk2_infcon}}
\end{subfigure}

\medskip

\begin{subfigure}[b]{\textwidth}
\centering

\begin{tikzpicture}[rgraph]

\draw (0,0) pic (p1) {node={$L^p(\Omega)$,$\mI_{\{ 0\}}$}};
\draw (1,0) pic (p2) {node={$\mM(\Omega\comma\R^d)$}};

\draw (2,0) pic (p3) {node={$\mM(\Omega\comma\R^d)$,$\|\cdot \|$}};
\draw (2,-1) pic (p4) {node={$\mM(\Omega\comma\Sym^2(\R^d))$,$\|\cdot \|$}};

\draw[edge] (p1) -- pic{edge={$\emb$,$L^{d'}(\Omega)$,$\nabla$}} (p2);
\draw[edge] (p2) -- pic{edge={$\Id$,$\mM(\Omega\comma\R^d)$,$\Id$}} (p3);
\draw[edge] (p2) -- (1,-1) -- pic{edge={$\alpha I$,$L^{d'}(\Omega\comma\R^d)$,$\symgrad$}} (p4);

\end{tikzpicture}
\subcaption{Second order total generalized variation.\label{fig:graph_tgv2}}

\end{subfigure}

\medskip

\begin{subfigure}[c]{0.5\textwidth}
\centering

\begin{tikzpicture}[rgraph]

\draw (0,0) pic (p1) {node={$L^p(\Omega)$,$ \mI_{\{0\}}$}};
\draw (0.8,0) pic (p2) {node={$\mM(\Omega\comma\R^2)$}};

\draw (1.7,0) pic (p3) {node={$\mM(\Omega\comma\R^2)$,$\|\cdot \|$}};
\draw (1.7,-0.8) pic (p4) {node={$\mM(\Omega\comma\Sym^2(\R^2))$,$\|\cdot \|$}};

\draw[edge] (p1) -- pic{edge={$I$,$L^{2}(\Omega)$,$\nabla$}} (p2);
\draw[edge] (p2) -- pic{edge={$\Id$,$\mM(\Omega\comma\R^2)$,$\Id$}} (p3);
\draw[edge] (p2) -- (0.8,-0.8) -- pic{edge={$\alpha_1 I$,$L^{2}(\Omega\comma\R^2)$,$\symgrad$}} (p4);

\draw (0.8,-1.6) pic (p3) {node={$\ell^{2}(\Z^4)$,$\|\cdot\|_{1}$}};

\draw[edge] (p1) --(0,-1.6) --  pic{edge={$\alpha_0 (I\circ r_\Omega)$,$L^{2}(\R^2)$,$\mathcal{S}\mathcal{H}$}} (p3);
\end{tikzpicture}

\subcaption{$\TGV^2$-shearlet infimal convolution.\label{fig:meyers}}

\end{subfigure}
\begin{subfigure}[c]{0.5\textwidth}

\centering

\begin{tikzpicture}[rgraph]

\draw (0,0) pic (p1) {node={$L^2(\Omega)$,$\mI_{\{ 0\}}$}};
\draw (1,0) pic (p2) {node={$\mathcal{M}\otimes_\pi L^{2}$,$\Psi$}};

\draw[edge] (p1) -- pic{edge={$\hat K$,$\mathcal{M}\otimes_\pi L^{2}$,$\Id$}} (p2);

\end{tikzpicture}
\subcaption{Convex convolutional sparse coding.\label{fig:convolution}}
\end{subfigure}

\medskip

\begin{subfigure}[c]{\textwidth}
\centering
\begin{tikzpicture}[rgraphtgv]

\draw (0,0) pic (p1) {node={$L^p(\Omega)$,$\mI_{\{ 0\}}$}};
\draw (1,0) pic (p2) {node={$\mM(\Omega\comma \R^d)$}};

\draw (2,0) pic (p3) {node={$\mM(\Omega\comma \R^d)$,$\|\cdot \|$}};
\draw (2,-1) pic (p4) {node={$\mM(\Omega\comma\Sym^2(\R^d))$}};

\draw (3,-1) pic (p5) {node={$\mM(\Omega\comma\Sym^2(\R^d))$,$\|\cdot \|$}};

\draw (3.5,-2.5) pic (p6) {node={$\mM(\Omega\comma \Sym^k(\R^d))$,$\|\cdot \|$}};

\draw[edgetgv] (p1) -- pic{edgetgv={$\emb$,$L^{d'}(\Omega)$,$\nabla$}} (p2);
\draw[edgetgv] (p2) -- pic{edgetgv={$\Id$,$\mM(\Omega\comma \R^d)$,$\Id$}} (p3);
\draw[edgetgv] (p4) -- pic{edgetgv={$\Id$,$\mM(\Omega\comma\Sym^2(\R^d))$,$\Id$}} (p5);
\draw (p2) -- (1,-1);
\draw[edgetgv] (0.9,-1) -- pic[xshift=0.5*\nsz cm]{edgetgv={$\alpha_1 \emb$,$L^{d'}(\Omega \comma \R^d)$,$\symgrad$}} (p4);
\draw[dashed] (p4) -- (2,-2.5) -- (2.5,-2.5);
\draw[edgetgv] (2.5 - 0.25*\nsz,-2.5) -- pic{edgetgv={$\alpha_{k-1}\emb$,$\mM(\Omega\comma\Sym^{k-1}(\R^d))$,$\symgrad$}} (p6);

\end{tikzpicture}
\subcaption{Total generalized variation of order $k$.\label{fig:graph_tgvk}}
\end{subfigure}

\caption{Examples of regularization graphs reproducing existing regularization functionals.}\label{fig:ex_reg_graph}
\end{figure}

\medskip

\noindent \textbf{Total variation.} Figure \ref{fig:graph_tv} shows the regularization graph corresponding to the total variation functional. The exponents for the Lebesgue spaces are chosen as $d'=d/(d-1)$ in case $d>1$, $d'=\infty$ else, and $1< p \leq d'$. Thanks to the embedding $\BV(\Omega) \hookrightarrow L^{d'}(\Omega)$, we set $\BV(\Omega)$ as the domain of the linear operator $\nabla : L^{d'}(\Omega) \rightarrow \mathcal{M}(\Omega, \R^d)$. 
With these choices, we now verify assumptions \ref{ass:lsc_functional}--\ref{ass:weak_star_compactness}. Hypotheses \ref{ass:lsc_functional}--\ref{ass:coerc_functional}  follow immediately from the weak* lower semicontinuity and coercivity of the total variation in $\mathcal{M}(\Omega, \R^d)$ \cite{ambrosiobook}. It can easily be verified that the linear operator $\nabla : \BV(\Omega) \subset L^{d'}(\Omega) \rightarrow \mathcal{M}(\Omega, \R^d)$ is weak* closed and its kernel is the set of constant functions implying \ref{ass:closed_fwd_operator} and \ref{ass:ker_fwd_operator}. Assumption \ref{ass:coerc_fwd_operator} 
follows from the Poincaré inequality for $\TV$ \cite{ambrosiobook} and \ref{ass:weak_star_compactness} is a consequence of the Banach-Alaoglu theorem. Finally, the embedding $I : L^{d'}(\Omega) \rightarrow L^p(\Omega)$ is weak*-to-weak* continuous thanks to the exponents choice $1< p \leq d'$, verifying assumption \ref{ass:cont_bkw_operator}.

The functional $R_\alpha : L^p(\Omega) \rightarrow [0,+\infty]$ associated to the regularization graph depicted in Figure \ref{fig:graph_tv} is given as
\[ R_\alpha(u) = \inf_{w \in \BV(\Omega)} \mI_{\{0\}}(u-w) + \|\nabla w\|_{\mM} = \|\nabla u\|_{\mM}
\]
for every $u\in \BV(\Omega)$ and $+\infty$  otherwise. 

\medskip

\noindent \textbf{Infimal convolution of $\TV^{k_1}-\TV^{k_2}$.} Figure \ref{fig:graph_tvk1_tvk2_infcon} shows the regularization graph corresponding to the infimal convolution of $\TV^{k_1}$ and $\TV^{k_2}$ with $k_1,k_2 \in \N$. Here the exponents for the Lebesgue spaces are chosen as $d_i'=d/(d-k_i)$ in case $d_i>k_i$, $d_i'=\infty$ else, and $1<p \leq \min\{d'_1,d'_2\}$. 
Thanks to the embeddings $\BV^{k_i}(\Omega) \hookrightarrow L^{d_i'}(\Omega)$ we set $\BV^{k_i}(\Omega)$ as the domains of the linear operators $\nabla^{k_i} : L^{d_i'}(\Omega) \rightarrow \mathcal{M}(\Omega,\Sym^{k_i}(\R^d))$, where $\Sym^{k}(\R^d)$ denotes the space of symmetric tensors of order $k$, e.g., $\R^d$ for $k=1$ and the space of symmetric $d \times d$ matrices for $k=2$. We refer to \cite{Holler19_ip_review_paper} for details and basic properties of $\BV^k(\Omega)$ and $\TV^k$.

By similar arguments to those used in the previous example and the generalized Poincaré inequality for $\TV^{k_i}$ \cite[Corollary 3.23]{Holler19_ip_review_paper}, it follows that our general assumptions \ref{ass:lsc_functional}--\ref{ass:weak_star_compactness} are satisfied. 
The functional $R_\alpha : L^p(\Omega) \rightarrow [0,+\infty]$ associated to the regularization graph depicted in Figure \ref{fig:graph_tvk1_tvk2_infcon} is given as
\begin{align*}
 R_\alpha(u) &= \inf_{w_i \in \BV^{k_i}(\Omega)} \mI_{\{0\}}(u-w_1-\alpha w_2) + \|\nabla^{k_1} w_1\|_{\mM} + \|\nabla^{k_2} w_2\|_{\mM} \\
 &= \inf_{u=w_1+ \alpha w_2} \TV^{k_1}(w_1) +\TV^ {k_2}(w_2).
\end{align*}

\medskip

\noindent \textbf{Total generalized variation.} Figure \ref{fig:graph_tgv2} shows the regularization graph corresponding to $\TGV^2_\alpha$, the second order $\TGV$ functional as in \cite{tgvbredieskunischpock}.
The exponents for the Lebesgue spaces are chosen as $d'=d/(d-1)$ in case $d>1$, $d'=\infty$ else, and $1< p \leq d'$.
The domain of the linear operator $\nabla : L^{d'}(\Omega) \rightarrow \mathcal{M}(\Omega, \R^d)$ is $\BV(\Omega)$ and the domain of the symmetrized gradient $\symgrad : L^{d'}(\Omega, \R^d) \rightarrow \mathcal{M}(\Omega, \Sym^2(\R^d)  )$ is $\BD(\Omega)$, the space of functions of bounded deformation, where again we take advantage of the embeddings $\BV(\Omega) \hookrightarrow L^{d'}(\Omega)$ and $\BD(\Omega) \hookrightarrow L^{d'}(\Omega,\R^d)$.

By similar arguments to those used the previous examples and the generalized Poincaré inequality for $w \mapsto \|\symgrad w \|_{\mM}$ \cite[Corollary 4.20]{bredies2013tensor} it follows that in this setting our general assumptions \ref{ass:lsc_functional}--\ref{ass:weak_star_compactness} are satisfied.
The  functional $R_\alpha : L^p(\Omega) \rightarrow [0,+\infty]$ associated to the regularization graph functional depicted in Figure \ref{fig:graph_tgv2} is given as
\begin{align*}
 R_\alpha(u) 
 &= \inf_{\substack{w \in \BV(\Omega), \\ w_1 \in \mM(\Omega, \R^d), w_2 \in \BD(\Omega, \R^d)}} 
\mI_{\{0\}}(u-w) + \mI_{\{0\}}(\nabla w -  w_1- \alpha w_2) + \|w_1\|_{\mM} + \|\symgrad w_2\|_{\mM} \\ 
&= \inf_{w \in \BD(\Omega)}  \|\nabla u -\alpha w\|_{\mM} + \|\symgrad w\|_{\mM}
\end{align*}
for $u\in \BV(\Omega)$ and $+\infty$ otherwise. Building on results in \cite{brediesholler2014}, also the TGV functional of arbitrary order $k \in \N$ can be realized via a regularization graph as in Figure \ref{fig:graph_tgvk}.

\medskip

\noindent \textbf{$\TGV^2$-shearlet infimal convolution.}
Figure  \ref{fig:meyers} shows the regularization graph that recovers a $\TGV^2$-shearlet infimal convolution model introduced in \cite{Gao19infimal} (see also \cite{Guo14tgvshearlet}).
Here $\Omega \subset \R^2$ is a bounded Lipschitz domain.
The exponent for the Lebesgue space $L^p(\Omega)$ is chosen as $1< p \leq 2$. The domain of the linear operator $\nabla : L^{2}(\Omega) \rightarrow \mathcal{M}(\Omega, \R^2)$ is $\BV(\Omega)$ and the domain of the symmetrized gradient $\symgrad : L^{2}(\Omega, \R^2) \rightarrow \mathcal{M}(\Omega, \Sym^2(\R^2)  )$ is $\BD(\Omega)$, where again we take advantage of the embeddings $\BV(\Omega) \hookrightarrow L^{2}(\Omega)$ and $\BD(\Omega) \hookrightarrow L^{2}(\Omega, \R^2)$. By similar arguments to those used the previous examples and the generalized Poincaré inequality for $w \mapsto \|\symgrad w \|_{\mM}$ \cite[Corollary 4.20]{bredies2013tensor} it follows that in this setting our general assumptions \ref{ass:lsc_functional}--\ref{ass:weak_star_compactness} are satisfied for edges and nodes realizing the total generalized variation.

In order to introduce the shearlet transform in $L^2(\R^2)$ we start with several notations. First, for $a > 0$ and $s \in \R$ let $A_a$ and $S_s$ be the dilatation matrix and the shearing matrix defined respectively as
\begin{align*}
A_a = \left( \begin{array}{ll}
a & 0 \\
0 & \sqrt{a}
\end{array}
\right),
\qquad 
S_s = \left( \begin{array}{ll}
1 & s \\
0 & 1
\end{array}
\right).
\end{align*}
The discrete shearlet system of $\Psi \in L^2(\R^2)$ is defined as
\begin{align*}
\Psi_{j,k,m}(x) = 2^{\frac{3}{4}j}\Psi (S_k A_{2^j}(x-m))
\end{align*}
for $k,j \in\Z$ and $m \in \Z^2$ \cite[Definition 8]{Kutyniok12_shearlet_book_mh}. This allows to define the discrete shearlet transform operator $\mathcal{SH}$ as
\begin{align}\label{eq:sheardef}
\mathcal{SH}  f(j,k,m) =  \langle f, \Psi_{j,k,m} \rangle_{L^2}  
\end{align}
for $f \in L^2(\R^2)$. 
By standard results in shearlet theory it holds that if $\Psi$ is a classical shearlet, then $\mathcal{SH} : L^2(\R^2) \rightarrow \ell^2(\Z^4)$  is a Parseval frame for $L^2(\R^2)$, that is equivalent to $\|\mathcal{SH} f\|_{\ell^2(\Z^4)} = \|f\|_{L^2(\R^2)}$ for every $f \in L^2(\R^2)$ \cite[Proposition 2]{Kutyniok12_shearlet_book_mh}. In particular, this verifies \ref{ass:coerc_fwd_operator} for $\mathcal{SH}$.  Moreover, a simple computation using that $\|\Psi_{j,k,m}\|_{L^2(\R^2)} = \|\Psi\|_{L^2(\R^2)}$ for every $j,k,m$ together with H\"older's inequality shows  that $\mathcal{SH}$ is weak*-to-weak* continuous, implying \ref{ass:closed_fwd_operator}. The backward operator $I \circ r_\Omega$  is the composition of the embedding $I : L^2(\Omega) \rightarrow L^p(\Omega)$ with the restriction  $r_\Omega : L^2(\R^2) \rightarrow L^2(\Omega)$. It is immediate to check that $I \circ r_\Omega$ is weak*-to-weak* continuous showing \ref{ass:cont_bkw_operator}. Finally, we remark that the functional $\|\cdot\|_1 : \ell^2(\Z^4) \rightarrow [0,+\infty]$ is intended as  the extension to $+\infty$ of the $\ell^1$-norm on $\ell^2$. Such extension is convex, coercive and weak* lower semicontinuous showing \ref{ass:lsc_functional}--\ref{ass:psi_0_is_0}. The  functional $R_\alpha : L^p(\Omega) \rightarrow [0,+\infty]$ associated to the regularization graph functional depicted in Figure \ref{fig:meyers} is then given as
\begin{align*}
R_\alpha(u) 
& = \inf_{\substack{w_1 \in \BV(\Omega), w_2 \in L^2(\R^2), \\ w_3 \in \mM(\Omega, \R^2), w_4 \in \BD(\Omega)}}  \mI_{\{0\}}(u-w_1-  \alpha_0 r_\Omega w_2 )  + \mI_{\{0\}}(\nabla w_1 - w_3 - \alpha_1 w_4 )  \\
& \quad  \qquad  \qquad \quad  \qquad \qquad+ \|w_3\|_\M + \|\symgrad w_4 \|_\M  + \|\mathcal{S}\mathcal{H} w_2\|_{1} \\
 & =  \inf_{\substack{ w_2 \in L^2(\R^2), \\ w_4 \in \BD(\Omega)} } \|\nabla (u-\alpha_0  r_\Omega w_2) - \alpha_1 w_4\|_\M   + \|\symgrad w_4 \|_\M + \|\mathcal{S}\mathcal{H} w_2\|_{1}.
\end{align*}

\medskip

\noindent \textbf{Convex convolutional sparse coding.}
Figure \ref{fig:convolution} shows the regularization graph corresponding to a data-adaptive convolutional-sparse-coding-based method recently introduced in \cite{hollerlearningconvolution2019}.
As such methods are in general non-convex, in \cite{hollerlearningconvolution2019}, the authors proposed a convex relaxation of the convolution LASSO problem in the tensor product of convolutional filter kernels and coefficient images. We refer to \cite{hollerlearningconvolution2019} for a more detailed description of the model.
We denote by $\mathcal{M} \otimes_\pi L^2$ the projective tensor product between $\mathcal{M}(\Omega_\Sigma)$ and $L^2(\Sigma)$ \cite[Appendix A]{hollerlearningconvolution2019}, where $\Sigma$ is a bounded Lipschitz domain and $\Omega_\Sigma := \Omega + \Sigma \subset \R^d$ is the Minkowski sum of $\Omega$ and $\Sigma$. The operator $\hat K : \mathcal{M} \otimes_\pi L^2 \rightarrow L^2(\Omega)$ is the unique tensor lifting of the bilinear operator $K: \mathcal{M}(\Omega_\Sigma) \times L^2(\Sigma) \rightarrow L^2(\Omega)$ defined essentially as
\begin{equation}
K(\mu,\theta)(x) = \int_{\Omega_\Sigma} \theta(x-y)\, d\mu(y).
\end{equation}
Thanks to \cite[Lemma 2]{hollerlearningconvolution2019}, the operator $\hat K$ is weak* to weak* continuous. We also define the convex functional $\Psi :  \mathcal{M} \otimes_\pi L^2 \rightarrow [0,+\infty]$ as $\Psi(C) = \|C\|_\pi + \nu \|C\|_{{\rm nuc}}$ for every $C \in \mathcal{M} \otimes_\pi L^2$, where $\nu>0$ is a parameter,
\begin{equation}
\|C\|_\pi = \inf\left\{\sum_{i=1}^\infty \|\mu_i\|_{\mathcal{M}}\|\theta_i\|_{L^2} \st C = \sum_{i=1}^\infty \mu_i \otimes_\pi \theta_i\right\} \quad 
\end{equation}
is the projective norm of $\mathcal{M} \otimes_\pi L^2$ and 
\begin{equation}
\|C\|_{{\rm nuc}} = \left\{\begin{array}{ll}
 \sum_{i=1}^\infty \sigma_i(T_C)  & \text{if } C \in L^2(\Omega_\Sigma) \otimes_\pi L^2(\Sigma), \\
 +\infty & \text{otherwise,}
\end{array}
\right.
\end{equation}
is an extension of the nuclear norm, where $\sigma_i(T_C)$ are the singular values of $C$ interpreted as a bounded linear map from $L^2(\Omega_\Sigma)$ to $L^2(\Sigma)$.
By Lemma 1 and Lemma 7 in \cite{hollerlearningconvolution2019} it follows that $\Psi$ is weak* lower semicontinuous.  Hence the general assumptions \ref{ass:lsc_functional}--\ref{ass:weak_star_compactness}  for a regularization graph are satisfied and the functional $R_\alpha : L^2(\Omega) \rightarrow [0,+\infty]$ associated to the regularization graph depicted in Figure \ref{fig:convolution} is given as
\begin{align*}
 R_\alpha(u)& = \inf_{C \in\mathcal{M} \otimes_\pi L^2 } \mI_{\{0\}}(u- \hat KC) + \|C\|_\pi + \|C\|_{{\rm nuc}} \,.
\end{align*}

\begin{rem}
The regularization graph functional for $\TV$, infimal convolution of $\TV^{k_1}-\TV^{k_2}$ and total generalized variation can be extended to $L^1(\Omega)$ even if $L^1(\Omega)$ does not admit a predual.  Such extension is described for general regularization graphs in Proposition \ref{prop:reg_graph_extension_to_X}.
\end{rem}

\section{Algebraic properties of regularization graphs} \label{sec:algebraic}
This section provides a recursive representation of regularization graphs and deals with estimates between different regularization graph functionals as well as their combination via addition or infimal convolution. First we need the definition of the height of a graph.
\begin{dfnz}[Height of a regularization graph]\label{def:height}
Given a regularization graph $\Gc_\alpha$ with $G=(V,E)$ the associated directed graph, we denote by $\height(\Gc_\alpha)$ its height defined as the number of edges in the longest path of $G$ connecting the root to one of the leaves. That is, with $n_0 = \hat{n}$ the root node, we define
\[ \height(\Gc_\alpha) = \max \{ M \st \exists\,  n_1,\ldots,n_{M} \text{ with } (n_{i-1},n_{i}) \in E, \text{ for }i=1,\ldots,M \}\]
if this set is non-empty and define $\height(\Gc_\alpha) = 0$ otherwise, i.e., in case of a trivial graph.
\end{dfnz}
Note that the height of a regularization graph does not depend on the particular choice of weights $\alpha$.
Next, we provide a recursion result that allows us to rewrite a regularization graph of height $h$ in terms of regularization graphs of height $h-1$.

\begin{figure}[t]
\centering
\begin{tikzpicture}[rgraph]

\draw (0,0) pic (p0) {node={$X_{\hat{n}}$,$\Psi_{\hat{n}}$}};

\draw[hanblue] (1,1.5) pic (p1) {node={$X_{\hat{n}^{\hat{e}^1}}$,$\Psi_{\hat{n}^{\hat{e}^1}}$}};

\draw[gray,dashed] (1,0.45) pic (pl1) {node={$X_{\hat{n}^{\hat{e}^l}}$,$\Psi_{\hat{n}^{\hat{e}^l}}$}};
\draw[gray,dashed] (1,-0.45) pic (pl2) {node={$X_{\hat{n}^{\hat{e}^{l+1}}}$,$\Psi_{\hat{n}^{\hat{e}^{l+1}}}$}};

\draw[darkred] (1,-1.5) pic (pL) {node={$X_{\hat{n}^{\hat{e}^{L}}}$,$\Psi_{\hat{n}^{\hat{e}^{L}}}$}};

\draw[edge] (p0) -- pic{edge={$\Phi_{\hat{e}^1}$,$X_{\hat{e}^1}$,$\Theta_{\hat{e}^1}$}} (p1);

\draw[edge,gray,dashed] (p0) --  (pl1);
\draw[edge,gray,dashed] (p0) --  (pl2);

\draw[edge] (p0) -- pic{edge={$\Phi_{\hat{e}^L}$,$X_{\hat{e}^L}$,$\Theta_{\hat{e}^L}$}} (pL);

\foreach \i in {-0.5,0,0.5}
{
	\draw[edge,dashed,hanblue] (p1) -- (1.5,1.5+\i);
	
	\draw[edge,dashed,gray] (pl1) -- (1.5,0.45+\i);
	\draw[edge,dashed,gray] (pl2) -- (1.5,-0.45+\i);

	\draw[edge,dashed,darkred] (pL) -- (1.5,-1.5+\i);
}

\draw[hanblue,thick] (0.75,1) rectangle  (1.5,2);
\node[hanblue] at (1.6,1.5) {\tiny $R_{\alpha^{\hat{e}^1}}^{\hat{e}^1}$};

\draw[darkred,thick] (0.75,-1) rectangle  (1.5,-2);
\node[darkred] at (1.6,-1.5) {\tiny $R_{\alpha^{\hat{e}^L}}^{\hat{e}^L}$};

\end{tikzpicture}

\caption{\label{fig:graph_recursive} The recursive representation of a regularization graph according to Lemma \ref{lem:primal_graph_recursion}.}
\end{figure}
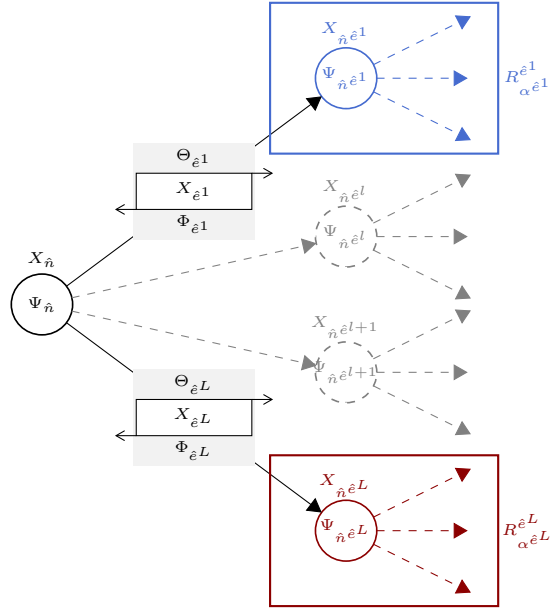

 \begin{lemma}[Recursive representation of regularization graphs] \label{lem:primal_graph_recursion} For $\Gc_\alpha $ a regularization graph of height $h\geq 1$, $G = (V,E)$ the associated directed graph and $\hat{n}$ the root node, let 
$\hat{E}\subset E$ be the set all edges connected to the root node $\hat{n}$, $\hat{n}^{\hat{e}}$ for $\hat{e} \in \hat{E}$ be their endpoints
and let $G^{\hat{e}} = (V^{\hat{e}},E^{\hat{e}})$ be the subtree of $G=(V,E)$ with $\hat{n}^{\hat{e}}$ as root node.
 
Then, there exist regularization graphs $\Gc^{\hat{e}}_{\alpha^{\hat{e}}}$ with associated directed graphs $G^{\hat{e}} = (V^{\hat{e}},E^{\hat{e}})$ of height at most $h-1$ and weights $(\alpha^{\hat{e}}_e)_{e \in E^{\hat{e}}}$ such that, with $R_{\alpha} = R(\Gc_{\alpha})$ and $R^{\hat{e}}_{\alpha^{\hat{e}}} = R(\Gc^{\hat{e}}_{\alpha^{\hat{e}}})$ the associated functionals, the following recursive representation holds
\begin{equation}\label{eq:recursiveformula}
R_\alpha(u) = \inf\Big
  \{ \Psi_{\hat{n}} \big(u-  \sum_{{\hat{e}} \in \hat{E}} \alpha_{{\hat{e}}} \Phi_{{\hat{e}}} w_{{\hat{e}}} \big) + \sum_{{\hat{e}} \in \hat{E}}  R_{\alpha^{\hat{e}}}^{\hat{e}}(\Theta_{{\hat{e}}} w_{{\hat{e}}}) \Bigst w_{{\hat{e}}} \in \domain(\Theta_{{\hat{e}}}) \text{ for all } {\hat{e}} \in \hat{E}
 \Big\}.
\end{equation}
 \begin{proof}

We explicitly construct the claimed recursive representation as visualized in Figure \ref{fig:graph_recursive}. 
First note that we can re-write $R_\alpha$ as
 \begin{multline*}
  R_\alpha(u) = \inf \bigg\{
  \Psi_{\hat{n}} \Big (u-\sum_{{\hat{e}} \in \hat{E}} \alpha_{{\hat{e}}}  \Phi_{{\hat{e}}}  w_{{\hat{e}}} \Big) +  \sum_{{\hat{e}} \in \hat{E}} \sum_{n \in V^{\hat{e}}} \Psi_{n}(v_{n})    \Bigst w_{e} \in \domain(\Theta_{e})  \text{ for all } {e} \in E,
 \\ 
\forall {\hat{e}} \in \hat{E},  n \in V^{\hat{e}}: v_{n} =    \Theta_{(n^-,n)}w_{(n^-,n)} - \sum_{(n,m) \in E^{\hat{e}}} \alpha_{(n,m)}  \Phi_{(n,m)} w_{(n,m)} \bigg\}.
 \end{multline*}
 Now define $\Gc^{\hat{e}}_{\alpha^{\hat{e}}}$ to be a regularization graph with graph structure $G^{\hat{e}}= (V^{\hat{e}},E^{\hat{e}})$ and the associated operators, functionals and weights, such that $R^{\hat{e}}_{\alpha^{\hat{e}}} = R(\Gc^{\hat{e}}_{\alpha^{\hat{e}}})$ is given as
\begin{multline*}
 R^{\hat{e}}_{\alpha^{\hat{e}}}(z) = \inf \Bigg\{   \sum_{n \in V^{\hat{e}}} \Psi_{n}(v_{n}) \Bigst w_{e}\in \domain (\Theta_e) \text{ for all }e\in E^{\hat{e}},   \\
\forall n \in V^{\hat{e}}:  v_{n} = \Theta_{(n^-,n)}w_{(n^-,n)} - \sum_{(n,m) \in E^{\hat{e}}} \alpha_{(n,m)}  \Phi_{(n,m)} w_{(n,m)} 
 \Bigg\} \,,
 \end{multline*}
 where we note that here $\hat{n}^{\hat{e}}$ is regarded as a node of $\Gc^{\hat{e}}_{\alpha^{\hat{e}}}$, and thus $\Theta_{((\hat{n}^{\hat{e}})^-,\hat{n}^{\hat{e}})} = \Id$ and $w_{((\hat{n}^{\hat{e}})^-,\hat{n}^{\hat{e}})} = z$. The recursive representation of $R_\alpha$ is then given as
 \begin{multline*}
  R_\alpha(u)  = \inf\bigg\{
  \Psi_{\hat{n}}\Big(u - \sum_{\hat{e} \in \hat{E}} \alpha_{{\hat{e}}}  \Phi_{{\hat{e}}}  w_{{\hat{e}}} \Big)  
  +  \sum_{{\hat{e}} \in {\hat{E}}} \sum_{n \in V^{\hat{e}} } \Psi_{n}(v_{n})  
  \Bigst w_{e} \in \domain(\Theta_{e})  \text{ for all } e \in E,\\ 
 \forall {\hat{e}} \in \hat{E},  \forall n \in V^{\hat{e}}\setminus \{ \hat{n}^{\hat{e}}\}  :  v_{n} =  \Theta_{(n^-,n)}w_{(n^-,n)} - \sum_{(n,m) \in E^{\hat{e}}} \alpha_{(n,m)}\Phi_{(n,m)} w_{(n,m)}, \
     \\
      v_{\hat{n}^{\hat{e}}} =  \Theta_{{\hat{e}}}w_{{\hat{e}}} - \sum_{(\hat{n}^{\hat{e}},m) \in E^{\hat{e}}} \alpha_{(\hat{n}^{\hat{e}},m)}\Phi_{(\hat{n}^{\hat{e}},m)} w_{(\hat{n}^{\hat{e}},m)}  \bigg\}\\
 = \inf\Big
  \{ \Psi_{\hat{n}} \big(u-  \sum_{{\hat{e}} \in \hat{E}} \alpha_{{\hat{e}}} \Phi_{{\hat{e}}} w_{{\hat{e}}} \big) + \sum_{{\hat{e}} \in \hat{E}}  R_{\alpha^{\hat{e}}}^{\hat{e}}(\Theta_{{\hat{e}}} w_{{\hat{e}}}) \Bigst w_{{\hat{e}}} \in \domain(\Theta_{\hat{e}}) \text{ for all } {\hat{e}} \in \hat{E}
 \Big\},
 \end{multline*}
 which proves the assertion.
 \end{proof}
 \end{lemma}
As first consequence of this recursive representation, we obtain an estimate between two regularization graph functionals corresponding to regularization graphs with different weights.
\begin{lemma} \label{lem:weight_equivalence} Let $\Gc_{\alpha^1}$ and $\Gc_{\alpha^2}$ be two regularization graphs with the same underlying graph structure $\Gc$ and directed graph $G = (V,E)$ with root node $\hat{n}$, and let $\alpha^1$, $\alpha^2$ be weights such that $\alpha^1_e \geq \alpha^2_e $ for all $e \in E$.  Then, with 
\begin{equation} \label{eq:constant_reg_graph_equivalent}
C_{\alpha^1,\alpha^2}:= \max \big \{ \prod_{e \in F} \frac{\alpha_e^2}{\alpha_e^1} \st F \subset E \text{ is either empty or a chain with $\hat{n}$ as root} \big \},
\end{equation} 
where we use the conventions $ \prod_{e \in \emptyset} \frac{\alpha_e^2}{\alpha_e^1} = 1$ and $\frac{0}{0}=0$, for the associated regularization graph functionals $R_{\alpha^1} = R(\Gc_{\alpha^1})$ and $R_{\alpha^2} = R(\Gc_{\alpha^2})$ it holds that
\[R_{\alpha^1}(u)  \leq C_{\alpha^1,\alpha^2} R_{\alpha^2} (u) \quad \text{for all }u \in X_{\hat{n}}.
\]
\begin{proof}
We prove the result by induction over the height $h$ of the graphs. Assume the result holds true for any two regularization graphs with height less than $h$.

Now note that, by assumption, $(\alpha^1_e =0)$ implies $(\alpha^2_e=0)$, so we can adapt the graph $G=(V,E)$ by removing all edges $e \in E$ with $\alpha^1_e = 0$ and all subsequently disconnected nodes, without increasing its height, changing $C_{\alpha^1,\alpha^2}$ or the values of the $R_{\alpha^i}$. Hence, without loss of generality, assume that $\alpha_e^1>0$ for all $e \in E$. 
Now for $h=0$ the result holds trivially and for $h \geq 1 $ we can use the recursive representation of Lemma \ref{lem:primal_graph_recursion} to obtain
\begin{align*}
 R_{\alpha^1}(u)  
 & = \inf\left\{ \Psi_{\hat{n}} \left(u-  \sum_{\hat{e} \in \hat{E}}  \Phi_{\hat{e}} w_{\hat{e}} \right) + \sum_{\hat{e} \in \hat{E}}   R_{(\alpha^1)^{\hat{e}}}^{\hat{e}}\left(\frac{\Theta_{\hat{e}} w_{\hat{e}}}{\alpha_{\hat{e}}^1}\right) \Biggst w_{\hat{e}} \in \domain(\Theta_{\hat{e}})  
 \text{ for all } \hat{e} \in \hat{E} 
 \right\} \\
&  \leq \inf \left\{ \Psi_{\hat{n}} \left(u-  \sum_{{\hat{e}}\in \hat{E}} \alpha_{\hat{e}}^2 \Phi_{\hat{e}} w_{\hat{e}} \right) + \sum_{\hat{e} \in \hat{E} }  \frac{\alpha_{\hat{e}}^2}{\alpha_{\hat{e}}^1} R_{(\alpha^1)^{\hat{e}}}^{\hat{e}}(\Theta_{\hat{e}} w_{\hat{e}}) \Biggst w_{\hat{e}} \in \domain(\Theta_{\hat{e}}) 
\text{ for all } \hat{e} \in \hat{E}
 \right\} \\
& \leq C_{\alpha^1,\alpha^2} R_{\alpha^2} (u),
\end{align*}
where in the first line we substituted $w_{\hat{e}}$ to $\alpha_{\hat{e}}^1 w_{\hat{e}}$ and in the second line  we substituted $\alpha_{\hat{e}}^2 w_{\hat{e}}$ to $w_{\hat{e}}$; additionally, in the first inequality 
we used that $R_{(\alpha^1)^{\hat{e}}}^{\hat{e}} (\lambda u) \leq \lambda R_{(\alpha^1)^{\hat{e}}}^{\hat{e}} (u)$ for $\lambda \in [0,1]$, see Proposition \ref{prop:convexineq}, and we obtained the last estimate from the induction hypothesis and the definition of $C_{\alpha^1,\alpha^2}$.
\end{proof}
\end{lemma}
\begin{rem} It is easy to see from the proof above that, whenever $\Psi_{n}$ for some $n \in V$ is positive one-homogeneous, the assumption $\alpha^1_e \geq \alpha^2_e $ for $e = (n^-,n)$ can be replaced by $(\alpha^1_e =0)$ implying $(\alpha^2_e=0) $. In particular, if $(\alpha^1_e =0)$ if and only if $(\alpha^2_e=0) $ for all $e= (n^-,n)$ such that $\Psi_{n}$ is positive one-homogeneous and all other $\alpha^1_e$ and $\alpha^2_e$ coincide, then $R_{\alpha^1}$ and $R_{\alpha^2}$ are equivalent, i.e., $R_{\alpha^1}$ can be estimated from above and below by a constant times $R_{\alpha^2}$, $\domain(R_{\alpha^1} ) = \domain(R_{\alpha^2} )  $ and also their zero-sets coincide.
\end{rem}

\subsection{Combining regularization graphs}
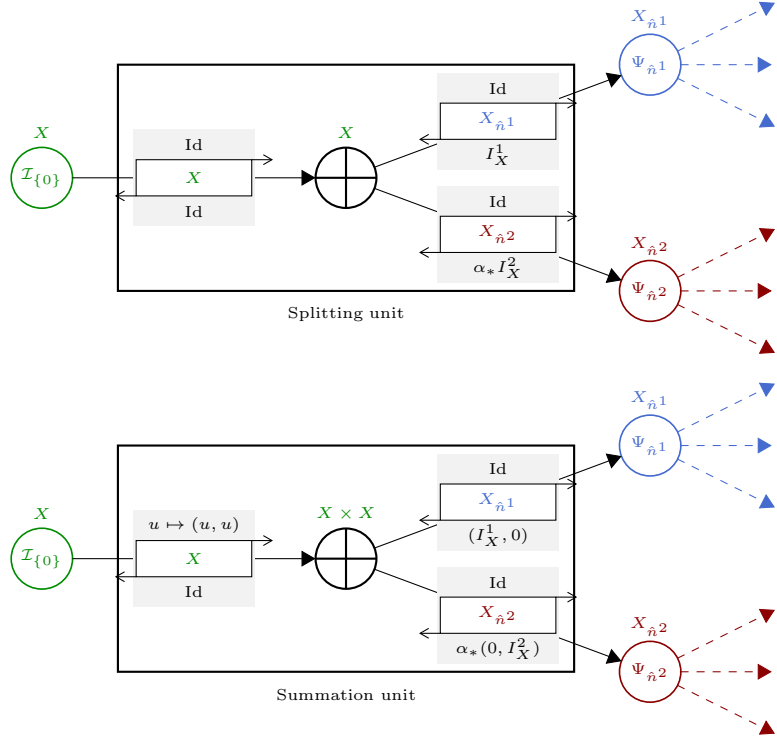
\begin{figure}[t]
\centering
\begin{tikzpicture}[rgraph]

\draw[thick] (0.25,0.75) rectangle  (1.75,-0.75);
\node at (1,-0.9) {\tiny Splitting unit};

\draw[darkgreen] (0,0) pic (p0) {node={$X$,$\Ic_{\{0\}}$}};
\draw (1,0) pic (p1) {node={$\color{darkgreen}{X}$}};

\draw[hanblue] (2,0.75) pic (p2) {node={$X_{\hat{n}^1}$,$\Psi_{\hat{n}^1}$}};
\draw[darkred] (2,-0.75) pic (p3) {node={$X_{\hat{n}^2}$,$\Psi_{\hat{n}^2}$}};

\draw[edge] (p0) -- pic{edge={$\Id$,$\color{darkgreen}{X}$,$\Id$}} (p1);
\draw[edge] (p1) -- pic{edge={$\emb^1_X$,$\color{hanblue}{X_{\hat{n}^1}}$,$\Id$}} (p2);
\draw[edge] (p1) -- pic{edge={$\alpha_*\emb^2_X$,$\color{darkred}{X_{\hat{n}^2}}$,$\Id$}} (p3);

\foreach \i in {-0.5,0,0.5}
{
	\draw[edge,dashed,hanblue] (p2) -- (2.5,0.75+\i);

	\draw[edge,dashed,darkred] (p3) -- (2.5,-0.75+\i);
}

\end{tikzpicture}
\begin{tikzpicture}[rgraph]

\draw[thick] (0.25,0.75) rectangle  (1.75,-0.75);
\node at (1,-0.9) {\tiny Summation unit};

\draw[darkgreen] (0,0) pic (p0) {node={$X$,$\Ic_{\{0\}}$}};
\draw (1,0) pic (p1) {node={$\color{darkgreen}{X \times X}$}};

\draw[hanblue] (2,0.75) pic (p2) {node={$X_{\hat{n}^1}$,$\Psi_{\hat{n}^1}$}};
\draw[darkred] (2,-0.75) pic (p3) {node={$X_{\hat{n}^2}$,$\Psi_{\hat{n}^2}$}};

\draw[edge] (p0) -- pic{edge={$\Id$,$\color{darkgreen}{X}$,$u \mapsto (u\comma u)$}} (p1);
\draw[edge] (p1) -- pic{edge={$(\emb^1_X \comma 0)$,$\color{hanblue}{X_{\hat{n}^1}}$,$\Id$}} (p2);
\draw[edge] (p1) -- pic{edge={$\alpha_*(0\comma \emb^2_X )$,$\color{darkred}{X_{\hat{n}^2}}$,$\Id$}} (p3);

\foreach \i in {-0.5,0,0.5}
{
	\draw[edge,dashed,hanblue] (p2) -- (2.5,0.75+\i);

	\draw[edge,dashed,darkred] (p3) -- (2.5,-0.75+\i);
}

\end{tikzpicture}
\caption{\label{fig:summation_splitting_units}
Combining regularization graphs via infimal convolution (top) and summation (bottom). }
\end{figure}

Obviously, for $R_{\alpha} = R(\Gc_{ \alpha})$ being a regularization graph funtional and $\lambda>0$, also $\lambda R_\alpha $ is a regularization graph functional (corresponding to an adaption of the regularization graph $\Gc_\alpha$ where all node functionals $\Psi_n$ are replaced by $\lambda \Psi_n$). In this subsection we show that also the sum and infimal-convolution of two regularization graph functionals are again regularization graph functionals. 
\begin{prop}[Infimal-convolution and sum of regularization graphs] \label{prop:infcon_sum_reggraphs}
Let $\Gc^1_{\alpha^1}= (\Gc^1,\alpha^1)$ and $\Gc^2_{\alpha^2}= (\Gc^2,\alpha^2)$ be two regularization graphs with associated directed graphs $G^1 = (V^1,E^1)$ and $G^2 = (V^2,E^2)$ and root nodes $\hat{n}^1$ and $\hat{n}^2$, respectively. Let $X$ be a Banach space admitting a predual such that bounded sequences in $X$ admit weak* convergent subsequences and such that the embeddings $\emb^1_{X} :X_{\hat{n}^1} \rightarrow X$ and $\emb^2_{X} :X_{\hat{n}^2} \rightarrow X$ are continuous w.r.t. weak* convergence. 

For additional nodes $\hat{n}, \hat{n}^0 \notin V^1 \cup V^2$ define the combined graph $G = (V,E)$ with
\[ V = \{\hat{n},\hat{n}^0 \} \cup V^1 \cup V^2,  \quad E = \{ (\hat{n},\hat{n}^0),(\hat{n}^0,\hat{n}^1),(\hat{n}^0,\hat{n}^2) \} \cup E^1 \cup E^2.
\]
Then, combined regularization graphs can be constructed as follows.
\begin{itemize}
\item \textbf{Infimal-convolution: } For the additional nodes and edges in $G$, define the spaces
\[ X_{\hat{n}}=X_{\hat{n}^0} = X, \quad  X_{(\hat{n},\hat{n}^0)} = X, \quad X_{(\hat{n}^0,\hat{n}^l)} =X_{\hat{n}^l}, \quad l=1,2,\]
the operators 
\[\Theta_{(\hat{n},\hat{n}^0)} = \Theta_{(\hat{n}^0,\hat{n}^1)} = \Theta_{(\hat{n}^0,\hat{n}^2)} = \Id, \, \quad \Phi_{(\hat{n},\hat{n}^0)} = \Id, \,  \quad \Phi_{(\hat{n}^0,\hat{n}^1)} = \emb^1_{X},\quad  \Phi_{(\hat{n}^0,\hat{n}^2)} = \emb^2_X,\]
and the functionals and weights
\[ \Psi_{\hat{n}}= \Psi_{\hat{n}^0}= \Ic_{\{ 0 \}} , \quad \alpha_{(\hat{n},\hat{n}^0)} = \alpha_{(\hat{n}^0,\hat{n}^1)} = 1, \quad \alpha_{(\hat{n}^0,\hat{n}^2)}  = \alpha_* \in [0,\infty)  ,\]
and adopt the elements of $\Gc^1_{\alpha^1}$ and $\Gc^2_{\alpha^2}$ for all other nodes and edges.
Then, the associated structure $\Gc_\alpha = (\Gc,\alpha)$ defines a regularization graph and, for $R_\alpha = R(\Gc_\alpha)$, $R^1_{\alpha^1} = R(\Gc^1_{\alpha^1})$ and $R^2_{\alpha^2} = R(\Gc^2_{\alpha^2})$, it holds that
\[
R_\alpha(u) = \inf_{v \in X_{\hat n}} R^1_{\alpha^1}(u-\alpha_* v) +  R^2_{\alpha^2}(v).
\]
\item \textbf{Summation: } For the additional nodes and edges in $G$, define the spaces
\[ X_{\hat{n}}=X , \,  \quad X_{\hat{n}^0} = X \times X, \, \quad X_{(\hat{n},\hat{n}^0)}=X,\, \quad X_{(\hat{n}^0,\hat{n}^l)} =X _{\hat{n}^l}, \quad  l=1,2,\]
the operators 
\begin{align*}
& \Theta_{(\hat{n},\hat{n}^0)}=[u \mapsto (u,u)],\quad \Theta_{(\hat{n}^0,\hat{n}^1)} =  \Theta_{(\hat{n}^0,\hat{n}^2)} = \Id, \\
&  \Phi_{(\hat{n},\hat{n}^0)}   = \Id, \quad  \Phi_{(\hat{n}^0,\hat{n}^1)} = [u \mapsto (\emb^1_{X} u,0)], \quad \Phi_{(\hat{n}^0,\hat{n}^2)} = [u \mapsto (0,\emb^2_{X} u)],
\end{align*}
and the functionals and weights
\[ \Psi_{\hat{n}}= \Psi_{\hat{n}^0}= \Ic_{\{ 0 \}} , \quad \alpha_{(\hat{n},\hat{n}^0)} = \alpha_{(\hat{n}^0,\hat{n}^1)} = 1, \quad \alpha_{(\hat{n}^0,\hat{n}^2)}  = \alpha_* \in (0,\infty)  ,\]
and adopt the elements of $\Gc^1_{\alpha^1}$ and $\Gc^2_{\alpha^2}$ for all other nodes and edges.
Then, the associated structure $\Gc_\alpha = (\Gc,\alpha)$ defines a regularization graph and, for $R_\alpha = R(\Gc_\alpha)$, $R^1_{\alpha^1} = R(\Gc^1_{\alpha^1})$ and $R^2_{\alpha^2} = R(\Gc^2_{\alpha^2})$, it holds that
\[
R_\alpha(u) = R^1_{\alpha^1}(u) +  R^2_{\alpha^2}(\alpha_*^{-1} u).
\]
\end{itemize}

\begin{proof}It is easy to see that all spaces, functionals, operators and weights involved in the definition of $\Gc_\alpha$ fulfill Assumptions \ref{ass:lsc_functional} to \ref{ass:weak_star_compactness},  such that $\Gc_\alpha$ defines a regularization graph. 
Denote the edges $e^l= (\hat{n},\hat{n}^l)$ for $l\in \{ 0,1,2\} $.
For the claimed representation of $R_\alpha$ in case of the infimal-convolution, we observe that
\begin{multline*}
  R_\alpha(u) = \inf \bigg\{
  \Ic_{\{0\}} (u- w_{e^0}) + \Ic_{\{0\}} (w_{e^0} - w_{e^1} - \alpha_* w_{e^2}) +  \sum_{l=1}^2 \sum_{ n^l \in V^l}  \Psi_{n^l}(v_{n^l})    \Bigst
 \\ 
\forall l \in \{1,2\}, n^l \in V^l\setminus\{\hat{n}^l\}:  v_{n^l} =    \Theta_{((n^l)^-,n^l)}w_{((n^l)^-,n^l)} - \sum_{(n^l,m) \in E^l} \alpha_{(n^l,m)}  \Phi_{(n^l,m)} w_{(n^l,m)},  \\
  v_{\hat{n}^l} =    w_{e^l} - \sum_{(\hat{n}^l,m) \in E^l} \alpha_{(\hat{n}^l,m)}  \Phi_{(\hat{n}^l,m)} w_{(\hat{n}^l,m)} , \forall e \in E : w_e \in \domain(\Theta_e)
\bigg\}  \\
= \inf \bigg\{
 \sum_{l=1}^2 \sum_{ n^l \in V^l}  \Psi_{n^l}(v_{n^l})    \Bigst   w_{e^1} + \alpha_* w_{e^2}= u , \forall l \in \{1,2\}, n^l \in V^l\setminus\{\hat{n}^l\}:
 \\
  v_{n^l} =    \Theta_{((n^l)^-,n^l)}w_{((n^l)^-,n^l)} - \sum_{(n^l,m) \in E^l} \alpha_{(n^l,m)}  \Phi_{(n^l,m)} w_{(n^l,m)},  \\
  v_{\hat{n}^l} =    w_{e^l} - \sum_{(\hat{n}^l,m) \in E^l} \alpha_{(\hat{n}^l,m)}  \Phi_{(\hat{n}^l,m)} w_{(\hat{n}^l,m)} , \forall e \in E : w_e \in \domain(\Theta_e)
\bigg\} \\
= \inf_{v \in X_{\hat n}} R^1_{\alpha^1}(u-\alpha_* v) +  R^2_{\alpha^2}(v).
 \end{multline*}

Likewise, for the claimed representation of $R_\alpha$ in case of summation, we observe that
\begin{multline*}
  R_\alpha(u) = \inf \bigg\{
  \Ic_{\{0\}} (u- w_{e^0}) + \Ic_{\{(0,0)\}} ((w_{e^0}- w_{e^1},w_{e^0}- \alpha_* w_{e^2})) +  \sum_{l=1}^2 \sum_{ n^l \in V^l}  \Psi_{n^l}(v_{n^l})    \Bigst  
 \\
\forall l \in \{1,2\}, n^l \in V^l\setminus\{\hat{n}^l\}:  v_{n^l} =    \Theta_{((n^l)^-,n^l)}w_{((n^l)^-,n^l)} - \sum_{(n^l,m) \in E^l} \alpha_{(n^l,m)}  \Phi_{(n^l,m)} w_{(n^l,m)},  \\
  v_{\hat{n}^l} =    w_{e^l} - \sum_{(\hat{n}^l,m) \in E^l} \alpha_{(\hat{n}^l,m)}  \Phi_{(\hat{n}^l,m)} w_{(\hat{n}^l,m)} , \forall e \in E : w_e \in \domain(\Theta_e)
\bigg\}  \\
= \inf \bigg\{
 \sum_{l=1}^2 \sum_{ n^l \in V^l}  \Psi_{n^l}(v_{n^l})    \Bigst   w_{e^1} = u,  w_{e^2} =  \alpha_*^{-1}u , \forall l \in \{1,2\}, n^l \in V^l\setminus\{\hat{n}^l\}:
 \\
  v_{n^l} =    \Theta_{((n^l)^-,n^l)}w_{((n^l)^-,n^l)} - \sum_{(n^l,m) \in E^l} \alpha_{(n^l,m)}  \Phi_{(n^l,m)} w_{(n^l,m)},  \\
  v_{\hat{n}^l} =    w_{e^l} - \sum_{(\hat{n}^l,m) \in E^l} \alpha_{(\hat{n}^l,m)}  \Phi_{(\hat{n}^l,m)} w_{(\hat{n}^l,m)} , \forall e \in E : w_e \in \domain(\Theta_e)
\bigg\} \\
= R^1_{\alpha^1}(u) +  R^2_{\alpha^2}(\alpha_*^{-1}u).
 \end{multline*} 
\end{proof}
\end{prop}
More generally, note that any regularization graph can be extended by appending another regularization graph to one of its leaves, and in particular by appending a regularization graph corresponding to the infimal convolution or the sum of two other regularization graph functionals. The latter can be achieved by appending a splitting or summation unit as in Figure \ref{fig:summation_splitting_units} to a leaf-node, where the $\Ic_{\{0\}}$ and $X$ in the left, green nodes in Figure \ref{fig:summation_splitting_units} are replaced by the corresponding node functional and node space of the leaf node.

\begin{rem}[Assumptions on the sum of two regularization graphs]
The notion of regularization graphs was designed mainly for an infimal-convolution-type combination of functionals and operators, since we believe this situation is more interesting in practice. For infimal-convolution-type combinations, we believe that our assumptions on the underlying functionals and operators are rather minimal. Our framework also allows for the summation of two functionals, but in this situation our assumptions are suboptimal. Indeed, one would expect that in a summation, only one of the two functionals needs to fulfill the assumptions of a regularization graph in order to provide a suitable regularization strategy.
Indeed, when using the sum of two (suitable) functionals for regularization, generically, only one of them needs to fulfill coercivity properties (such as \ref{ass:coerc_functional} together with \ref{ass:coerc_fwd_operator}) in order to obtain well-posedness results for linear inverse problems.
Nevertheless, we do not further generalize our framework towards weakening the assumptions for the sum of two functionals since i) we believe this situation is less relevant and ii) this would significantly complicate our basic assumptions and results for instance on convergence for vanishing noise and bilevel optimization, thereby hindering our main goal of providing an easily applicable framework.
\end{rem}

\section{Analytic properties of regularization graphs}\label{sec:analytic_properties}
The goal of this section is to obtain analytic properties of regularization graph functionals that provide the basis for well-posedness results for the regularization of inverse problems.  To this aim, we first consider lower semi-continuity and coercivity properties, for which we need a general lemma that deals with projections.

For the lemma, remember that for a Banach space $X$ and a finite dimensional subspace $L$ of $X$, there always exists a bounded linear projection $P:X \rightarrow L$.
\begin{lemma}\label{lem:projection_transfer} Let $X$ be a Banach space, $R:X \rightarrow [0,\infty]$ be a functional, $L \subset X$ be a finite dimensional subspace and assume there is a function $G:X \rightarrow L$ and $C >0,D\geq 0$ such that
\[ \| u - G(u)\|_X   \leq C R(u) +D\]
for all $u \in X$. Then, for a closed subspace $K \subset X$ and a bounded, linear projection $P_{K\cap L} :X \rightarrow K \cap L$ there exist constants $\tilde{C}>0 ,\tilde{D} \geq 0$, with $\tilde{D}=0$ in case $D=0$, such that for all $u \in K$,
\begin{equation}\label{eq:eest}
\| u - P_{K\cap L}u \|_X \leq \tilde{C} R(u) + \tilde{D}.
\end{equation}
In particular, if $K = X$, this holds for any bounded linear projection onto $L$. 
\begin{proof}
Assume this does not hold true, then we can pick a sequence $(u^k)_k$ in $K$ such that
\[ \| u^k - P_{K\cap L} u^k \|_X  > k R(u^k) + kD/C \]
for each $k$, with $C,D$ being the constants of the original estimate. This implies in particular that $\| u^k - P_{K\cap L} u^k \|_X > 0$ for each $k$. 
Defining $\tilde{u}^k = (u^k - P_{K\cap L} u^k) / \|u^k - P_{K\cap L} u^k\|_X$ such that $\|\tilde{u}^k \|_X = 1$, we can estimate 
\[   C/k  \geq C \frac{R(u^k) + D/C}{\|u^k - P_{K\cap L} u^k\|_X }\geq \frac{\|u^k - G(u^k)  \|_X}{\|u^k - P_{K\cap L} u^k\|_X} = \|\tilde{u}^k - \frac{G(u^k)- P_{K\cap L}u^k}{\|u^k - P_{K\cap L} u^k\|_X} \|_X
 .\]
In particular, this implies that $( (G(u^k)- P_{K\cap L}u^k)/\|u^k - P_{K\cap L} u^k\|_X)_k$ is bounded and, by finite dimensionality, admits a (non-relabeled) subsequence strongly converging to some $z \in L$. Consequently, also $(\tilde{u}^k)_k$ converges strongly to $z$ and from closedness of $K$ we get that $z \in K$. Hence
\[ 0 = (\Id- P_{K\cap L} )z = \lim_k(\Id- P_{K\cap L} )\tilde{u}^k = \lim_k \tilde{u}^k \]
strongly, which is a contradiction to $\|\tilde{u}^k\|_X = 1$ for each $k$ and concludes the proof of the coercivity estimate in the general form. Also, it can be seen from this argument that we can choose $\tilde{D}=0$ in case $D=0$, which completes the proof.
\end{proof}
\end{lemma} 

The following lemma provides a standard lower semi-continuity and compactness result. Since it will be used frequently in the paper, we provide its proof for the sake of completeness.

\begin{lemma}\label{lem:general_existence}
Let $X,W$ and $Y$ be Banach spaces such that bounded sequences in $X$ and $W$ admit weak*-convergent subsequences. For $\Theta:\domain(\Theta) \subset X \rightarrow W$ a weak* to weak* closed operator and $F:W \rightarrow [0,\infty)$ convex and weak* lower semi-continuous, suppose that:
\begin{itemize}
\item[i)] There exists a finite dimensional subspace $L \subset W$ such that $F(u+v) = F(u)$ for all $u\in W$, $v \in L$, a continuous, linear projection $P_{L}:W \rightarrow L$ and $C>0$, $D\geq 0$ such that 
\begin{equation}\label{eq:poincforlemma}
\|z - P_{L}z\|_X  \leq C F(z)+D \qquad \forall z \in W.
\end{equation}
\item[ii)] The space $\ker(\Theta)$ is finite dimensional and there exists a continuous projection $P_{\ker(\Theta)}:X \rightarrow \ker(\Theta)$ and $B>0$ such that 
\begin{equation}\label{eq:coerctheta}
\|u - P_{\ker(\Theta)}u\|_X \leq B \|\Theta u \|_W \qquad \forall u \in \domain(\Theta) \subset  X.
\end{equation}
\end{itemize}
Then, defining $F \circ \Theta:X \rightarrow [0,\infty]$ as 
\[
(F \circ \Theta)(u) =
\begin{cases}
F(\Theta u) & \text{if } u \in \domain(\Theta),\\
\infty &\text{else,}
\end{cases}
\]
implies that $F\circ \Theta$ is weak* lower semi-continuous.

Further, for $K: X \rightarrow Y$ a linear, bounded operator, for any sequence $(w^k)_k$ in $X$ such that both $(Kw^k)_k$ and $((F\circ \Theta) (w^k))_k$ are bounded, there exists a sequence $(\tilde{w}^k)_k$ such that $w^k - \tilde{w}^k \in \ker(K) \cap \Theta^{-1}( L) $ and both $(\tilde{w}^k)_k$ and $(\Theta \tilde{w}^k)_k$ are bounded.

\end{lemma}
\begin{proof}
At first note that, by Lemma \ref{lem:projection_transfer} and using that $\rg(\Theta)$ is closed (Lemma \ref{lem:equivalence_coercivity_closed_range}) there exist $\tilde{C} >0$, $\tilde D \geq 0$ such that
\[ \|z - P_{\rg(\Theta) \cap L} z \|_W \leq \tilde{C} F(z)+\tilde{D} \quad \text{for all } z \in \rg(\Theta) ,\]
where $P_{\rg(\Theta) \cap L}$ is a bounded linear projection onto the finite dimensional space $\rg(\Theta) \cap L$.

Defining $\ker(\Theta)^\perp:= \rg (\Id - P_{\ker(\Theta)}) \cap \domain(\Theta)$, it is easy to see that $\ker(\Theta)^\perp$ is a complement of $\ker(\Theta)$ in $\domain (\Theta)$ and that $\Theta$ is injective on $\ker(\Theta)^\perp$. 
Hence, with $M:= \Theta^{-1}( L) $, we can define 
$P:\domain(\Theta) \rightarrow M$ as
\[ Pw := \Theta^{-1} P_{\rg(\Theta) \cap L} \Theta w + P_{\ker(\Theta)} w, \]
where  $\Theta^{-1}$ is the inverse of $\Theta:\ker(\Theta)^\perp \rightarrow \rg(\Theta)$. Then, we observe that $F\circ \Theta$ is invariant on $M$ and that $M$ is a finite dimensional vector space. Further, for any $w \in \domain(\Theta)$, 
\begin{equation} \label{eq:coercivity_composition}
\begin{aligned}
\| w - Pw \| _X 
&= \|w - \Theta^{-1} P_{\rg(\Theta) \cap L} \Theta w - P_{\ker(\Theta)} ( w - \Theta^{-1} P_{\rg(\Theta) \cap L } \Theta w )\|_X \\
& \leq B \|\Theta( w - \Theta^{-1} P_{\rg(\Theta) \cap L} \Theta w)\|_W \\
& = B \|\Theta( w - Pw) \|_W \\
& =  B \|\Theta w - P_{\rg(\Theta) \cap L} \Theta w\|_W  \leq B\tilde{C} F(\Theta w ) + B\tilde{D}.
\end{aligned} 
\end{equation}

We now prove that $F\circ \Theta$ is weak* lower semi-continuous. Take $(w^k)_k$ weak* converging to some $w \in X$. Without loss of generality, we can assume that $\liminf_k (F\circ \Theta) (w^k) < \infty$ and, up to extracting a subsequence, we can choose $w^k$  such that it realizes the $\liminf$. Then, by the estimate \eqref{eq:coercivity_composition} above, both $(\|w^k - Pw^k \|_X)_n$ and $ (\| \Theta (w^k - P w^k )\|_W)_k $ are bounded such that, by taking a non-relabeled subsequence, we can assume that $\Theta (w^k - Pw^k) $ weak* converges to some $z \in W$ and $w^k - Pw^k$ weak* converges to some $v \in X$. Weak* closedness of $\Theta$ then implies that $v \in \domain (\Theta)$ and $\Theta v =z$. Also, thanks to the finite dimensionality of $M$ we have that $w-v = \text{\textsf{w*}-}\lim_k w^k - (w^k - Pw^k) \in M$ such that, by weak* lower semi-continuity of $F$, we conclude
\[ F(\Theta w) = F(\Theta v) \leq \liminf_k F(\Theta (w^k - Pw^k) ) = \liminf_k F(\Theta w^k  ) \]
as claimed.

Now take $(w^k)_k$ in $X$ such that both $(Kw^k)_k$ and $(F\circ \Theta (w^k))_k$ are bounded. Let $Z$ be a complement of $\ker(K) \cap M$ in $M$ and $P_Z: M \rightarrow Z$ a continuous linear projection onto $Z$, i.e., $\Id - P_Z : M \rightarrow M$ is a projection onto $\ker(K) \cap M$.
Then, define
 $\tilde{w}^k = w^k - ( Pw^k - P_Z Pw^k)$. It follows that $w^k - \tilde{w}^k = (\Id - P_Z)Pw^k  \in \ker(K) \cap M$ and that $(w^k - Pw^k)_k$ is bounded in $X$ by the estimate \eqref{eq:coercivity_composition} and the boundedness of $(F\circ \Theta (w^k))_k$. Now since $K$ is injective on the finite dimensional space $Z$ and $KP_Z = K$, we further get for $A>0$ a generic constant that
\[
 \| P_Z Pw^k \|_X \leq A \|K P_Z Pw^k \|_Y \leq A( \|K w^k \|_Y + \|K (w^k - Pw^k) \| _Y) < A < \infty
\]
such that $(\tilde{w}^k)_k$ is bounded. Boundedness of $\Theta$ on the finite dimensional space $Z \subset \domain(\Theta)$ together with the estimate \eqref{eq:coercivity_composition} further implies that $(\Theta \tilde{w}^k)_k$  is bounded. 
\end{proof}

Using the previous lemma, we now deal with the kernel and coercivity of regularization graph functionals.

\begin{teo} \label{thm:graphessentiallycoercive}
Let $\Gc _\alpha$ be a regularization graph with weights $(\alpha_e)_e$, underlying graph structure $G = (V,E)$ and root node $\hat{n} \in V$,  and let $R_\alpha = R (\Gc_\alpha) : {X_{\hat{n}}} \rightarrow [0,+\infty]$ be the associated regularization graph functional. Then:
\begin{enumerate}
\item[i)] The infimum in the recursive representation of $R_\alpha$ \eqref{eq:recursiveformula}  provided  in Lemma \ref{lem:primal_graph_recursion} is attained for any $u \in X_{\hat{n}}$.
\item[ii)]  $R_\alpha$ is weak* lower-semicontinous.
\item[iii)] There exists a finite dimensional subspace $L \subset X_{\hat{n}}$ such that $R_\alpha$ is invariant on $L$ and for $P_{L}:X \rightarrow L$ a bounded, linear projection there exist $C >0$, $D\geq 0$ such that, for $u \in X_{\hat n}$,
\begin{equation} \label{eq:graph_coercive}
 \|u- P_{L} u\|_{X_{\hat{n}}} \leq C  R_\alpha(u) + D .
 \end{equation}
\end{enumerate}

\end{teo}

\begin{proof}
We prove the result via induction over the height of the graph. Assume that the claimed assertions hold true for any regularization graph of height less than $h$ and let $\Gc_\alpha $ be a regularization graph of height $h$ with associated functional $R_\alpha = R(\Gc_\alpha)$. If $h=0$, the results hold trivially with $L = \{0\}$ thanks to Assumptions \ref{ass:lsc_functional}, \ref{ass:coerc_functional} and the definition of trivial regularization graphs. Otherwise, using Lemma \ref{lem:primal_graph_recursion} we write $R_\alpha$ as 
  \begin{equation} \label{eq:recursion_tmp}
R_\alpha(u) = \inf\Big
  \{ \Psi_{\hat{n}} \big(u-  \sum_{ \hat{e} \in \hat{E}} \alpha_{\hat{e}} \Phi_{\hat{e}} w_{\hat{e}} \big) + \sum_{\hat{e} \in \hat{E}}  R_{\alpha^{\hat{e}}}^{\hat{e}}(\Theta_{\hat{e}} w_{\hat{e}}) \Bigst w_{\hat{e}} \in \domain(\Theta_{\hat{e}})   \text{ for all } \hat{e} \in \hat{E}
 \Big\},
 \end{equation} 
 where $\hat E$ are the edges connected to $\hat n$ and for each $\hat e \in \hat E$, $R_{\alpha^{\hat{e}}}^{\hat{e}} : X_{n^{\hat{e}}} \rightarrow [0,\infty]$ is a functional associated to a regularization graph $\Gc_{\alpha^{\hat{e}}}^{\hat{e}}$ with root node $\hat{n}^{\hat{e}}$. %
Also, remember that by \ref{ass:ker_fwd_operator} and \ref{ass:coerc_fwd_operator} (see also Remark \ref{rem:initialremark}) each $\Theta_{\hat{e}}:X_{\hat{e}} \rightarrow X_{\hat{n}^{\hat{e}}}$ has closed range, finite dimensional kernel and satisfies 
\begin{equation} \label{eq:thmeq1}
  \| w - P_{\ker(\Theta_{\hat{e}})}w \|_{X_{\hat{e}}} \leq B^{\hat{e}} \|\Theta_{\hat{e}} w\|_{X_{\hat{n}^{\hat{e}}}} 
\end{equation}
 for $B^{\hat{e}} >0$ and all $w \in \domain(\Theta_{\hat{e}})$.

 Applying the induction hypothesis, each $R_{\alpha^{\hat{e}}}^{\hat{e}}$ is weak* lower-semicontinuous and there exists a finite dimensional subspace $L^{\hat{e}}$ where $R_{\alpha^{\hat{e}}}^{\hat{e}}$ is invariant such that for 
 $P_{L^{\hat{e}}}$ a bounded linear projection there exist constants $C^{\hat{e}} > 0$ and $D^{\hat{e}}\geq 0$ with
 \begin{equation}\label{eq:inductiveessential}
\|v - P_{L^{\hat{e}}}v\|_{X_{\hat{n}^{\hat{e}}}} \leq C^{\hat{e}} R_{\alpha^{\hat{e}}}^{\hat{e}}(v) + D^{\hat{e}} \qquad  \forall v\in X_{\hat{n}^{\hat{e}}}.
\end{equation}
 Moreover, applying Lemma \ref{lem:projection_transfer} with $L = L^{\hat{e}}$, $G =P_{L^{\hat{e}}}$, $R = R_{\alpha^{\hat{e}}}^{\hat{e}}$ and $K = \rg(\Theta_{\hat{e}})$ (that is closed thanks to \ref{ass:ker_fwd_operator}, \ref{ass:coerc_fwd_operator} and Remark \ref{rem:initialremark})  yields that for $P_{\rg(\Theta_{\hat{e}}) \cap L^{\hat{e}}}$ a linear, continuous projection there exists $\tilde C^{\hat{e}} > 0$ and $\tilde D^{\hat{e}} \geq 0$ with
\begin{equation}\label{eq:formulaessential}
\|v - P_{\rg(\Theta_{\hat{e}}) \cap L^{\hat{e}}} v\|_{X_{\hat{n}^{\hat{e}}}} \leq \tilde C^{\hat{e}} R_{\alpha^{\hat{e}}}^{\hat{e}}(v) + \tilde D^{\hat{e}} \qquad  \forall v\in X_{\hat{n}^{\hat{e}}}.
\end{equation}

Now proceeding as in the proof of Lemma \ref{lem:general_existence}, we define $\ker(\Theta_{\hat{e}})^\perp:= \rg (\Id - P_{\ker(\Theta_{\hat{e}})}) \cap \domain(\Theta_{\hat{e}})$. It is easy to see that $\ker(\Theta_{\hat{e}})^\perp$ is a complement of $\ker(\Theta_{\hat{e}})$ in $\domain (\Theta_{\hat{e}})$ and that $\Theta_{\hat{e}}$ is injective on $\ker(\Theta_{\hat{e}})^\perp$. 
Hence, with 
\begin{equation} \label{eq:kernel_Rl_definition} 
M^{\hat{e}}:= \Theta_{\hat{e}}^{-1}( L^{\hat{e}}) ,
\end{equation} we can define 
$P^{\hat{e}}:\domain(\Theta_{\hat{e}}) \rightarrow M^{\hat{e}}$ as
\begin{equation}\label{eq:projection_kernel_Rl_definition}
 P^{\hat{e}}w := \Theta_{\hat{e}}^{-1} P_{\rg(\Theta_{\hat{e}}) \cap L^{\hat{e}}} \Theta_{\hat{e}} w + P^{\hat{e}}_{\ker(\Theta_{\hat{e}})} w, 
 \end{equation}
where  $\Theta_{\hat{e}}^{-1}$ is the inverse of $\Theta_{\hat{e}}:\ker(\Theta_{\hat{e}})^\perp \rightarrow \rg(\Theta_{\hat{e}})$. Then, we observe that $R_{\alpha^{\hat{e}}}^{\hat{e}} \circ \Theta_{\hat{e}}$ is invariant on $M^{\hat{e}}$ and that $M^{\hat{e}}$ is a finite dimensional vector space. 
Then, estimating as in \eqref{eq:coercivity_composition}, we obtain that for each $w \in X_{\hat{e}}$ that
\begin{equation} \label{eq:coercivitychosen}
 \| w - P^{\hat{e}}w\| _{X_{\hat{e}}}  \leq B^{\hat{e}} \tilde{C}^{\hat{e}} R_{\alpha^{\hat{e}}}^{\hat{e}} (\Theta_{\hat{e}} w) + B^{\hat{e}}\tilde{D}^{\hat{e}}.
\end{equation}

Now we write the recursive representation \eqref{eq:recursion_tmp} in a vectorized form as follows: With $w = (w_{\hat{e}})_{\hat{e} \in \hat E} \in X:= \bigtimes_{\hat{e} \in \hat{E}} X_{\hat{e}}$, define $K :X\rightarrow X_{\hat{n}}$ and $\Theta : \bigtimes_{\hat{e} \in \hat{E}} \domain(\Theta_{\hat{e}}) \subset X  \rightarrow \bigtimes _{\hat{e}\in \hat{E}}  X_{{\hat n}^{\hat{e}}}$ as  
\begin{equation}
Kw := \sum_{\hat{e} \in \hat{E}} \alpha_{\hat{e}} \Phi_{\hat{e}} w_{\hat{e}} \quad \text{and} \quad \Theta w := (\Theta_{\hat{e}}w_{\hat e})_{\hat e \in \hat E}.
\end{equation}
Further, with $v = (v_{{\hat n}^{\hat{e}}})_{\hat e \in \hat E} \in \bigtimes _{\hat{e}\in \hat{E}} X_{{\hat n}^{\hat{e}}}$ define $F : \bigtimes _{\hat{e}\in \hat{E}} X_{{\hat n}^{\hat{e}}} \rightarrow [0,+\infty]$ as 
\begin{equation}
\quad F(v) = \sum_{\hat e \in \hat E} R_{\alpha^{\hat{e}}}^{\hat{e}} (v_{{\hat n}^{\hat{e}}}).
\end{equation}
Finally, define $F\circ \Theta:X \rightarrow [0,+\infty]$ and $P_M :  X \rightarrow M$, where  $M :=  \bigtimes _{\hat{e}\in \hat{E}} M^{\hat{e}} $ is a finite dimensional subspace of $X$, as
\begin{equation}
(F\circ \Theta)(w) = \begin{cases}
\sum_{\hat{e} \in \hat{E}} R_{\alpha^{\hat{e}}}^{\hat{e}} (\Theta_{\hat{e}} w_{\hat{e}}) & \text{if } w \in \bigtimes_{\hat{e} \in \hat{E}} \domain(\Theta_{\hat{e}}),  \\ \infty & \text{else,}
\end{cases}
\quad  \mbox{and }\quad P_M w = (P^{\hat{e}} w_{\hat{e}})_{\hat{e} \in \hat{E}}.
\end{equation}
We observe that $F\circ \Theta$ is invariant on $M = \Theta^{-1}(L)$ where $L = \bigtimes_{\hat e \in \hat E} L^{\hat e}$. With these definitions we can write the recursive representation \eqref{eq:recursion_tmp} in the vectorized form
\begin{align*}
R_\alpha(u) =  \inf\Big
  \{ \Psi_{\hat{n}} \big(u-  Kw \big) +  (F\circ \Theta)(w) \Bigst w \in \bigtimes_{\hat{e} \in \hat{E}} \domain(\Theta_{\hat{e}})
 \Big\}.
\end{align*}

Now we first show weak* lower semi-continuity of $R_\alpha$ on $X_{\hat{n}}$. To this aim, take $(u^k)_k$ to be a sequence in $X_{\hat{n}}$ converging weakly* to some $u \in X_{\hat{n}}$. Without loss of generality, we can assume  that $\liminf_k R_\alpha (u^k) < \infty$ and, up to extracting a subsequence, we can assume that $(u^k)_k$  realizes the $\liminf$.
Next, take $(w^k)_k$ to be a sequence in $X$ such that
\[  \Psi_{\hat{n}} (u^k - Kw^k) + (F\circ \Theta)(w^k) \leq R_\alpha(u^k) + 1/k .\]
Together with assumption \ref{ass:coerc_functional}, this implies boundedness of $(Kw^k)_k$ and $(F\circ \Theta)(w^k)$.  We now want to apply Lemma \ref{lem:general_existence} choosing $\bigtimes_{\hat e \in \hat E} L^{\hat e}$ for $L$ and $v\mapsto (P_{L^{\hat e}}v_{\hat{n}^{\hat{e}}})_{\hat e \in \hat E}$ with $v = (v_{{\hat n}^{\hat{e}}})_{\hat e \in \hat E} \in \bigtimes _{\hat{e}\in \hat{E}} X_{{\hat n}^{\hat{e}}}$ for $P_L$. Note that $F$ is weak* lower semicontinuous,  it is invariant on $\bigtimes_{\hat e \in \hat E} L^{\hat e}$ and the estimate \eqref{eq:poincforlemma} holds thanks to the inductive assumption \eqref{eq:inductiveessential} applied to each $R_{\alpha^{\hat{e}}}^{\hat{e}}$. Moreover, it can be readily verified that the operator $\Theta$ is weak* to weak* closed, has finite dimensional kernel and the estimate \eqref{eq:coerctheta} holds as a direct consequence of \eqref{eq:thmeq1} applied to each $\Theta_{\hat e}$. So, applying  Lemma \ref{lem:general_existence} and using the weak* to weak* closedness of $\Theta$ we can select, up to a non-relabeled subsequence, $(\tilde{w}^k)_k$ such that $\tilde{w}^k - w^k \in \ker(K) \cap M$, $\tilde{w}^k $ converges weak* to some $w \in X_{\hat{n}}$ and $\Theta \tilde{w}^k$ converges weak* to $\Theta w$. Weak* lower semi-continuity of $\Psi_{\hat{n}}$ and $F\circ \Theta$ together with weak* continuity of $K$ and the invariance of $F \circ \Theta$ on $M$ then imply
\[R_{\alpha}(u) \leq  \Psi_{\hat{n}} (u - Kw) + F(\Theta w)\leq  \liminf_k  \Psi_{ \hat{n}} (u^k - Kw^k) + F(\Theta  w^k) \leq  \liminf_k R_\alpha(u^k),
\]
thus weak* lower semi-continuity of $R_\alpha$ on $X_{\hat{n}}$ follows. In addition, given any $u \in X_{\hat{n}}$, choosing $u^k = u$ for every $k$ implies existence of minimizers in \eqref{eq:recursion_tmp} as claimed.

Now we note that $R_\alpha$ is invariant on the finite dimensional space
 \[ L:= \bigplus_{\hat{e} \in \hat{E}} \alpha_{\hat{e}}\Phi_{\hat{e}}(M^{\hat{e}}).\]
Indeed, if $u= \sum_{\hat{e} \in \hat{E}} \alpha_{\hat{e}}\Phi_{\hat{e}} u_{\hat{e}}$ with $u_{\hat{e}} \in M^{\hat{e}}$ it follows from $(F\circ \Theta)((\alpha_{\hat e}u_{\hat{e}})_{\hat{e} \in \hat{E}})=0$ that $R_\alpha(\tilde{u} + u) = R_\alpha(\tilde{u} )$ for all $\tilde{u} \in X_{\hat{n}}$.
 
Next we show the coercivity estimate. To this aim, for any given $u \in X_{\hat{n}}$, we select $(w_{\hat{e}})_{\hat{e}\in \hat{E}}$ to be minimizers in \eqref{eq:recursion_tmp} and define $v = v(u):= \sum_{\hat{e} \in \hat{E}}\alpha_{\hat{e}} \Phi_{\hat{e}} P^{\hat{e}} w_{\hat{e}}  $. Using \eqref{eq:coercivitychosen}, the coercivity of $\Psi_{\hat n}$ (see Remark \ref{rem:equivcoercivity}) and the continuity of $\Phi_{\hat{e}}$ we  estimate
\begin{align*}
\|u - v(u)\|_{X_{\hat{n}}} & = \left\|u - \sum_{\hat{e} \in \hat{E}}  \alpha_{\hat{e}}\Phi_{\hat{e}} P^{\hat{e}} w_{\hat{e}}\right\|_{X_{\hat{n}}} \\ 
&= \left\|u -\sum_{\hat{e} \in \hat{E}}\alpha_{\hat{e}}  \Phi_{\hat{e}} w_{\hat{e}} + \sum_{\hat{e} \in \hat{E}} \alpha_{\hat{e}}  \Phi_{\hat{e}} w_{\hat{e}} - \sum_{\hat{e} \in \hat{E}}  \alpha_{\hat{e}}\Phi_{\hat{e}} P^{\hat{e}} w_{\hat{e}} \right\|_{X_{\hat{n}}} \\
&  \leq C_1 \Psi_{\hat{n}} \left(u - \sum_{\hat{e}\in \hat{E}} \alpha_{\hat{e}}  \Phi_{\hat{e}} w_{\hat{e}}\right)  + D_1 + C_2\sum_{\hat{e} \in \hat{E}} \alpha_{\hat{e}}\| w_{\hat{e}} - P^{\hat e} w_{\hat{e}}\|_{X_{\hat{e}}} \\ 
&\leq C_1 \Psi_{\hat{n}} \left(u - \sum_{\hat{e} \in \hat{E}} \alpha_{\hat{e}}  \Phi_{\hat{e}} w_{\hat{e}} \right)  + D_1  +C_2 \sum_{\hat{e} \in \hat{E}} ( \alpha_{\hat{e}} B^{\hat e}  \tilde C^{\hat e} R_{\alpha^{\hat{e}}}^{\hat{e}} ( \Theta_{\hat{e}} w_{\hat{e}}) +  \alpha_{\hat{e}} B^{\hat e} \tilde D^{\hat{e}})\\
& \leq C   R_\alpha(u) + D,
\end{align*}
with suitable $C >0$ and $D \geq 0$.
Using Lemma \ref{lem:projection_transfer} with $G(u) = v(u)$ and $K = X_{\hat{n}}$, this implies the claimed coercivity.
\end{proof}

\begin{rem}[Domain and invariant subspace]\label{rem:domain_invariant_subspace} With $R_\alpha :X_{\hat n} \rightarrow [0,\infty]$ a regularization graph functional with recursive representation as in Lemma \ref{lem:primal_graph_recursion}, we have that
\[ \domain(R_\alpha) =  \domain(\Psi_{\hat{n}}) + \bigplus_{\hat{e} \in \hat{E}} \alpha_{\hat{e}}\Phi_{\hat{e}}\Big(\Theta_{\hat{e}}^{-1} (\domain(R_{\alpha^{\hat{e}}}^{\hat{e}}) ) \Big),\]
where equality follows since the infimum in the recursive represenation of Lemma \ref{lem:primal_graph_recursion} is attained.
Further, a simple contradiction argument shows that the finite dimensional subspace $L$ where $R_\alpha$ is invariant and coercive in the sense of \eqref{eq:graph_coercive} is unique (and will henceforth be called \emph{the invariant subspace of $R_\alpha$}). We also have the recursive representation
\[ L =   \bigplus_{\hat{e}\in \hat{E}} \alpha_{\hat{e}}\Phi_{\hat{e}}\Big(\Theta_{\hat{e}}^{-1} (L^{\hat{e}}) )  \Big),\]
with $L_{\hat{e}}$ the invariant subspace of $R_{\alpha_{\hat{e}}}^{\hat{e}}$. Via induction, this implies in particular that $L$ only depends on the support of $(\alpha_e)_e$, i.e., the set of edges where $\alpha_e \neq 0$, but not on their values.
\end{rem}
Existence for the infimum over edge variables associated with edges connected to the root node in the recursive representation of $R_\alpha$, as stated in Theorem \ref{thm:graphessentiallycoercive}, immediately implies, via induction, existence of infimizing edge variables in the definition of the regularization  graph for all edges. This is stated in the following corollary.
\begin{cor}[Existence of infimizing edge variables]\label{cor:existenceedges}
Let $\Gc_\alpha$ be a regularization graph with root node $\hat{n}$ and $R_\alpha = R(\Gc_\alpha)$ be the associated functional. Then, for each $u \in X_{\hat{n}}$, there exists $(w_e)_{e \in E}$ such that
\[ R_{\alpha}(u) =  \sum_{n \in V} \Psi_{n} 
\Big( \Theta_{(n^-,n)} w_{(n^-,n)} - \sum_{(n,m) \in E} \alpha_{(n,m)} \Phi_{(n,m)} w_{(n,m)}
\Big) ,
\] i.e., the infimum in the definition of the regularization graph functional is attained.
\end{cor}
\begin{rem}[Regularity]
Let us observe how an infimal-based combination and an extension of regularization graphs affect the coercivity of regularization graph functionals as in Theorem \ref{thm:graphessentiallycoercive}.
\begin{itemize}
\item When combining two different regularization graph functionals defined on two different normed spaces via infimal convolution, the norm for underlying joint space and hence the norm for the coercivity estimate needs to be the weaker of the two norms. In the construction of Proposition \ref{prop:infcon_sum_reggraphs}, this is reflected in the assumption that the embeddings $\emb^1_{X} :X_{\hat{n}^1} \rightarrow X$ and $\emb^2_{X} :X_{\hat{n}^2} \rightarrow X$ need to be weak* continuous. 
An example here is the infimal convolution of $\TV$ and $\TV^2$, where $\TV$ and $\TV^2$ are coercive up to their kernels on $L^{d/(d-1)}$ and $L^{d/(d-2)}$, respectively (here the exponents are set to $\infty$ for $d=1$ and $d\leq 2$, respectively). The infimal-convolution-based combination of the regularization graphs corresponding to $\TV$ and $\TV^2$, according to Proposition \ref{prop:infcon_sum_reggraphs}, is then coercive on the weaker space $L^{d/(d-1)}$, see \cite[Section 4.2]{Holler19_ip_review_paper} for details.
\item When extending a given regularization graph with a further edge, stronger norms can be chosen. A particular example is the composition of two gradient operators $\nabla _1 ,\, \nabla _2$ to obtain $\TV^2 = \|(\nabla _2 \circ \nabla _1 )\cdot \|_M$. Given that $\nabla _2$ is coercive up to its kernel on $L^{d/(d-1)}$, we can define $\nabla _1 $ as operator from $L^{d/(d-2)}$ to $L^{d/(d-1)}$ and again obtain coercivity up to constant functions between those spaces by standard Sobolev embeddings. In this case, the overall regularization graph functional corresponding to $\TV^2$ is coercive up to affine functions with respect to the norm in $L^{d/(d-2)}$, which is the improved regularity of $\TV^2$, see \cite[Section 3]{Holler19_ip_review_paper}.
\end{itemize}
\end{rem}

The following proposition deals with the dependence of the coercivity estimate for a regularization graph functional $R_\alpha$ on the weights $\alpha$.
\begin{prop}[Dependence on the weights]\label{prop:coercivity_dependence_on_weights} Let $\Gc_\alpha$ be a regularization graph with weights $\alpha$ and root node $\hat{n}$, let $L$ be the invariant subspace of $R_\alpha = R(\Gc_\alpha)$ that only depends on the structure of the regularization graph $\Gc$ and the support of $(\alpha_e)_{e \in E}$, let $K \subset X_{\hat{n}} $ be a closed subspace and let $P_{K \cap L}:X_{\hat{n}} \rightarrow L$ be a bounded, linear projection.
Then, for any $A \geq \max\{ \alpha_e \st e \in E\}$ there exist $C>0$, $D\geq 0$ that only depend on $A$ such that for any $u \in X_{\hat{n}}$ we have
\[ \|u- P_{L} u\|_{X_{\hat{n}}} \leq C_{\alpha} C  R_\alpha(u) + D ,\]
where
\[ C_{\alpha}:= \max \big \{ \prod_{e \in F} \alpha_e \st F \subset E \text{ is either empty or a chain with $\hat{n} \in V$ as root } \big \},\]
using the same conventions as in Lemma \ref{lem:weight_equivalence}.
\begin{proof}
This follows from first applying Theorem \ref{thm:graphessentiallycoercive} and, subsequently, Lemma \ref{lem:projection_transfer} to $R_{\tilde{\alpha}} = R(\Gc_{\tilde{\alpha}})$, where $\tilde{\alpha}_e = A$ if $\alpha_e >0$ and $\tilde{\alpha}_e = 0$ else, and then using the estimate of Lemma \ref{lem:weight_equivalence}.
\end{proof}
\end{prop}

The next proposition deals with extending a regularization graph functional $R_\alpha:X_{\hat{n}} \rightarrow [0,\infty]$ by infinity to a Banach space $X$ not satisfying hypothesis \ref{ass:weak_star_compactness}, but with $X_{\hat{n}} \hookrightarrow X$. The prototypical application of this result would be, e.g., to extend $R_\alpha$ from $L^p(\Omega)$ with $p>1$ to $L^1(\Omega)$, where $\Omega \subset \R^d$ is a bounded domain. Note that directly choosing $X_{\hat{n}} = L^1(\Omega)$ is not feasible since, in general, bounded sequences in $L^1(\Omega)$ do not admit weak* convergent subsequences (or weakly convergent subsequences since $L^1(\Omega)$ is generally not a dual space).
\begin{prop}[Extended domain] \label{prop:reg_graph_extension_to_X}
Let $\Gc_\alpha$ be a regularization graph with weights $\alpha$ and root node $\hat{n}$. Let $L$  be its invariant space. Suppose that $X$ is a Banach space such that $X_{\hat{n}} \hookrightarrow X$ and $X_{\hat{n}}$ is reflexive. Then, with $R_\alpha= R(\Gc_\alpha)$ extended to $X$ via $R_\alpha(x) = \infty$ for $x \in X\setminus X_{\hat{n}}$, $R_\alpha$ is convex and lower semi-continuous w.r.t. weak convergence in $X$, and for any continuous, linear projection $\tilde P_L:X \rightarrow L$, there exists $C >0$, $D \geq 0$ such that
\begin{equation}\label{eq:coercextension}
\|u-\tilde  P_Lu\|_X \leq C R_\alpha(u) + D \quad \text{for all } u \in X.
\end{equation}
\begin{proof}
Convexity is immediate and the coercivity estimate follows directly from the continuous embedding $X_{\hat{n}} \hookrightarrow X$ and Theorem \ref{thm:graphessentiallycoercive} by defining $P_L$ as the restriction of $\tilde P_L$ to $X_{\hat{n}}$. Regarding weak lower semi-continuity, take $(u^k)_k$ to be a sequence in $X$ converging weakly to some $u \in X$. Without loss of generality, we can assume  that $\liminf_k R_\alpha (u^k) < \infty$ and, up to extracting a subsequence, we can choose $u^k$  such that it realizes the $\liminf$ and $u^k \in X_{\hat{n}}$ for every $k$.
With $P_L:X_{\hat{n}} \rightarrow L$ a continuous, linear projection, from the coercivity estimate of Theorem \ref{thm:graphessentiallycoercive} applied to  $R_\alpha : X_{\hat{n}} \rightarrow [0,\infty]$ we obtain that $v^k := u^k - P_L u^k $ is 
bounded in $X_{\hat{n}}$ such that we may assume weak convergence of the latter to $v \in X_{\hat{n}}$. 
Also, by the embedding $X_{\hat{n}} \hookrightarrow X $, $P_L u^k = u^k - v^k$ is bounded in $X$ and hence, by finite dimensionality of $L$, admits a subsequence converging to some $z \in X_{\hat{n}}\cap L$ with respect to $\|\cdot \|_{X_{\hat{n}}}$. Again by the embedding $X_{\hat{n}} \hookrightarrow X $, weak convergence in $X_{\hat{n}}$ implies weak convergence in $X$ such that, by uniqueness of the weak limit, $u=v+z\in X_{\hat{n}}$. Lower semi-continuity of $R_\alpha$ with respect to weak convergence in $X_{\hat{n}}$ finally implies
\[ R_\alpha(u) = R_\alpha(v+z) \leq \liminf_k R_\alpha(v^k + P_L u^k) = \liminf_k R_\alpha (u^k)
\]
implying the lower semi-continuity of $R_\alpha$ with respect to weak convergence in $X$. 
\end{proof}
\end{prop}
\begin{rem}\label{rem:reg_graph_extension_to_X}  It can be observed that, in the above result, reflexivity of $X_{\hat{n}}$ (instead of just requiring that bounded sequences admit weak* convergent subsequences) is only needed to conclude from (weak) convergence of $(v^k + P_Lu^k)_k$ to $v+z$ in $X_{\hat{n}}$ and the weak convergence of  $(v^k + P_Lu^k)_k$ to $u$ in $X$ that, by uniqueness of limits, $v+z=u$ follows. The same could be achieved for weak* convergence of $(v^k + P_Lu^k)_k$ in $X_{\hat{n}}$, thus not requiring reflexivity, if, for instance, $X_{\hat{n}} = L^\infty(\Omega)$ and $X = L^1(\Omega)$.
\end{rem}

\section{Predual formulation of regularization graphs} \label{sec:predual}
The goal of this section is to provide an equivalent, predual reformulation of regularization graphs. Remember that a regularization graph functional $R_\alpha:X_{\hat{n}} \rightarrow [0,\infty]$ can be written in a vectorized form as
\[
R_\alpha(u) = \inf \left\{ \Psi_u(v) \st v \in \rg(\Lambda_\alpha) \right\}
\]
with $\Lambda_\alpha $ and $\Psi_u$ given in \eqref{eq:vectorized_operator_graph} and \eqref{eq:vectorized psi_reduced}, respectively. With $\Lambda^\#_\alpha$ and $\Psi^\#_u$ predual versions of $\Lambda_\alpha $ and $\Psi_u$, respectively, our goal is to show that every regularization graph functional $R_\alpha$ can be written equivalently as
\[
R_\alpha(u)  =  -\inf \left\{  \Psi_u^\#(v)  \st v  \in \ker (\Lambda_\alpha^\#)\right\}. 
\]

To this aim, we need in particular that the functionals $\Psi_n$ and the operators $\Theta_e$ and $\Phi_e$ admit predual versions. By an application of  the Fenchel--Moreau theorem \cite[Proposition 4.1]{CAAV} it is easy to see that there exist
convex, proper, lower semicontinuous functionals $\Psi^\#_n:X_n^\# \rightarrow [0,\infty]$ such that their convex conjugates are $\Psi_n$.

\begin{lemma} \label{lem:predual_node_functional}
For each $n \in V$, $ \Psi_n:X_n \rightarrow [0,\infty]$ is the convex conjugate of a convex, proper, lower semicontinuous functional $\Psi^\#_n:X_n^\# \rightarrow [0,\infty]$.
\end{lemma}
\begin{proof}
Consider the dual pair $(V,V^*)$ for $V = (X_n,\text{\textsf{w*}})$ and $V^* = (X^\#_n, \text{\textsf{w}})$. Note that $\Psi_n$ is convex, proper and lower semicontinuous on $V$. Therefore by the Fenchel--Moreau theorem \cite[Proposition 4.1]{CAAV} there holds that $\Psi_n^{**} = \Psi_n$. In particular, defining $\Psi^\#_n = \Psi_n^*$, we have that $\Psi^\#_n : X^\#_n \rightarrow [-\infty,+\infty]$ is proper, convex and strongly lower semicontinuous and its convex conjugate is $\Psi_n$. The positivity of $\Psi^\#_n$ follows from Assumption \ref{ass:psi_0_is_0}.
\end{proof}

Moreover, Remark \ref{rem:initialremark} ensures the existence of a bounded predual of $\Phi_{(n,m)}$ that we are going to denote by $\Phi^\#_{(n,m)} : X_{n}^\# \rightarrow  X_{(n,m)}^\# $.
Finally, we suppose that the operators $\Theta_e$ admit closed, densely defined preadjoints as stated in the following additional assumption. 
\begin{enumerate}[label=(H\arabic*)]
\setcounter{enumi}{8}
\item\label{ass:predual} For each $e= (n,m)\in E$, $\Theta_{(n,m)}$ is the adjoint of a closed, densely defined operator $\Theta_{(n,m)}^{\#} : \domain(\Theta_{(n,m)}^{\#}) \subset X^\#_m \rightarrow X^\#_{(n,m)}$.
\end{enumerate}

Define the following predual spaces of $X_V$ and $X_E$: 
 \begin{equation}
 X^\#_V = \bigtimes_{n\in V} X^\#_n  \quad \mbox{and} \quad  X^\#_E = \bigtimes_{e\in E} X^\#_e.
 \end{equation}

Now we characterize the predual of the linear operator $\Lambda_\alpha: X_E \rightarrow X_V$ from \eqref{eq:vectorized_operator_graph}.

\begin{lemma} Let $\Gc_\alpha$ be a regularization graph with root node $\hat{n}$ and corresponding operator $\Lambda_\alpha$ as in \eqref{eq:vectorized_operator_graph} such that \ref{ass:predual} holds. Then $\Lambda_\alpha$ is the dual of the linear operator $\Lambda_\alpha^\#: \domain(\Lambda_\alpha^\#)  \subset X^\#_V \rightarrow X^\#_E$ with
\begin{align*}
\domain(\Lambda_\alpha^\#) = \{v \in X^\#_V \, \st \forall (n,m) \in E,\ v_m \in \domain(\Theta^\#_{(n,m)})\} 
\end{align*}
given as
\begin{equation}\label{eq:preduallambda}
(\Lambda^\#_\alpha v)_{(n,m)} = \Theta^{\#}_{(n,m)} v_m - \alpha_{(n,m)}\Phi^{\#}_{(n,m)} v_n \quad \text{for all} \quad v \in \domain(\Lambda_\alpha^\#).
\end{equation}
\end{lemma}
\begin{proof}
First, let us verify that $\domain((\Lambda_\alpha^\#)^*) = \domain(\Lambda_\alpha)$, where $\domain(\Lambda_\alpha) = \bigtimes_{e\in E} \domain (\Theta_e)$. Note that for every $v \in \domain(\Lambda_\alpha^\#)$ and $w \in X_E$ it holds
\begin{align}
\langle w ,\Lambda_\alpha^\# v \rangle & =  \sum_{(n,m)\in E} \langle w_{(n,m)} ,  
 \Theta^\#_{(n,m)} v_m - \alpha_{(n,m)}\Phi^{\#}_{(n,m)} v_n \rangle
  \nonumber \\ 
  & =  \sum_{m\in V} \sum_{(n,m) \in E}  \langle w_{(n,m)}, \Theta^\#_{(n,m)} v_m \rangle -\sum_{n\in V}  \sum_{(n,m) \in E}  \langle  w_{(n,m)}, \alpha_{(n,m)}\Phi^{\#}_{(n,m)} v_n \rangle. \label{eq:predomain}
\end{align} 
By the definition of the domain of the adjoint we have
\begin{align}\label{eq:defadj}
\domain((\Lambda_\alpha^\#)^*) & = \{w \in X_E \st \exists C>0 \ \text{ such that }\ \langle w,\Lambda_\alpha^\# v \rangle \leq C \|v\|_{X_V^\#} \quad \forall v \in \domain(\Lambda_\alpha^\#) \}.
\end{align}
Therefore, using \eqref{eq:predomain}, the boundedness of $\Phi_{(n,m)}$ and the fact that $\Theta_{(n,m)}$ is the adjoint of $\Theta_{(n,m)} ^{\#}$ (see Assumption \ref{ass:predual}), we immediately deduce that $\domain(\Lambda_\alpha)\subset \domain((\Lambda_\alpha^\#)^*)$. To prove the opposite inclusion, consider $w  = (w_e)_{e\in E} \notin \domain(\Lambda_\alpha)$. Then, there exists $(n,m) \in E$ such that $w_{(n,m)} \notin \domain(\Theta_{(n,m)})$.
Using that 
\begin{align*}
\domain(\Theta_{(n,m)})& = \domain((\Theta_{(n,m)}^\#)^*)\\
& = \{w \in X_{(n,m)} \st \exists C>0 \ \text{ such that }\ \langle w,\Theta^\#_{(n,m)} v \rangle \leq C \|v\|_{X_m^\#} \, \forall v \in \domain(\Theta_{(n,m)}^\#) \}
\end{align*}
(see Assumption \ref{ass:predual}), we can find a sequence $(v_m^k)_k$ with   $v^k_m \in \domain(\Theta_{(n,m)}^\#) \setminus \{0\}$ for each $k$ such that 
\begin{align}\label{eq:infin}
\lim_{k \rightarrow \infty}\frac{1}{\|v_m^k\|_{X_m^\#}}\langle w_{(n,m)},\Theta^\#_{(n,m)} v_m^k \rangle  = +\infty\,.
\end{align}
Then, with $v^k \in X_V^\#$ defined as $(v^k)_{\tilde{m}} = v^k_m$ if $\tilde{m}=m$ and $(v^k)_{\tilde{m}} = 0$ else, one notices that, using the boundedness of $\Phi_{e}^\#$ for every $e \in E$, there exists a positive constant $C$ independent on $k$ such that
\begin{align}
\frac{1}{\|v^k\|_{X_V^\#}}\langle w,\Lambda_\alpha^\# v^k \rangle & = \frac{1}{\|v_m^k\|_{X_m^\#}}\langle w_{(n,m)},\Theta_{(n,m)}^\# v_m^k \rangle - \sum_{\tilde{n}\, :\,  (m,\tilde{n}) \in E}\frac{\alpha_{(m,\tilde{n})}}{\|v_m^k\|_{X_m^\#}}\langle w_{(m,\tilde{n})},\Phi_{(m,\tilde{n})}^\# v_m^k \rangle  \nonumber\\
& \geq \frac{1}{\|v_m^k\|_{X_m^\#}}\langle w_{(n,m)},\Theta_{(n,m)}^\# v_m^k \rangle - C\label{eq:infin1}.
\end{align}
Therefore, combining \eqref{eq:infin1} and \eqref{eq:infin} with \eqref{eq:defadj}, we deduce that $w \notin \domain((\Lambda_\alpha^\#)^*)$ showing that    $\domain((\Lambda_\alpha^\#)^*) \subset \domain(\Lambda_\alpha)$.

It remains to prove \eqref{eq:preduallambda}. For every $v \in \domain(\Lambda_\alpha^\#)$ and $w \in \domain(\Lambda_\alpha)$, using \eqref{eq:predomain} we have
\begin{align*}
\langle(\Lambda_\alpha^\#)^*w,v \rangle &  = \langle w ,\Lambda_\alpha^\# v \rangle \\
 &   =  \sum_{m\in V \setminus \{\hat{n}\}}  \langle \Theta_{(m^-,m)} w_{(m^-,m)}, v_m \rangle - \sum_{m \in V}\sum_{(m,n) \in E}  \langle  \alpha_{(m,n)}\Phi_{(m,n)} w_{(m,n)}, v_m \rangle \\
&  =  \sum_{m\in V \setminus \{\hat{n}\}} \bigg\langle  \Theta_{(m^-,m)} w_{(m^-,m)} - \sum_{(m,n) \in E}   \alpha_{(m,n)}\Phi_{(m,n)} w_{(m,n)}, v_m  \bigg \rangle\\
& \quad  +  \bigg \langle - \sum_{(\hat{n},n) \in E}    \alpha_{(\hat n,n)}\Phi_{(\hat n,n)} w_{(\hat n,n)} , v_{\hat n} \bigg \rangle \\
& = \langle \Lambda_\alpha w, v\rangle
\end{align*} 
ending the proof. \qedhere
\end{proof}

With this, we define a predual regularization graph functional as follows.
\begin{dfnz}[Predual regularization graph functional] \label{def:predual_graph}
Given a regularization graph $\Gc_\alpha $ with root node $\hat{n}$ and associated functional $R_\alpha = R(\Gc_\alpha)$  such that \ref{ass:predual} holds, we define the predual regularization graph functional $R^\#_\alpha = R^\#(\Gc_\alpha) :  X_{\hat{n}} \rightarrow [0,+\infty]$ as
\begin{align*} 
R_\alpha^\#(u) 
& =  -\inf \left\{  \Psi^\#_{\hat{n}}(v_{\hat{n}}) - \langle u, v_{\hat{n}}\rangle + \sum_{n \in V \setminus \{\hat{n}\}} \Psi^\#_{n}(v_n)  \Biggst  (v_n)_{n \in V} \in  \domain (\Lambda_\alpha^\#),\,  \right. \\
& \left.\qquad \qquad \ \, \Theta^{\#}_{(n,m)} v_m = \alpha_{(n,m)}\Phi^{\#}_{(n,m)} v_n \, \text{ for all } \,  (n,m) \in  E\right\} \\
& =  -\inf \left\{  \Psi_u^\#(v)   \bigst v  \in \ker (\Lambda_\alpha^\#)\right\} 
\end{align*}
with $\Psi^\#_u(v) :=  \Psi^\#_{\hat{n}}(v_{\hat{n}}) - \langle u, v_{\hat{n}}\rangle + \sum_{n \in V \setminus \{\hat{n}\}} \Psi^\#_{n}(v_n) $, and $\Theta^\#_e$, $\Phi^\#_e$ and $\Psi_n^\#$ being the predual operators and functionals of $\Theta_e$, $\Phi_e$ and $\Psi_n$, respectively, according to hypothesis \ref{ass:predual}, Remark \ref{rem:initialremark} and Lemma \ref{lem:predual_node_functional}.
\end{dfnz}

Our goal is now to show that $R_\alpha^\# = R_\alpha$. As first step, we obtain the following proposition.
\begin{prop}\label{prop:duality_essentially_coercive_seminorms}
Assuming again that hypothesis \ref{ass:predual} holds, any predual regularization graph functional $R^\#_\alpha = R^\#(\Gc_\alpha)$ according to Definition \ref{def:predual_graph}, where $\Gc_\alpha$ is a regularization graph with root node $\hat{n}$, can be written as
\begin{equation}
R_\alpha^\#(u) = \min \left\{ \Psi_u(v) \bigst  v \in \overline{\rg(\Lambda_\alpha)}^{w^*} \right\},
\end{equation}
where $\Psi_u$ is defined as in \eqref{eq:vectorized psi_reduced}.
\end{prop}
\begin{proof}
Since $\Psi_n^\#$ is proper, convex and lower semicontinuous according to Lemma \ref{lem:predual_node_functional}, the Fenchel--Moreau theorem \cite[Proposition 4.1]{CAAV} yields $\Psi_n^\# = (\Psi_n^\#)^{**}$. 
Hence, since $\Psi_n$ is coercive thanks to hypothesis \ref{ass:coerc_functional},  using \cite[Theorem 10]{rockafellar}, we deduce that $\Psi_n^\#$ is continuous in zero. Consequently, also $v \mapsto \Psi^\#_{\hat{n}}(v) - \langle u, v\rangle$ is continuous in zero, such that $\domain(\Psi_u)$ contains a neighborhood of zero and hence is absorbing.

In order to apply the Attouch--Br\'ezis theorem \cite[Corollary 2.2]{attouchbrezis}, we note that
\begin{align*}
 \bigcup_{\lambda \geq 0} \lambda (\text{dom}(\Psi_u^\#) + \ker (\Lambda_\alpha^\#)) &  = \bigcup_{\lambda \geq 0} \lambda \text{dom}(\Psi_u^\#)  + \ker (\Lambda_\alpha^\#) =  X_V.
\end{align*}
Therefore, noting that $(\Psi_u^\#)^* = \Psi_u$ we deduce that
\begin{align*}
-\inf \left\{ \Psi_u^\#(v) \bigst  v \in \ker (\Lambda_\alpha^\#) \right\} = \min \left\{ \Psi_u(v) \bigst  v \in (\ker (\Lambda_\alpha^\#))^\perp \right\},
\end{align*}
where the minimum is attained and $(\ker (\Lambda_\alpha^\#))^\perp$ denotes the annihilator of $\ker (\Lambda_\alpha^\#)$.
As $\Lambda_\alpha^\#$ is closed and densely defined we apply Remark 17 in \cite{brezisfunctionalanalysis} to deduce that $(\ker (\Lambda_\alpha^\#))^\perp = \overline{\rg(\Lambda_\alpha)}^{w*}$ and conclude.
\end{proof}

In the next proposition we use Theorem \ref{thm:graphessentiallycoercive} to prove that $\text{rg}(\Lambda_\alpha^\#)$ is weak*-closed, and hence obtain the desired duality formulation.
\begin{prop} Let $\Lambda_\alpha :X_E \rightarrow X_V$ be as in \eqref{eq:vectorized_operator_graph} corresponding to a regularization graph $\Gc_\alpha$. Then, $\rg(\Lambda_\alpha)$ is weak*-closed in $X_V$. If in addition \ref{ass:predual} holds and $R_\alpha = R(\Gc_\alpha) $ and $R_\alpha^\# = R^\#(\Gc_\alpha)$ are the primal and predual regularization graph functionals according to Definitions \ref{def:reg_graph} and \ref{def:predual_graph}, respectively, then 
\[R_\alpha^\# = R_\alpha.\]
\end{prop}
\begin{proof}
We are only going to show weak* closedness of $\text{rg}(\Lambda_\alpha)$ since, under assumption \ref{ass:predual}, the assertion $R_\alpha^\# = R_\alpha$ then follows as immediate consequence of Proposition \ref{prop:duality_essentially_coercive_seminorms} and the definition of $R_\alpha$. Assume that the result holds true for all operators corresponding to a regularization graph of height less than $h$ and let $\Gc_\alpha$ be a regularization graph of height $h$ with corresponding operator $\Lambda_\alpha$ and associated directed graph $G = (V,E)$. The case $h=0$ is immediate since $\text{rg}(\Lambda_\alpha) = \{0\}$.

Denote by $\hat{n}$ the root node of $\Gc_\alpha$  and let $\hat{E}$ be the edges connected to the root node and $\hat{V}$ their corresponding endpoints. 
Further, for $\hat{e} = (\hat{n},\hat{n}^{\hat{e}}) \in \hat{E}$, denote by $G^{\hat{e}} = (V^{\hat{e}},E^{\hat{e}})$ the subtree of $G$ with root node $\hat{n}^{\hat{e}}$ and by $\Lambda^{\hat{e}}_{\alpha^{\hat{e}}}$ the operator corresponding to the subtree $G^{\hat{e}}$. Then, we define $\Gc^{\hat{e}}_{\alpha^{\hat{e}}}$ to be the regularization graph with structure $\Gc^{\hat{e}} = (G^{\hat{e}},(\|\cdot \|_{X_{n}})_{n \in V^{\hat{e}}},(\Theta_e)_{e \in E ^{\hat{e}}},(\Phi_e)_{e \in E^{\hat{e}}})$ and weights $\alpha^{\hat{e}} = (\alpha_{e})_{e \in E^{\hat{e}}}$. It follows that $\Gc^{\hat{e}}_{\alpha^{\hat{e}}}$ is indeed a regularization graph and that the associated functional $R^{\hat{e}}_{\alpha^{\hat{e}}} = R(\Gc^{\hat{e}}_{\alpha^{\hat{e}}}):X_{n^{\hat{e}}} \rightarrow [0,\infty)$ is given as 
\begin{equation} \label{eq:duality_rl_vec_definition}
 R^{\hat{e}}_{\alpha^{\hat{e}}}(u) = \inf \{ \Upsilon^{\hat{e}}_u(v) \st v \in \rg(\Lambda ^{\hat{e}}_{\alpha^{\hat{e}}}) \} 
 \end{equation}
with $\Upsilon^{\hat{e}}_u(v):= \| u + v_{\hat{n}^{\hat{e}}}\|_{X_{\hat{n}^{\hat{e}}}} + \sum_{n \in V^{\hat{e}} \setminus \{ \hat{n}^{\hat{e}}\}}  \|v_n\|_{X_n}$.

Now take $(\Lambda_{\alpha}w^k)_k $ with $(w^k)_k$ in $\domain(\Lambda_\alpha)$ to be a sequence in $\rg(\Lambda_\alpha)$ weak*-converging to some $y \in X_V  = \bigtimes_{n\in V} X_n$. Then, we note that by the definition of $R^{\hat{e}}_{\alpha^{\hat{e}}}$ and $\Lambda ^{\hat{e}}_{\alpha^{\hat{e}}}$ we have that
\[R^{\hat{e}}_{\alpha^{\hat{e}}} (\Theta_{\hat{e}} w^k_{\hat{e}}) \leq 
   \sum_{n \in V^{\hat{e}}} \|(\Lambda_\alpha w^k)_{n}\|_{X_n} < \infty.
\]
Further, defining $w^k_{\hat{E}}  = (w^k_{\hat{e}})_{\hat{e} \in \hat E}$ and, for $w_{\hat{E}}= (w_{\hat{e}})_{\hat{e} \in \hat{E}}$, $ \Theta_{\hat{E}} w_{\hat{E}} = (\Theta_{\hat{e}}w_{\hat{e}})_{\hat{e} \in \hat{E}}$,
\[K w_{\hat{E}}:= \sum_{\hat{e} \in \hat{E}} \alpha_{\hat{e}} \Phi_{\hat{e}}w_{\hat{e}},\quad \text{and}\quad F((v_{\hat{n}^{\hat{e}}})_{\hat{e}\in \hat{E}}) := \sum_{\hat{e} \in \hat{E}}R^{\hat{e}}_{\alpha^{\hat{e}}} ( v_{\hat{n}^{\hat{e}}}),\]
	we observe that both $(K w^k_{\hat{E}})_k$ and $((F\circ \Theta_{\hat{E}})(w^k_{\hat{E}}))_k$ are bounded. Hence, using Lemma \ref{lem:general_existence}, we can choose $(\tilde{w}^k_{\hat{E}})_k$ such that $w^k_{\hat{E}}-\tilde{w}^k_{\hat{E}} \in \ker(K) \cap \Theta_{\hat{E}}^{-1}(L_{\hat{E}})$ and both $(\tilde{w}^k_{\hat{E}})_k$ and $(\Theta_{\hat{E}} \tilde{w}^k_{\hat{E}})_k$ are bounded, where $L_{\hat{E}}$ is the invariant subspace of $F$ is given as $L_{\hat{E}} = \bigtimes_{\hat{e} \in \hat{E}} L_{\hat{e}} $ with $L_{\hat{e}}$ the invariant subspace of $R^{\hat{e}}_{\alpha^{\hat{e}}}$ and the respective projection onto $L_{\hat E}$ is given as $P_{L_{\hat E}}w_{\hat E} = (P_{L_{\hat e}}w_{\hat e})_{\hat e \in \hat E}$, with $P_{L_{\hat e}}$ the projection onto $L_{\hat e}$. Hence, up to taking a further non-relabeled subsequence, we can assume that both $(\tilde{w}^k_{\hat{E}})_k$ and $(\Theta_{\hat{E}} \tilde{w}^k_{\hat{E}})_k$ are weak* converging, such that, by weak* closedness of $\Theta_{\hat{E}}$, for $w_{\hat{E}}:=\text{\textsf{w*}-}\lim_k \tilde{w}^k_{\hat{E}}$ we have that $\Theta_{\hat{E}} w_{\hat{E}}=\text{\textsf{w*}-}\lim_k \Theta_{\hat{E}} \tilde{w}^k_{\hat{E}}$.

Now since the infimum in the definition of $R^{\hat{e}}_{\alpha^{\hat{e}}}$ as in \eqref{eq:duality_rl_vec_definition} is attained thanks to Corollary \ref{cor:existenceedges}, and since $R^{\hat{e}}_{\alpha^{\hat{e}}} ( \Theta_{\hat{e}}( w^k_{\hat{e}}-\tilde{w}^k_{\hat{e}}) ) = 0$, there exist  minimizers $z_{E^{\hat{e}}}^k \in \domain(\Lambda^{\hat{e}}_{\alpha^{\hat{e}}})$ such that $0 = \Theta_{\hat{e}}( \tilde{w}^k_{\hat{e}}-w^k_{\hat{e}}) - (\Lambda^{\hat{e}}_{\alpha^{\hat{e}}}z_{E^{\hat{e}}}^k)_{\hat{n}^{\hat{e}}}$ and $0 = (\Lambda^{\hat{e}}_{\alpha^{\hat{e}}}z_{E^{\hat{e}}}^k)_{n}$ for all $n \in V^{\hat{e}}\setminus\{\hat{n}^{\hat{e}}\}$.
Defining, with $w^k_{E^{\hat e}} = (w^k_{e})_{e \in E^{\hat{e}}}$, 
\[ 
\tilde{w}^k := (\tilde{w}^k_{\hat{E}} ,  (w^k_{E^{\hat{e}}} -  z_{E^{\hat{e}}}^k) _{\hat e \in \hat E})
\]
we observe that $(\Lambda_\alpha \tilde{w}^k)_{\hat{n}} =(\Lambda_\alpha w^k)_{\hat{n}} $ since $w^k_{\hat E} - \tilde w^k_{\hat E} \in \ker(K)$, and
\[ (\Lambda_\alpha \tilde{w} ^k)_{\hat n^{\hat{e}}} = \Theta_{\hat{e}} \tilde{w}_{\hat{e}}^k   - (\Lambda^{\hat{e}}_{\alpha^{\hat{e}}} (w^k_{E^{\hat{e}}} -  z_{E^{\hat{e}}}^k))_{\hat n^{\hat{e}}} = \Theta_{\hat{e}} w_{\hat{e}}^k   - (\Lambda^{\hat{e}}_{\alpha^{\hat{e}}} w^k_{E^{\hat{e}}}  )_{\hat n^{\hat{e}}} = (\Lambda_\alpha w^k)_{\hat n^{\hat{e}}},
\]
where in the first equality we used the definition of $\tilde w^k$ and in the second equality the fact that $0 = \Theta_{\hat{e}}( \tilde{w}^k_{\hat{e}}-w^k_{\hat{e}}) - (\Lambda^{\hat{e}}_{\alpha^{\hat{e}}}z_{E^{\hat{e}}}^k)_{\hat{n}^{\hat{e}}}$.
Also, since $0 = (\Lambda^{\hat{e}}_{\alpha^{\hat{e}}}z_{E^{\hat{e}}}^k)_{n}$ for all $n \in V^{\hat{e}}\setminus\{\hat{n}^{\hat{e}}\}$ we have
\[ (\Lambda_\alpha \tilde{w}^k)_{n} = (\Lambda^{\hat{e}}_{\alpha^{\hat{e}}} (w_{E^{\hat{e}}}^k - z_{E^{\hat{e}}}^k) )_{n} = (\Lambda^{\hat{e}}_{\alpha^{\hat{e}}} w_{E^{\hat{e}}}^k )_n = (\Lambda_\alpha w^k)_{n}
\]
for $n \in V^{\hat{e}} \setminus \{ \hat n^{\hat{e}}\}$, ${\hat{e}} \in \hat{E}$.

This implies that also $\Lambda_\alpha \tilde{w}^k \weakstar y$ in $X_V$ and, since $(\Theta_{\hat{E}} \tilde{w}^k_{\hat{E}})_k$ is weakly* convergent, that also $(\Lambda^{\hat{e}}_{\alpha^{\hat{e}}} \tilde w^k_{E^{\hat{e}}})_k$, with $ \tilde w^k_{E^{\hat{e}}} := w^k_{E^{\hat{e}}} -  z_{E^{\hat{e}}}^k$,
  is weakly* convergent for each $\hat{e} \in \hat{E}$.
 By induction hypothesis, there hence exist $w_{E^ {\hat{e}}} \in \domain (\Lambda^{\hat{e}}_{\alpha^{\hat{e}}} )$ such that $\text{\textsf{w*}-}\lim_k \Lambda^{\hat{e}}_{\alpha^{\hat{e}}}  \tilde w^k_{E^{\hat{e}}} = \Lambda^{\hat{e}}_{\alpha^{\hat{e}}}  w_{E^{\hat{e}}}$. Defining $w=\left(w_{\hat{E}} , (w_{E^ {\hat{e}}})_{\hat{e}\in \hat{E}}  \right) $ we finally see that $\Lambda_\alpha w = y$, since, from $\Lambda_\alpha \tilde{w}^k \weakstar y$ it follows that
 \[ (\Lambda_\alpha w)_{\hat{n}} = Kw_{\hat{E}} = \text{\textsf{w*}-}\lim_k K \tilde w^k_{\hat{E}} = y_{\hat{n}},
 \]
 \[(\Lambda_\alpha w)_{n^{\hat{e}}} = \Theta_{\hat{e}}w_{\hat{e}}  - (\Lambda^{\hat{e}}_{\alpha^{\hat{e}}} w^k_{E^{\hat{e}}} )_{\hat n^{\hat{e}}} = \text{\textsf{w*}-}\lim_k \left( \Theta_{\hat{e}}\tilde	w^k_{\hat{e}} - (\Lambda^{\hat{e}}_{\alpha^{\hat{e}}}  \tilde w^k_{E^{\hat{e}}})_{n^{\hat{e}}} \right) 
 =  y_{n^{\hat{e}}}
 \]
 for each $\hat{e}\in \hat{E}$, and
 \[(\Lambda_\alpha w)_{n} = (\Lambda^{\hat{e}}_{\alpha^{\hat{e}}} w_{E^{\hat{e}}} )_{n} = \text{\textsf{w*}-}\lim_k  (\Lambda^{\hat{e}}_{\alpha^{\hat{e}}}  \tilde w^k_{E^{\hat{e}}})_{n} = y_{n}
 \]
 for each $n \in V^ {\hat{e}} \setminus \{n^{\hat{e}}\}$, $\hat{e} \in \hat{E}$. This completes the proof.
\end{proof}

\subsection{Examples of predual regularization graph functionals} 
Here we provide predual regularization graph functionals for several examples introduced in Section \ref{sec:examples} by verifying the additional assumption \ref{ass:predual}. We represent such predual regularization graphs as in Figure \ref{fig:predualreggraphs}.  In this context, we denote by $\mI_{B_p}$ the indicator function of the $L^p$ unit ball for $p \in [1,\infty]$. Note that, for sake of clarity, in the root node of each predual regularization graph we write the corresponding functional $v_{\hat n} \mapsto \Psi^\#_{\hat{n}}(v_{\hat{n}}) - \langle u, v_{\hat{n}}\rangle$ and not just $ \Psi^\#_{\hat{n}}$. Moreover, nodes represented by an empty circle are associated with the zero functional.

\medskip

\noindent \textbf{Total variation.} Figure \ref{fig:regtvdual} shows a predual regularization graph for $\TV$. We refer to Section \ref{sec:examples} for the construction of the regularization graph realizing $\TV$. We remind also that $X_1 = L^p(\Omega)$ and $X_2 = \mathcal{M}(\Omega, \R^d)$ where $1< p \leq d'$ with $d'=d/(d-1)$ in case $d>1$ and $d'=\infty$ else.
Therefore the predual Banach spaces associated with the nodes are $X_1^\# = L^{p'}(\Omega)$ and $X_1^\# = C_0(\Omega, \R^d)$ where $p'$ satisfies $1/p + 1/p' =1$. 
Moreover, it is easy to see that the convex conjugate of $\mI_{B_\infty}$ on $C_0(\Omega,\R^d)$ is $\|\cdot \|_\M$, %
and the convex conjugate of the zero function is $\mI_{\{0\}}$.
To show \ref{ass:predual}, we further claim that  $\nabla : \BV(\Omega) \subset L^{d'}(\Omega) \rightarrow \mathcal{M}(\Omega, \R^d)$ is the adjoint of the negative divergence operator $-{\rm div} : \domain(-{\rm div}) \subset C_0(\Omega, \R^d) \rightarrow L^d(\Omega)$ defined in the weak sense as
\begin{align}\label{eq:weakdiv}
\langle -\div \varphi , \psi \rangle = \int_\Omega \varphi \cdot \nabla \psi\, {\rm d}x \quad \forall \psi \in C^1(\Omega)
\end{align}
with domain
\begin{equation}
\domain(-{\rm div}) = \{\varphi \in C_0(\Omega, \R^d) \st \div \varphi \in L^d(\Omega)\}.
\end{equation}
Note that from \eqref{eq:weakdiv}, using a simple density argument we obtain that 
 $\domain(-{\rm div})$ is densely defined and closed in $C_0(\Omega, \R^d)$.
To show that $\nabla$ is the adjoint of $-\div$ according to the definitions above, it is enough to observe that 
\begin{displaymath}
\domain(-\text{div}^*) = \left\{u \in L^{d'}(\Omega) \Bigst \exists c>0 \text{ such that } \int_\Omega u\div \varphi\, {\rm d}x \leq c \|\varphi\|_\infty \ \forall \varphi \in \domain(-\text{div})\right\}
\end{displaymath}
implying that $\domain(-\text{div}^*) = \BV(\Omega)$ and $(-{\rm div})^* = \nabla$.
The predual regularization graph functional $R_\alpha^\#  : L^p(\Omega) \rightarrow [0,+\infty]$ is given as
\begin{align*}
R_\alpha^\#(u) 
& = -\inf \left\{ - \langle u, v_1\rangle +  \mI_{B_\infty} (v_2)   \bigst v_1 \in L^{p'} (\Omega), \,v_2 \in \domain(-{\rm div}),\,  v_1 = -\div v_2\right\}\\
& = \sup \left\{\int_\Omega u\, \div v\, {\rm d}x  \Bigst v\in C_0(\Omega, \R^d), \,\div v \in L^{p'}(\Omega), \, \|v\|_\infty \leq 1\right\}.
\end{align*}
\medskip

\noindent \textbf{Infimal convolution of $\TV^{k_1}-\TV^{k_2}$.} Figure \ref{fig:reginfconvdual} shows a predual regularization graph corresponding to the infimal convolution of $\TV^{k_1}$ and $\TV^{k_2}$ with $k_1,k_2 \in \N$. We refer to Section \ref{sec:examples} for the construction of the regularization graph realizing the infimal convolution of $\TV^{k_1}-\TV^{k_2}$. We remind also that $p$ is chosen such that $1< p \leq \min\{d'_1,d'_2\}$, where $d_i'=d/(d-k_i)$ in case $d_i>k_i$ and $d_i'=\infty$ else.
Similarly to the $\TV$ predual regularization graph, the pre-adjoint of each linear operator $\nabla^{k_i} : L^{d_i'}(\Omega) \rightarrow \mathcal{M}(\Omega,\Sym^{k_i}(\R^d))$ can be seen to be the (possibly negative) higher-order divergence $(-1)^{k_i}\div^{k_i} : \domain((-1)^{k_i}{\rm div}^{k_i}) \subset C_0(\Omega,\Sym^{k_i}(\R^d)) \rightarrow L^{d_i}(\Omega)$ 
defined in the weak sense as
\begin{align}\label{eq:weakdivh}
\langle (-1)^{k_i}\div^{k_i} \varphi , \psi \rangle = (-1)^{k_i+1}\int_\Omega \varphi \cdot \nabla^{k_i} \psi\, {\rm d}x \quad \forall \psi \in C^{k_i}(\Omega)
\end{align}
with domain
\begin{equation}
\domain((-1)^{k_i}{\rm div}^{k_i}) = \{\varphi \in C_0(\Omega, \Sym^{k_i}(\R^d)) \st  \div^{k_i} \varphi \in L^{d_i}(\Omega)\}
\end{equation}
which is again closed and densely defined, showing \ref{ass:predual}.
The predual regularization graph functional $R_\alpha^\# : L^p(\Omega) \rightarrow [0,+\infty]$ is given as
\begin{align*}
R_\alpha^\#(u) &= -\inf \left\{- \langle u, v_1\rangle + \mI_{B_\infty}(v_2) + \mI_{B_\infty}(v_3)  \bigst v_1 \in L^{p'}(\Omega),\, v_2 \in \domain((-1)^{k_1}{\rm div}^{k_1}),\, \right.\\
&   \qquad \qquad \qquad v_3 \in \domain((-1)^{k_2}{\rm div}^{k_2}),\, 
 \left.v_1 = (-1)^{k_1}\div^{k_1} v_2,\,  \alpha v_1 = (-1)^{k_2}\div^{k_2} v_3 \right\}  \\
& = \sup \bigg\{\int_\Omega u \div^{k_1} u_1 \,{\rm d}x  \Bigst u_i\in C_0(\Omega, \Sym^{k_i}(\R^d)),\, \div^{k_i} u_i \in L^{p'}(\Omega)\\
& \qquad \qquad \qquad \qquad  \qquad  \qquad  \quad \div^{k_2} u_2 = \alpha \div^{k_1} u_1,\,  \|u_i\|_\infty \leq 1,\, i=1,2  \bigg\}.
\end{align*}

\medskip

\textbf{Total generalized variation of order $2$.}
Figure \ref{fig:regtgv2dual} shows a predual regularization graph for $\TGV_\alpha^2$. We refer to Section \ref{sec:examples} for the construction of the regularization graph realizing the total generalized variation of order $2$. We remind also that $p$ is chosen such that $1< p \leq d'$ where $d'=d/(d-1)$ in case $d>1$ and $d'=\infty$ else.
The pre-adjoint of $\nabla : \BV(\Omega) \subset L^{d'}(\Omega) \rightarrow \mathcal{M}(\Omega, \R^d)$ is given as in the $\TV$ example.
Moreover, a pre-adjoint of $\mathcal{E} : \BD(\Omega) \subset L^{d'}(\Omega, \R^d) \rightarrow \mathcal{M}(\Omega, \Sym^2(\R^d))$ is the negative divergence operator $-{\rm div} :\domain(-{\rm div})\subset  C_0(\Omega, \Sym^2(\R^d)) \rightarrow L^d(\Omega, \R^d)$ defined in the weak sense as in \eqref{eq:weakdiv} with domain
\begin{equation}
\domain(-{\rm div}) = \{\varphi \in C_0(\Omega, \Sym^2(\R^d)) \st  \div \varphi \in L^d(\Omega, \R^d)\},
\end{equation}
which is again densely defined and closed, showing \ref{ass:predual}. 
The predual regularization graph functional $R_\alpha^\# : L^p(\Omega) \rightarrow [0,+\infty]$ is for $\alpha>0$ given as
\begin{align*}
R_\alpha^\#(u) &= -\inf \left\{- \langle u, v_1\rangle +  \mI_{\{\|v\|_\infty \leq 1\}}(v_3) + \mI_{\{\|v\|_\infty \leq 1\}}(v_4)  \bigst v_1 \in L^{p'}(\Omega),\, v_2 \in \domain (-{\rm div}),\right.\\
& \left.\qquad \qquad v_3 \in C_0(\Omega, \R^d),\, v_4 \in \domain (-{\rm div}),\, v_1 = -\div v_2,\,  v_2 = v_3,\,  \alpha  v_2 = -\div v_4 \right\}  \\
& = \sup \bigg\{\int_\Omega u\,  {\rm div}^2 v\, {\rm d}x  \Bigst v\in C_0(\Omega, \Sym^2(\R^d)),\  \div v\in C_0(\Omega, \R^d),\, {\rm div}^2 v\in L^{p'}(\Omega, \R^d)\\
& \qquad \qquad  \|\div v\|_\infty \leq 1,\, \|\alpha v\|_\infty \leq 1 \bigg\}.
\end{align*}

\medskip

\noindent \textbf{$\TGV^2$-shearlet infimal convolution.}
Figure \ref{fig:shearletdual} shows a predual regularization graph for the $\TGV^2$-shearlet infimal convolution model.
We refer to Section \ref{sec:examples} for the construction of the regularization graph realizing for the $\TGV^2$-shearlet infimal convolution. We also remind that the exponent $p$ is chosen as $1< p \leq 2$.
Note that a predual of the extension to infinity of the $\ell^1$ norm is the indicator function of the unit ball of $c_0$, denoted by $\mathcal{I}_{c_0} : \ell^2(\Z^4) \rightarrow [0,+\infty]$. Thanks to the closedness of $c_0 \cap \ell^2$ in $\ell^2$, such indicator function is lower semicontinuous. 
Moreover, as $\mathcal{SH} : L^2(\R^2) \rightarrow \ell^2(\Z^4)$ defined according to \eqref{eq:sheardef} is a bounded operator between Hilbert spaces, its pre-adjoint exists and is bounded, showing \ref{ass:predual}. Further, it can be easily characterized for $v \in \ell^2(\Z^4)$ as
\begin{align*}
\mathcal{SH}^\# v = \sum_{j,k \in \Z , \,m \in \Z^2} v_{j,k,m} \Psi_{j,k,m}.
\end{align*}
Finally, noticing the the pre-adjoint of the restriction operator $r_\Omega$ is the extension to zero outside $\Omega$ (denoted by $r_\Omega^0$), the predual regularization graph functional $R_\alpha^\# : L^p(\Omega) \rightarrow [0,+\infty]$ is for $\alpha_0, \alpha_1>0$ given as
\begin{align*}
R_\alpha^\#(u) &= -\inf \left\{- \langle u, v_1\rangle +  \mI_{\{\|v\|_\infty \leq 1\}}(v_3) + \mI_{\{\|v\|_\infty \leq 1\}}(v_4) + \mathcal{I}_{c_0}(v_5) \bigst v_1 \in L^{p'}(\Omega), \right. \\
& \qquad \qquad\left. v_2 \in \domain (- {\rm div}),\, v_3 \in C_0(\Omega, \R^2),\, v_4 \in   \domain (- {\rm div}),\, v_5 \in \ell^2(\Z^4),\, v_1 = -\div v_2, \right.\\
& \left. \qquad \qquad \qquad \qquad\alpha_0 r_\Omega^0 v_1  = \mathcal{SH}^\# v_5  ,\,  v_2 = v_3,\,  \alpha_1  v_2 = -\div v_4 \right\}  \\
& = \sup \bigg\{\int_\Omega u\,  {\rm div}^2 v\, {\rm d}x  \Bigst v\in C_0(\Omega, \Sym(\R^2)),\  \div v\in C_0(\Omega, \R^2),\, {\rm div}^2 v\in L^{p'}(\Omega, \R^2), \\
& \qquad \qquad  z\in c_0(\Z^4),\, \|\div v\|_\infty \leq 1,\, \|\alpha_1 v\|_\infty \leq 1, \,\alpha_0 r_\Omega^0 {\rm div}^2 v  = \mathcal{SH}^\# z ,\, \|z\|_\infty \leq 1 \bigg\}.
\end{align*}

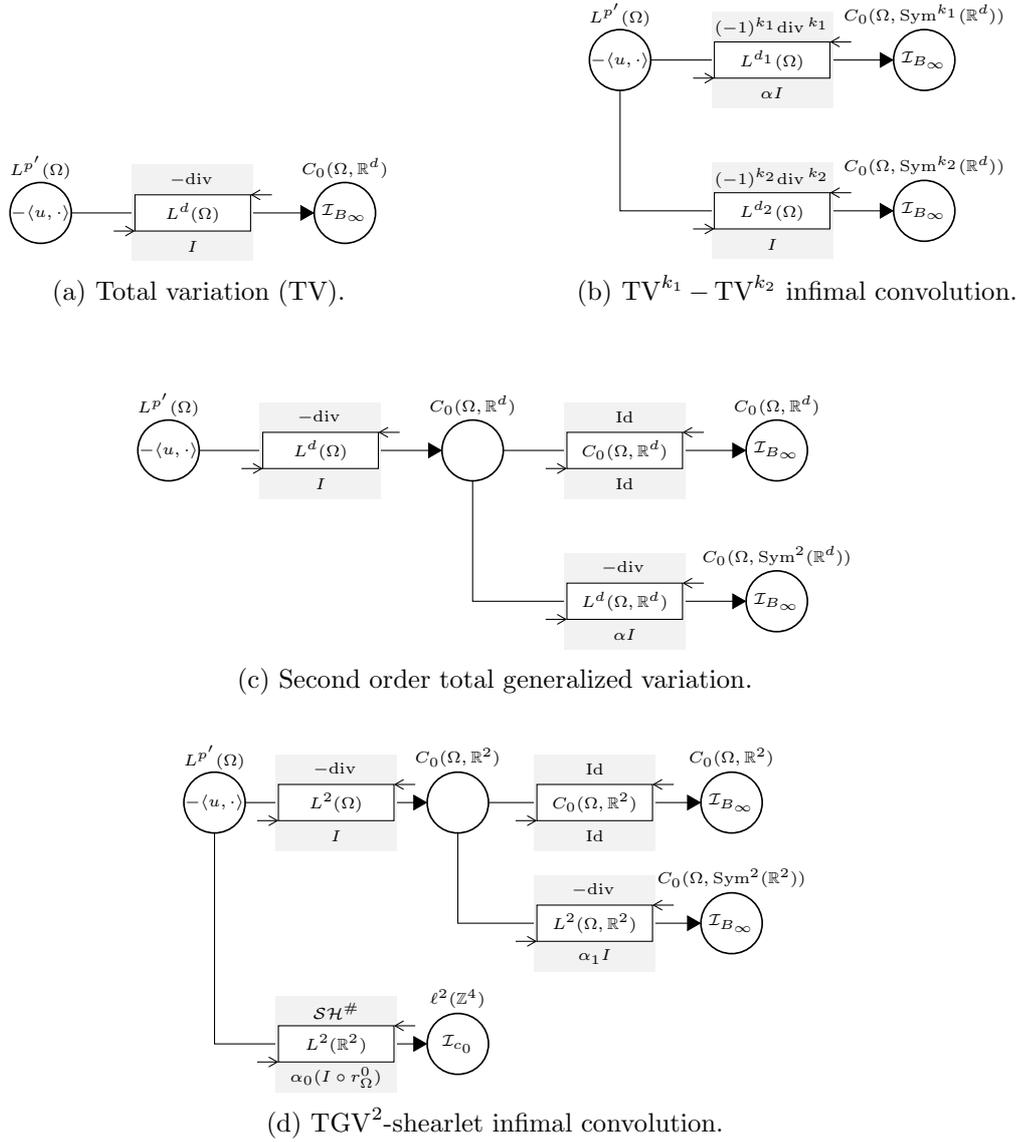
\begin{figure}
\begin{subfigure}[b]{0.5\textwidth}

\centering
\begin{tikzpicture}[predualrgraph]

\draw (0,0) pic (p1) {node={$L^{p'}(\Omega)$,$-\langle u\comma \cdot \rangle$}};
\draw (1,0) pic (p2) {node={$C_0(\Omega\comma \R^d)$,$\mI_{B_\infty}$}};

\draw[predge] (p1) -- pic{predge={$I$,$L^{d}(\Omega)$,$-\div$}} (p2);

\end{tikzpicture}
\subcaption{Total variation (TV).\label{fig:regtvdual}}
\end{subfigure}
\begin{subfigure}[b]{0.5\textwidth}
\centering
\begin{tikzpicture}[predualrgraph]

\draw (0,0) pic (p1) {node={$L^{p'}(\Omega)$,$-\langle u\comma \cdot \rangle$}};
\draw (1,0) pic (p2) {node={$C_0(\Omega\comma \Sym^{k_1}(\R^d))$,$\mI_{B_\infty}$}};

\draw (1,-1) pic (p3) {node={$C_0(\Omega\comma \Sym^{k_2}(\R^d))$,$\mI_{B_\infty}$}};

\draw[predge] (p1) -- pic{predge={$\alpha I$,$L^{d_1}(\Omega)$,$(-1)^{k_1}\div^{k_1}$}} (p2);
\draw[predge] (p1) -- (0,-1) --  pic{predge={$I$,$L^{d_2}(\Omega)$,$(-1)^{k_2}\div^{k_2}$}} (p3);

\end{tikzpicture}
\subcaption{$\TV^{k_1}-\TV^{k_2}$ infimal convolution.\label{fig:reginfconvdual}}
\end{subfigure}
\vspace{\baselineskip}

\begin{subfigure}[c]{\textwidth}
\centering
\begin{tikzpicture}[rgraph]

\draw (0,0) pic (p1) {node={$L^{p'}(\Omega)$,$-\langle u\comma \cdot \rangle$}};
\draw (1,0) pic (p2) {node={$C_0(\Omega\comma \R^d)$,}};

\draw (2,0) pic (p3) {node={$C_0(\Omega\comma \R^d)$,$\mI_{B_\infty}$}};
\draw (2,-1) pic (p4) {node={$C_0(\Omega\comma \Sym^2(\R^d))$,$\mI_{B_\infty}$}};

\draw[predge] (p1) -- pic{predge={$I$,$L^{d}(\Omega)$,$-\div$}} (p2);
\draw[predge] (p2) -- pic{predge={$\Id$,$C_0(\Omega\comma\R^d)$,$\Id$}} (p3);
\draw[predge] (p2) -- (1,-1) -- pic{predge={$\alpha I$,$L^{d}(\Omega\comma \R^d)$,$-\div$}} (p4);

\end{tikzpicture}
\subcaption{Second order total generalized variation.\label{fig:regtgv2dual}}
\end{subfigure}

\vspace{\baselineskip}

\begin{subfigure}[c]{\textwidth}
\centering

\begin{tikzpicture}[predualrgraph]

\draw (0,0) pic (p1) {node={$L^{p'}(\Omega)$,$-\langle u\comma \cdot \rangle$}};
\draw (0.8,0) pic (p2) {node={$C_0(\Omega\comma\R^2)$,}};

\draw (1.7,0) pic (p3) {node={$C_0(\Omega\comma\R^2)$,$\mI_{B_\infty}$}};
\draw (1.7,-0.8) pic (p4) {node={$C_0(\Omega\comma \Sym^2(\R^2))$,$\mI_{B_\infty}$}};

\draw[predge] (p1) -- pic{predge={$I$,$L^{2}(\Omega)$,$-\div$}} (p2);
\draw[predge] (p2) -- pic{predge={$\Id$,$C_0(\Omega\comma\R^2)$,$\Id$}} (p3);
\draw[predge] (p2) -- (0.8,-0.8) -- pic{predge={$\alpha_1 I$,$L^{2}(\Omega\comma\R^2)$,$-\div$}} (p4);

\draw (0.8,-1.6) pic (p3) {node={$\ell^{2}(\Z^4)$,$\mathcal{I}_{c_0}$}};

\draw[predge] (p1) --(0,-1.6) --  pic{predge={$\alpha_0 (I\circ r_\Omega^0)$,$L^{2}(\R^2)$,$\mathcal{S}\mathcal{H}^\#$}} (p3);
\end{tikzpicture}

\subcaption{$\TGV^2$-shearlet infimal convolution.\label{fig:shearletdual}}

\end{subfigure}

\caption{Examples of predual regularization graphs.}\label{fig:predualreggraphs}
\end{figure}

\section{Regularization of linear inverse problems}\label{sec:regularizationinverse}

\subsection{Setting and well-posedness}
We now consider the application of regularization graphs to the regularization of linear inverse problems. That is, with $K:X_{\hat{n}} \rightarrow Y$ a bounded linear operator (the forward model), $S_f:Y \rightarrow [0,\infty)$ a discrepancy functional associated with the data $f$ and $\beta>0$ a regularization parameter, we consider the minimization problem
 \begin{equation} \label{eq:linear_inverse_general}
 \min _{u \in X_{\hat{n}} } S_f(Ku) + \beta R_\alpha(u),
 \end{equation}
where $R_\alpha = R(\Gc_\alpha):X_{\hat{n}} \rightarrow [0,\infty]$ with $\Gc_\alpha$ a regularization graph with root node $\hat{n}$.
\begin{rem}[Forward operator with general domain $X$]
Note that considering only forward operators being defined on $X_{\hat{n}}$, where bounded sequences need to admit weak* convergent subsequences according to \ref{ass:weak_star_compactness}, is not a restriction compared to considering general operators $\tilde{K}:X \rightarrow Y$ with $X$ a Banach space such that $X_{\hat n} \hookrightarrow X$ and $R_\alpha$ being extended by $\infty$ to $X$ as in Proposition \ref{prop:reg_graph_extension_to_X}, since one can always recover this setting by choosing $K = \tilde{K} \circ \emb_{X_{\hat{n}},X}$, with $\emb_{X_{\hat{n}},X}$ the continuous embedding of $X_{\hat n} $ to $ X$.
\end{rem}

In order to study convergence in the data space for general discrepancies $S_f$, we introduce the following notion of convergence: We say the functionals $(S_{f^k})_k$ converge to $S_f$ if 
\begin{equation}
\left\{
\begin{aligned}
S_f(v) \leq \liminf_{k} S_{f^k}(v^k) \quad & \text{whenever } v^k \rightharpoonup v \text{ in } Y,\\
S_f(v) \geq \limsup_{k} S_{f^k}(v)  \quad & \text{for each } v \in  Y.\\
\end{aligned}
\right.
\end{equation}
Further, we say that $(S_{f^k})_k$ is \emph{equi-coercive} if there exists a coercive function $S_0:Y \rightarrow [0,\infty)$ such that $S_{f^k} \geq S_0$ in $Y$ for all $k$.
Note that this always holds true if $S_{f^k}(v):= \frac{1}{q}\|v-f^k\|_Y^q$ and $f^k \rightarrow f$ in $Y$, but the more general assumptions allow us to capture, for instance, also the situation when $S_f$ is the Kullback--Leibler divergence \cite[Example 2.16]{Holler19_ip_review_paper}.

Now, under weak assumptions, the previously established properties of $R_\alpha$ allow to obtain a standard well-posedness result for \eqref{eq:linear_inverse_general}. 

\begin{teo}  \label{thm:general_reg_existence_linear}
Let $R_\alpha = R(\Gc_\alpha) $ with $\Gc_\alpha$ being a regularization graph with weights $\alpha$ and root node $\hat{n}$ such that $X_{\hat{n}}$ is reflexive, $\beta >0$, and let $Y$ be a Banach space, $K: X_{\hat{n}}\to Y$ be linear and continuous and
  $S_f: Y \to {[{0,\infty}]}$ be a proper, convex, weakly lower semi-continuous and coercive discrepancy functional. 
  Then, the Tikhonov minimization problem \eqref{eq:linear_inverse_general}
  is well-posed, i.e., there exists a solution and the solution mapping is stable in sense
  that, if $S_{f^k}$ converges to $S_f$ and $(S_{f^k})_k$ is
  equi-coercive, then for each sequence of minimizers $(u^k)_k$
  of~\eqref{eq:linear_inverse_general} with discrepancy $S_{f^k}$,
  \begin{itemize}
  \item[i)]
    either $S_{f^k}(Ku^k) + \beta R_\alpha (u^k) \to \infty$ as
    $k \to \infty$ and~\eqref{eq:linear_inverse_general} with discrepancy
    $S_f$ does not admit a finite solution,
  \item[ii)] or
    $S_{f^k}(Ku^k) + \beta R_\alpha (u^k) \to \min_{u \in
      X_{\hat n}} S_f(Ku) + \beta R_\alpha (u)$
    as $k \to \infty$ and there exists, possibly up to shifts by functions in $\ker(K) \cap L$, with $L$ the invariant subspace of $R_\alpha$, a
    weak accumulation point $u \in X_{\hat n}$ of $(u^k)_k$ that
    minimizes~\eqref{eq:linear_inverse_general} with discrepancy $S_f$.
  \end{itemize}  
  Further, in case \eqref{eq:linear_inverse_general} with discrepancy $S_f$ admits a 
  finite solution, for each subsequence $(u^{k_i})_i$ weakly converging to some $u \in X_{\hat{n}}$,
  it holds that $R_\alpha (u^{k_i}) \to R_\alpha (u)$ as $i \to \infty$.
  Also, if $S_f$ is strictly convex and $K$ is injective, finite solutions $u$ of \eqref{eq:linear_inverse_general}
  are unique and $u^k \rightharpoonup u$ in $X_{\hat n}$.

\end{teo}
\begin{proof}
Existence follows by the application of the direct method of calculus of variations in $X_{\hat n}$. More precisely, given a minimizing sequence $(u^k)_k$ for \eqref{eq:linear_inverse_general} we can apply Lemma \ref{lem:general_existence} with $W =X= X_{\hat{n}}$,
$F = R_\alpha$, $\Theta = \Id$ and $L$ being the finite dimensional invariant space of $R_\alpha$ provided by Theorem \ref{thm:graphessentiallycoercive}, to obtain the existence of another minimizing sequence $(\tilde u^k)_k$  for \eqref{eq:linear_inverse_general} that is bounded in $X_{\hat n}$. Note that the assumptions of Lemma \ref{lem:general_existence} are fulfilled since the weak lower semi-continuity of $R_\alpha$ (which is equivalent to weak* lower semi-continuity of $R_\alpha$ by reflexivity of $X_{\hat n}$) and Assumption $i)$ of Lemma \ref{lem:general_existence} hold as a consequence of Theorem \ref{thm:graphessentiallycoercive}, and the boundedness of $(K u^k)_k$ follows from the coercivity of $S_f$.
Therefore, thanks to the weak lower semi-continuity of $R_\alpha$ and the boundedness of $K$ we can apply the direct method to the sequence $(\tilde u^k)_k$ and conclude existence of minimizers for \eqref{eq:linear_inverse_general}.

The claimed stability follows with standard arguments. 
For instance, it can be proven adapting straightforwardly \cite[Theorem 2.14]{Holler19_ip_review_paper}.
\end{proof}
\begin{rem} \label{rem:well_posed_non_reflexive}
Note that the results of Theorem \ref{thm:general_reg_existence_linear} 
can also be modified to hold without assuming reflexivity of $X_{\hat{n}}$ but assuming, for instance, that $K$ is weak*-to-weak continuous. Indeed, in this setting, Lemma \ref{lem:general_existence} applies the same way and existence follows from the coercivity statement of Lemma \ref{lem:general_existence} using weak*-to-weak continuity of $K$ and weak lower semi-continuity of $S_f$. Likewise, also the claimed stability can be shown by straightfoward adaptions.
\end{rem}

\subsection{Convergence and stability for varying parameters}
In this section we study the stability of solutions of \eqref{eq:linear_inverse_general} for varying parameters $\alpha$ and for vanishing noise. To this aim, we first define a variant of regularization graphs for vanishing weights.

\begin{dfnz} \label{def:extended_graph_weight_zero}
For $\Gc_\alpha$ a regularization graph with underlying graph $G = (V,E)$ and $\alpha \in [0,\infty)^{E}$ weights, define $\tilde{\alpha} = (\tilde{\alpha}_e)_{e \in E}$ as $\tilde{\alpha}_e = \alpha_e$ for $\alpha_e >0$ and $\tilde{\alpha}_e = 1$ else. Further, for $n \in V$, set $\tilde \Psi_n = \mI_{\{0\}}$ in case that $n \neq \hat n$ and $\alpha_{(n^-,n)} = 0$,
 and $\tilde{\Psi}_n = \Psi_n $ else. 
Then, we define $\hat{\Gc}_\alpha = \tilde{\Gc}_{\tilde{\alpha}}$ with $\tilde{\Gc} = (G, (\tilde{\Psi}_n)_{n\in V}, (\Theta_e)_{e\in E}, (\Phi_e)_{e\in E})$, that is, $\hat{\Gc}_\alpha$ is the regularization graph $\Gc_\alpha$ with zero weights being replaced by $1$ and node functionals associated with tails of zero-weight edges being replaced by $\mI_{\{0\}}$. Finally, we set $\hat{R}_\alpha = R(\hat{\Gc}_\alpha)$, being the regularization graph functional associated to $\hat{\Gc}_\alpha$.
\end{dfnz}
This definition is required to deal with lower semi-continuity with respect to weights $(\alpha_e)_e$ converging to zero. An example of $\hat{R}_\alpha$ in case $R_\alpha(u)=R(\Gc_\alpha)(u) = \min_{ w \in \BD(\Omega,\R^d)} \| \nabla u - \alpha_0  w \|_\M + \|\symgrad w\|_\M$ and $\alpha_0 = 0$ is given as
\[ \hat{R}_\alpha(u) = \min_{ w \in \BD(\Omega,\R^d)} \| \nabla u - w \|_\M \quad \text{s.t. } \symgrad w = 0.
\]
It is easy to see that for any regularization graph $\Gc_\alpha$ and any choice of weights $\alpha$, $\hat{\Gc}_\alpha$ is again a regularization graph such that all previous results apply. Moreover, the following lemma holds.
\begin{lemma} \label{lem:rhat_r_estimate}
Let $\Gc_\alpha $ be a regularization graph with root node $\hat{n}$ and weights $\alpha \in [0,\infty)^{E}$, $R_\alpha = R(\Gc_\alpha)$ and $\hat R_\alpha = R(\hat{\Gc}_\alpha)$. Then for every $u \in X_{\hat n}$ we have that $ \hat{R}_{\alpha}(u)\leq R_\alpha(u)$. 
\end{lemma}
\begin{proof}
Arguing by induction we suppose that the claimed assertion holds for any regularization graph of height less than $h$ and we assume that the height of $\Gc_\alpha$ is $h$. For $h=0$, the result holds since $\hat R_{\alpha} = R_{\alpha}$, so we assume $h\geq 1$.
Using the recursive representation of $R_{\alpha}$ and the notation from Lemma \ref{lem:primal_graph_recursion} it holds that
\begin{equation}  \label{eq:recursive_rep_tmp_proof}
R_{\alpha}(u) =  \inf \Bigg\{\Psi_{\hat{n}} \left(u-  \sum_{\hat{e} \in \hat{E}: \, \alpha_{\hat{e}}>0} \alpha_{\hat{e}} \Phi_{\hat{e}} w_{\hat{e}} \right) + \sum_{\hat{e}\in \hat{E}}  R^{\hat{e}}_{\alpha^{\hat{e}}}(\Theta_{\hat{e}}w_{\hat{e}})  \Biggst w_{{\hat{e}}} \in \domain(\Theta_{{\hat{e}}}) \text{ for all } {\hat{e}} \in \hat{E} \Bigg\}
 \end{equation} 
where, for $\hat{e}\in \hat{E}$, $R^{\hat{e}}_{\alpha^{\hat{e}}} = R(\Gc_{\alpha^{\hat{e}}}^{\hat{e}})$ and $\Gc^{\hat{e}}_{\alpha^{\hat{e}}} = (\Gc^{\hat{e}},\alpha^{\hat{e}})$, with $\Gc^{\hat{e}} = (G^{\hat{e}}, (\Psi_n)_{n\in V^{\hat{e}}}, (\Theta_e)_{e\in E^{\hat{e}}},$ $(\Phi_e)_{e\in E^{\hat{e}}})$ and $\alpha^{\hat{e}} = (\alpha_e)_{e\in E^{\hat{e}}}$, is a regularization graph of height at most $h-1$ with associated graph $G^{\hat{e}} = (V^{\hat{e}},E^{\hat{e}})$ that is a subtree of $G=(V,E)$ with $\hat{n}^{\hat{e}}$ as root node. 
Similarly, using Lemma \ref{lem:primal_graph_recursion}, we write $\hat{R}_{\alpha}(u)$ as
\begin{align}\label{eq:ineqp}
\hat{R}_{\alpha}(u) &  =   \inf \Bigg\{ \Psi_{\hat{n}} \bigg(u-  \sum_{ {\hat{e}} \in \hat{E}: \alpha_{{\hat{e}}} >0}  \alpha_{\hat{e}} \Phi_{\hat{e}}  w_{\hat{e}} -  \sum_{ {\hat{e}} \in \hat{E} : \alpha_{{\hat{e}}} =0}  \Phi_{\hat{e}}  w_{\hat{e}} \bigg) + \sum_{\hat{e} \in \hat{E}: \alpha_{\hat{e}}=0}  \bar{R}^{\hat{e}}_{\tilde \alpha^{\hat{e}}}(\Theta_{\hat{e}}  w_{\hat{e}}) \nonumber\\
& \qquad \qquad + \sum_{\hat{e} \in \hat{E}: \alpha_{\hat{e}}>0}  \bar{R}^{\hat{e}}_{\tilde \alpha^{\hat{e}}}(\Theta_{\hat{e}}  w_{\hat{e}}) \Biggst w_{{\hat{e}}} \in \domain(\Theta_{{\hat{e}}}) \text{ for all } {\hat{e}} \in \hat{E} \Bigg\},
\end{align}
where, for $\hat{e}\in \hat{E}$, $\bar{R}^{\hat{e}}_{\tilde \alpha^{\hat{e}}} = R(\bar \Gc_{\tilde \alpha^{\hat{e}}}^{\hat{e}})$ and $\bar \Gc_{\tilde \alpha^{\hat{e}}}^{\hat{e}} = (\bar \Gc^{\hat{e}},{\tilde \alpha^{\hat{e}}})$, with $\bar \Gc^{\hat{e}} =  (G^{\hat{e}}, (\tilde \Psi_n)_{n\in V^{\hat{e}}}, (\Theta_e)_{e\in E^{\hat{e}}},$ $ (\Phi_e)_{e\in E^{\hat{e}}})$ and $(\tilde \alpha_e)_{e\in E^{\hat{e}}}$ according to Definition \ref{def:extended_graph_weight_zero}, is a regularization graph of height at most $h-1$ with associated graphs $G^{\hat{e}} = (V^{\hat{e}},E^{\hat{e}})$.
Note that  $\tilde \Psi_{\hat n^{\hat e}} =  \Psi_{\hat n^{\hat e}}$  for every $\hat e \in \hat E$ such that  $\alpha_{\hat{e}}>0$. Therefore, by the way the modified regularization graph $\hat \Gc_{\alpha^{\hat{e}}}^{\hat{e}}$ is obtained from $\Gc_{\alpha^{\hat{e}}}^{\hat{e}}$ according to Definition \ref{def:extended_graph_weight_zero}, it follows that, for $\hat e \in \hat E$ with $\alpha_{\hat{e}}>0$, $\bar{R}^{\hat{e}}_{\tilde \alpha^{\hat{e}}} = \hat{R}^{\hat{e}}_{\alpha^{\hat{e}}}$ with $\hat{R}^{\hat{e}}_{\alpha^{\hat{e}}} = R(\hat \Gc_{\alpha^{\hat{e}}}^{\hat{e}})$.

Thus, choosing $w_{\hat e} = 0$ for every $\hat e \in \hat E$ such that  $\alpha_{\hat{e}}=0$ in  \eqref{eq:ineqp} and using the induction hypothesis as well as \eqref{eq:recursive_rep_tmp_proof}, we estimate for every $u \in X_{\hat n}$ as follows
\begin{multline*}
\hat{R}_{\alpha}(u)   \leq \inf \Bigg\{ \Psi_{\hat{n}} \bigg(u-  \sum_{ {\hat{e}} \in \hat{E}: \alpha_{{\hat{e}}} >0}  \alpha_{\hat{e}} \Phi_{\hat{e}}  w_{\hat{e}}\bigg) + \sum_{\hat{e} \in \hat{E}: \alpha_{\hat{e}}>0}  \hat{R}^{\hat{e}}_{\alpha^{\hat{e}}}(\Theta_{\hat{e}}  w_{\hat{e}}) \Biggst w_{{\hat{e}}} \in \domain(\Theta_{{\hat{e}}}),\\
 \qquad \qquad  \forall {\hat{e}} \in \hat{E} \text{ such that } \alpha_{\hat{e}}>0\Bigg\}
 \leq R_\alpha(u).
\end{multline*}
\end{proof}

We now prove a weak* lower semi-continuity result for regularization graph functionals with respect to the parameters. 
\begin{teo}\label{thm:l.s.c}
Let $\Gc_\alpha $ be a regularization graph with root node $\hat{n}$ and weights $\alpha \in [0,\infty)^{E}$, $R_\alpha = R(\Gc_\alpha)$, and $(\alpha^k)_k $ be a sequence of weights in $(0,\infty)^{E}$ such that $(\alpha^k)_k \rightarrow \alpha$.

Then, for every sequence $(u^k)_k$ in  $X_{\hat{n}}$ such that $u^k \weakstar u\in X_{\hat{n}}$
it holds that
\begin{equation}\label{eq:lscweights_extended_r}
\hat{R}_{\alpha}(u) \leq \liminf_{k\rightarrow \infty} R_{\alpha^k}(u^k) .
\end{equation}

Moreover, for $u \in X_{\hat{n}}$ and
\begin{align}\label{eq:gamma_definition}
\gamma_k & :=  \min \Big \{ \prod_{e \in F} \alpha^k_e/\alpha_e \Big |  F \subset E \text{ is either empty or} \\
&  \quad \quad \quad \quad \quad \quad \text{a chain with root }\hat{n} \in V \text{ and } \alpha_e >0 \ \forall e \in F \Big \},
\end{align}
using again the conventions that for $F=\emptyset$, $ \prod_{e \in \emptyset} \frac{\alpha_e^2}{\alpha_e^1} = 1$, we have that $\gamma_k \leq 1$, $\gamma_k \rightarrow 1$, 
\begin{equation}\label{eq:convergenceweights}
R_{\alpha^k}(\gamma_k u) \leq  \hat{R}_{\alpha}(u) \ \ \text{for all } k\in \N \quad \text{and}\quad 
\lim_{k\rightarrow \infty} R_{\alpha^k}(\gamma_k u) = \hat{R}_{\alpha}(u).
\end{equation}

\end{teo}
\begin{rem} \label{rem:parameter_convergence_pos_hom}
Note that, in case each node functional $\Psi_n$ is positively one homogeneous (such that $R_{\alpha^k}$ is positively one homogeneous according to Proposition \ref{prop:convexineq}), the convergence of \eqref{eq:convergenceweights} implies that $\lim_{k\rightarrow \infty} R_{\alpha^k}(u) = \hat{R}_\alpha (u)$. Also, in case $\alpha_e>0$ for each $e\in E$, $\hat{R}_\alpha$ can be replaced by $R_\alpha$.
\end{rem}
\begin{proof}[Proof of Theorem \ref{thm:l.s.c}]
We argue again by induction and, supposing that the claimed assertions hold for any regularization graph of height less than $h$, assume that the height of $\Gc_\alpha$ is $h$. Again, for $h=0$, the result holds trivially, so we assume $h\geq 1$.

We first deal with lower semi-continuity of $R_\alpha = R(\Gc_\alpha)$, for which, up to taking a non-relabeled subsequence, we assume that $\liminf_{k\rightarrow \infty} R_{\alpha^k}(u^k) = \lim_{k\rightarrow \infty} R_{\alpha^k}(u^k)  <+\infty$. 
Using the recursive representation of $R_{\alpha^k}$ and the notation from Lemma \ref{lem:primal_graph_recursion}, and the result of Theorem \ref{thm:graphessentiallycoercive} we can select a sequence $(w^k)_k$ in $\bigtimes_{\hat{e}\in \hat{E}} \domain(\Theta_{\hat{e}})$ such that
\begin{equation}  \label{eq:recursive_rep_exact_parameter_convergence_proof}
R_{\alpha^k}(u^k) =  \Psi_{\hat{n}} \left(u^k-  \sum_{\hat{e} \in \hat{E}} \alpha^k_{\hat{e}} \Phi_{\hat{e}} w^k_{\hat{e}} \right) + \sum_{\hat{e}\in \hat{E}}  R^{\hat{e}}_{(\alpha^k)^{\hat{e}}}(\Theta_{\hat{e}}w^k_{\hat{e}}),
 \end{equation} 
 with $R^{\hat{e}}_{(\alpha^k)^{\hat{e}}} = R(\Gc_{(\alpha^k)^{\hat{e}}}^{\hat{e}})$ and $\Gc_{(\alpha^k)^{\hat{e}}}^{\hat{e}}$ being regularization graphs of height at most $h-1$ with graph structure $G^{\hat{e}} = (V^{\hat{e}},E^{\hat{e}})$ and root node $\hat{n}^{\hat{e}}$. %
 Thanks to Proposition \ref{prop:coercivity_dependence_on_weights} and the weights $\alpha^k_e$ being positive, the invariant subspace $L^{\hat{e}}$ of $R^{\hat{e}}_{(\alpha^k)^{\hat{e}}}$ does not depend on $k$ and for 
\begin{equation}\label{eq:ccc}
 C_{\hat{e},\alpha^k}:= \max \big \{ \prod_{e \in F} \alpha^k_e \st F \subset E^{\hat{e}} \text{ is either empty or a chain with root }\hat{n}^{\hat{e}}  \big \},
\end{equation}
it follows that $(C_{\hat{e},\alpha^k})_k$ is bounded and that
\begin{equation}\label{eq:coercivityinductionnoise}
\|u- P_{\rg(\Theta_{\hat{e}}) \cap L^{\hat{e}}}  u\|_{X_{\hat{n}^{\hat{e}}}} \leq C_{{\hat{e}},\alpha^k} C_{\hat{e}}  R^{\hat{e}}_{(\alpha^k)^{\hat{e}}}(u)+ D_{\hat{e}} \quad \forall u \in X_{\hat{n}^{\hat{e}}} 
\end{equation}
with $C_{\hat{e}}, D_{\hat{e}}$ independent of $k$.
Further, remember that each $\Theta_{\hat e}$ satisfies 
 \[ \| w - P_{\ker(\Theta_{{\hat{e}}})}w \|_{X_{{\hat{e}}}} \leq B_{\hat{e}} \|\Theta_{{\hat{e}}} w\|_{X_{\hat{n}^{\hat{e}}}} \quad \forall w\in \domain(\Theta_{{\hat{e}}}),
\]
with $B_{\hat{e}} >0$, thanks to Assumption \ref{ass:coerc_fwd_operator}. Now defining $M^{\hat{e}}\subset \domain(\Theta_{\hat{e}})$ and $P^{\hat{e}}: \domain(\Theta_{\hat{e}}) \rightarrow M^{\hat{e}}$ as in \eqref{eq:kernel_Rl_definition} and \eqref{eq:projection_kernel_Rl_definition}, respectively, one sees that they also do not depend on $k$ and, estimating as in \eqref{eq:coercivitychosen}, we obtain
\begin{equation} \label{eq:theta_rl_coercivity_parameter_convergence_proof}
 \|w - P^{\hat{e}} w \|_{X_{\hat{e}}} \leq B_{\hat{e}} C_{\hat{e}}  C_{{\hat{e}},\alpha^k} R^{\hat{e}}_{(\alpha^k)^{\hat{e}}} (\Theta_{\hat{e}} w)  +  B_{\hat{e}} D_{\hat{e}}.
 \end{equation}
Further, define  $K :W \rightarrow X_{\hat{n}}$ with $W :=\bigtimes_{\hat{e}\in \hat{E}} X_{\hat{e}}$ as  
\begin{equation}
Kw = \sum_{\hat{e} \in \hat{E}} \Phi_{\hat{e}} w_{\hat{e}}
\end{equation}
and $P_M w:= (P^{\hat{e}}w_{\hat{e}})_{\hat{e}}$ for $w \in W$. Note that for 
$M = \bigtimes_{\hat{e} \in \hat{E}}M^{\hat{e}}$, $P_M : W \rightarrow M$ is a projection.
Choose $P_{\ker(K) \cap M} : M \rightarrow \ker(K) \cap M$ to be a projection onto $\ker(K) \cap M$ and let $P_Z : M  \rightarrow Z$ be defined as $P_Z = \Id - P_{\ker(K) \cap M}$, where $Z = \rg(\Id - P_{\ker(K) \cap M})$. Note that $P_Z$ is a projection onto $Z$. Then, defining $\alpha_{\hat{E}} w := (\alpha_{\hat{e}} w_{\hat{e}})_{\hat{e}\in \hat{E}}$ for $w \in W$ and $\alpha_{\hat{E}} \in [0,\infty)^{\hat{E}}$, we can observe that, with $\alpha^k_{\hat{E}} = (\alpha^k_{\hat{e}})_{\hat{e} \in \hat{E}}$,
\begin{equation}\label{eq:projectioncompl}
 \tilde{w}^k := w^k - (P_Mw^k - (1/\alpha_{\hat{E}}^k) P_{Z} \alpha_{\hat{E}}^k P_M w^k)
\end{equation}
also realizes the minimum in
 \eqref{eq:recursive_rep_exact_parameter_convergence_proof}, since 
 \[w^k -  \tilde{w}^k = P_Mw^k -(1/\alpha_{\hat{E}}^k) P_{Z} \alpha_{\hat{E}}^k P_M w^k = (1/\alpha^k_{\hat{E}}) (\Id-P_{Z}) \alpha_{\hat{E}}^k P_Mw^k \] 
 with $(\Id-P_{Z}) \alpha^k_{\hat{E}} P_Mw^k \in \ker(K) \cap M$, such that
 $\sum_{\hat e \in \hat E} \alpha^k_{\hat e} \Phi_{\hat e} (w^k - \tilde w^k) = K(\Id - P_Z)\alpha_{\hat{E}}^k P_M w^k =0$ and $\Theta_{\hat e}(w^k - \tilde w^k)_{\hat e} = (1/\alpha_{\hat{e}}^k)\Theta_{\hat e}[(\Id-P_{Z}) \alpha_{\hat{E}}^k P_Mw^k]_{\hat e} = 0$. By the estimate \eqref{eq:theta_rl_coercivity_parameter_convergence_proof} we obtain for some constants $C,D, \tilde D >0$ that
\begin{equation}\label{eq:projectionesti}
 \| w^k - P_M w^k \|_W \leq C \left(\max_{\hat e \in \hat E}C_{\hat e, \alpha^k} \right) \sum_{\hat{e} \in \hat{E}} R_{(\alpha^k)^{\hat{e}}}^{\hat{e}} (\Theta_{\hat{e}} w^k_{\hat{e}}) + D \leq \tilde D < \infty
\end{equation}
where the constant $\tilde D$ does not depend on $k$ thanks to the boundedness of $R_{\alpha^k}(u^k)$, the recursive formula \eqref{eq:recursive_rep_exact_parameter_convergence_proof}, and the boundedness of  $(C_{\hat{e},\alpha^k})_k$. 

Since $K$ is injective and bounded (see Remark \ref{rem:initialremark}) on the finite dimensional space $Z$, there exists $C>0$ independent from $k$ such that $\|z\|_W \leq C \|K z\|_{X_{\hat{n}}}$ for all $z \in Z$. Thus we can estimate by coercivity of $\Psi_{\hat{n}}$ and using $P_{Z} \alpha_{\hat{E}}^k P_M w^k = \alpha_{\hat{E}}^k(\tilde w^k - (w^k - P_M w^k))$ that for generic constants $\tilde C, D$ (and $\tilde D$ as in \eqref{eq:projectionesti}) we have
\begin{align}\label{eq:projfinal}
\| P_{Z} \alpha^k_{\hat{E}} P_M w^k\|_W & \leq C \|K P_{Z} \alpha^k_{\hat{E}} P_M w^k\|_{X_{\hat{n}}} \nonumber\\ 
&  \leq C \|K \alpha^k_{\hat{E}} \tilde w^k \|_{X_{\hat{n}}} + \tilde C \| \alpha^k_{\hat{E}} w^k - \alpha^k_{\hat{E}} P_Mw^k\|_{W} \nonumber\\
& \leq C  \|K\alpha^k_{\hat{E}} \tilde w^k - u^k \|_{X_{\hat{n}}} + C\|u^k\|_{X_{\hat{n}}} +  \tilde C (\max_{\hat{e}\in \hat{E}} 
\alpha^k_{\hat{e}})\|w^k - P_Mw^k\|_{W}  \nonumber\\ 
& \leq \tilde C \Psi_{\hat n}(u^k - K\alpha_{\hat{E}}^k \tilde w^k)  + C\|u^k\|_{X_{\hat{n}}}  + \tilde C (\max_{\hat{e}\in \hat{E}} 
\alpha^k_{\hat{e}})\tilde D\nonumber \\ 
& \leq \tilde C R_{\alpha^k}(u^k) + D < \infty,
\end{align}
 where also used \eqref{eq:projectionesti}, the fact that $(u^k)_k$ is uniformly bounded as it is weak* converging and that $\tilde w^k$ realizes the minimum in \eqref{eq:recursive_rep_exact_parameter_convergence_proof}.

Now for $\hat{e} \in \hat{E}$ with $\alpha_{\hat{e}} >0$ this together with \eqref{eq:projectioncompl} and \eqref{eq:projectionesti} implies that $(\tilde{w}_{\hat{e}}^k)_k$ is bounded, hence admits a (non-relabeled) subsequence weak* converging to some $\tilde w_{\hat{e}} \in X_{\hat{e}}$ by \ref{ass:weak_star_compactness}.
Moreover, using \eqref{eq:recursive_rep_exact_parameter_convergence_proof}, \eqref{eq:coercivityinductionnoise}, \eqref{eq:projfinal}, the boundedness of $R_{\alpha^k}(u^k)$ and the  finite dimensionality of $Z$  we have for $\hat e \in \hat E$ that
\begin{align*}
 \|\Theta_{\hat{e}} \tilde w^k_{\hat{e}} \| _{X_{\hat{n}^{\hat{e}}}} & \leq  \|\Theta_{\hat{e}}  w^k_{\hat e} - \Theta_{\hat e} P^{\hat e} w_{\hat e}^k\|_{X_{\hat{n}^ {\hat e}}} + \|\Theta_{\hat{e}} (1/\alpha_{\hat{e}}^k) (P_{Z} \alpha^k_{\hat{E}} P_M w^k)_{\hat{e}} \|_{X_{\hat{n}^{\hat{e}}}}\\
& = \|\Theta_{\hat{e}} w^k_{\hat{e}} - P_{\rg(\Theta_{\hat{e}}) \cap L^{\hat e}}\Theta_{\hat{e}}w_{\hat{e}}^k\|_{X_{\hat{n}^{\hat{e}}}} + \|\Theta_{\hat{e}} (1/\alpha_{\hat{e}}^k) (P_{Z} \alpha^k P_M w^k)_{\hat{e}} \|_{X_{\hat{n}^{\hat{e}}}} \\
& \leq C_{{\hat{e}},\alpha^k} C_{\hat{e}}  R^{\hat{e}}_{(\alpha^k)^{\hat{e}}}(\Theta_{\hat{e}}  w^k_{\hat{e}})  + D_{\hat{e}} +C \| (P_{Z} \alpha^k P_M w^k)_{\hat{e}} \|_{X_{\hat{n}^{\hat{e}}}} \\
& \leq \tilde C < +\infty,
\end{align*}
where the constant $\tilde C$ is independent of $k$ and we use the definition of $P^{\hat e}$ in \eqref{eq:projection_kernel_Rl_definition}.
Hence, by weak* sequential compactness of the $X_{\hat{n}^{\hat{e}}}$ and weak*-closedness of $\Theta_{\hat{e}}$ we obtain $\tilde w_{\hat{e}} \in \domain(\Theta_{\hat{e}})$ and, up to taking a further non-relabeled subsequence, $\text{\textsf{w*}-}\lim_{k \rightarrow +\infty}  \Theta \tilde w_{\hat{e}}^k = \Theta \tilde w_{\hat{e}}$. 

Further, for ${\hat{e}} \in \hat{E}$ with $\alpha_{\hat{e}} = 0$, we see from \eqref{eq:projectioncompl}, \eqref{eq:projectionesti} and \eqref{eq:projfinal} that $(\alpha^k_{\hat{e}}  \tilde{w}^k_{\hat{e}})_k$ is bounded. Hence, up to taking a further subsequence, we can assume that 
\begin{equation}\label{eq:staykernel}
\text{\textsf{w*}-}\lim_{k \rightarrow \infty} \alpha^k_{\hat{e}}  \tilde{w}^k_{\hat{e}} = \text{\textsf{w*}-}\lim_{k \rightarrow \infty} (P_Z \alpha^k_{\hat{E}} P_M w^k)_{\hat{e}} = z_{\hat e} \in M^{\hat{e}}
\end{equation} 
by \eqref{eq:projectionesti} since $\alpha^k_{\hat{e}} \rightarrow 0$ as $k \rightarrow +\infty$.
Using Lemma \ref{lem:primal_graph_recursion} we can write $\hat{R}_{\alpha}(u)$ as
\begin{align*}
\hat{R}_{\alpha}(u) &  =   \inf \Bigg\{ \Psi_{\hat{n}} \bigg(u-  \sum_{ {\hat{e}} \in \hat{E}: \alpha_{{\hat{e}}} >0}  \alpha_{\hat{e}} \Phi_{\hat{e}}  w_{\hat{e}} -  \sum_{ {\hat{e}} \in \hat{E}: \alpha_{{\hat{e}}} =0}  \Phi_{\hat{e}}  w_{\hat{e}} \bigg) + \sum_{\hat{e} \in \hat{E}: \alpha_{\hat{e}}=0}  \bar{R}^{\hat{e}}_{\tilde \alpha^{\hat{e}}}(\Theta_{\hat{e}}  w_{\hat{e}}) \\
& \qquad \qquad + \sum_{\hat{e} \in \hat{E}: \alpha_{\hat{e}}>0}  \bar{R}^{\hat{e}}_{\tilde \alpha^{\hat{e}}}(\Theta_{\hat{e}}  w_{\hat{e}}) \Biggst w_{{\hat{e}}} \in \domain(\Theta_{{\hat{e}}}) \text{ for all } {\hat{e}} \in \hat{E} \Bigg\},
\end{align*}
where, for $\hat{e} \in \hat{E}$, $\bar{R}^{\hat{e}}_{\tilde \alpha^{\hat{e}}} = R(\bar \Gc_{\tilde \alpha^{\hat{e}}}^{\hat{e}})$ and $\bar \Gc_{\tilde \alpha^{\hat{e}}}^{\hat{e}} = (\bar \Gc^{\hat{e}},{\tilde \alpha^{\hat{e}}}) $, with $\bar \Gc^{\hat{e}} =  (G^{\hat{e}}, (\tilde  \Psi_n)_{n\in V^{\hat{e}}}, (\Theta_e)_{e\in E^{\hat{e}}},$ $ (\Phi_e)_{e\in E^{\hat{e}}})$ and $\tilde{\alpha}^{\hat{e}} = (\tilde  \alpha_e)_{e\in E^{\hat{e}}}$ according to Definition \ref{def:extended_graph_weight_zero}, is a regularization graph of height at most $h-1$ with graph structure $G^{\hat{e}} = (V^{\hat{e}},E^{\hat{e}})$ and root node $\hat{n}^{\hat{e}}$. Note that for $\hat e \in \hat E$ such that  $\alpha_{\hat{e}}>0$ we have $\bar{R}^{\hat{e}}_{\tilde \alpha^{\hat{e}}} = \hat{R}^{\hat{e}}_{\alpha^{\hat{e}}}$ with $\hat{R}^{\hat{e}}_{\alpha^{\hat{e}}} = R(\hat \Gc_{\alpha^{\hat{e}}}^{\hat{e}})$ and $\hat \Gc_{\alpha^{\hat{e}}}^{\hat{e}}$ being the modification of the regularization graph $\Gc_{\alpha^{\hat{e}}}^{\hat{e}}$ according to Definition \ref{def:extended_graph_weight_zero}.
Therefore, weak* lower semi-continuity of $\Psi_{\hat{n}}$, the induction hypothesis and \eqref{eq:staykernel}, leading to $\hat{R}^{\hat{e}}_{\alpha^{\hat{e}}}(\Theta_{\hat{e}} z_{\hat{e}})  = 0$ for $\alpha_{\hat e} = 0$, then yields
\begin{align*}
\hat{R}_{\alpha}(u) 
& \leq    \Psi_{\hat{n}} \bigg(u-  \sum_{ {\hat{e}} \in \hat{E}: \alpha_{{\hat{e}}} >0}  \alpha_{\hat{e}} \Phi_{\hat{e}} \tilde w_{\hat{e}} - \sum_{\hat{e} \in \hat{E}:   \alpha_{\hat{e}}=0}  \Phi_{\hat{e}}z_{\hat{e}} \bigg) + \sum_{\hat{e} \in \hat{E}: \alpha_{\hat{e}} >0}  \hat{R}^{\hat{e}}_{\alpha^{\hat{e}}}(\Theta_{\hat{e}} \tilde w_{\hat{e}}) \\
&  \leq \liminf_k \Psi_{\hat{n}} \bigg(u^k-  \sum_{\hat{e} \in \hat{E}}  \alpha^k_{\hat{e}} \Phi_{\hat{e}} \tilde w^k_{\hat{e}}\bigg) + \sum_{\hat{e} \in \hat{E}}  R^{\hat{e}}_{(\alpha^k)^{\hat{e}}}(\Theta_{\hat{e}} \tilde w^k_{\hat{e}}) =  \liminf_k R_{\alpha^k}(u^k).
\end{align*}

Now take $u \in X_{\hat{n}}$ and observe that, since the convergence $\gamma_k \rightarrow 1$ as $k\rightarrow +\infty$ is immediate, the second assertion of \eqref{eq:convergenceweights} follows directly from what we just showed, provided that $R_{\alpha^k} (\gamma_k u) \leq \hat{R}_\alpha (u)$ for every $k \in \N$ holds.
In order to show the latter, we first select $w\in W$ to attain the minimum in the recursive representation of $\hat{R}_{\alpha}(u)$ according to Lemma \ref{lem:primal_graph_recursion} (which is possible by Theorem \ref{thm:graphessentiallycoercive}), noting that we can choose $w_{\hat{e}}=0$ for $\hat{e}\in \hat{E}$ with $\alpha_{\hat{e}}=0$, and that $\alpha^k_{\hat{e}} w_{\hat{e}} \rightarrow \alpha_{\hat{e}} w_{\hat{e}}$ for all $\hat{e} \in \hat{E}$. In particular,  $R_{(\alpha^k)^{\hat{e}}}^{\hat{e}}(\Theta_{\hat{e}}w_{\hat{e}})= \hat{R}_{\alpha^{\hat{e}}}^{\hat{e}}(\Theta_{\hat{e}}w_{\hat{e}}) = 0$ for ${\hat{e}}\in \hat{E}$ with $\alpha_{\hat{e}}=0$.  
Also, define
\[\gamma^{\hat{e}}_k:=  \min \big \{ \prod_{e \in F} \alpha^k_e/\alpha_e \st F \subset E^{\hat{e}} \text{ is either empty or a chain with root }\hat{n}^{\hat{e}}  \text{ and } \alpha_e >0 \ \forall e \in F \big \},\]
using again the convention $ \prod_{e \in \emptyset} \frac{\alpha_e^k}{\alpha_e} = 1$.
Therefore, using the induction hypothesis together with Remark \ref{rem:initialremark} and Proposition \ref{prop:convexineq} we obtain 
\begin{align*}
 R_{\alpha^k}(\gamma_k u) 
& \leq      \Psi_{\hat{n}} \left(\gamma_k u-  \gamma _k\sum_{\hat{e} \in \hat{E}: \alpha_{\hat{e}}>0}  \alpha^k_{\hat{e}} \Phi_{\hat{e}}\left(\frac{\alpha_{\hat{e}}}{\alpha^k _{\hat{e}}} w_{\hat{e}}\right)-  \gamma_k\sum_{\hat{e} \in \hat{E}: \alpha_{\hat{e}}=0}\Phi_{\hat{e}} w_{\hat{e}} \right) \\
 & \quad \  + \sum_{\hat{e} \in \hat{E}:\alpha_{\hat{e}}>0}  R^{\hat{e}}_{(\alpha^k)^{\hat{e}}}\left(\Theta_{\hat{e}} \left(\gamma_k\frac{\alpha_{\hat{e}}}{\alpha^k_{\hat{e}}}w_{\hat{e}}\right)\right)  \\
& \leq   \gamma_k \Psi_{\hat{n}} \left(u-  \sum_{\hat{e} \in \hat{E}: \alpha_{\hat{e}}>0}  \alpha_{\hat{e}} \Phi_{\hat{e}}(w_{\hat{e}})-  \sum_{\hat{e} \in \hat{E}: \alpha_{\hat{e}}=0}\Phi_{\hat{e}} w_{\hat{e}} \right)  \\
&  \quad \  + \sum_{\hat{e} \in \hat{E}: \alpha_{\hat{e}}>0}   \frac{\gamma_k}{\gamma_k^{\hat{e}}}\frac{ \alpha_{\hat{e}}}{\alpha^k _{\hat{e}}} R^{\hat{e}}_{(\alpha^k)^{\hat{e}}}(\gamma_k^{\hat{e}} \Theta_{\hat{e}} w_{\hat{e}})  \\
& \leq    \Psi_{\hat{n}} \left(u-  \sum_{\hat{e} \in \hat{E}: \alpha_{\hat{e}}>0}  \alpha_{\hat{e}} \Phi_{\hat{e}} w_{\hat{e}} -  \sum_{\hat{e} \in \hat{E}: \alpha_{\hat{e}}=0}\Phi_{\hat{e}} w_{\hat{e}} \right)   + \sum_{\hat{e} \in \hat{E}: \alpha_{\hat{e}}>0}  \hat{R}^{\hat{e}}_{\alpha^{\hat{e}}}(\Theta_{\hat{e}}  w_{\hat{e}})   =  \hat{R}_{\alpha}(u), 
\end{align*}
where we used that $\gamma_k \leq 1$ as well as $\frac{\gamma_k}{\gamma_k^{\hat{e}}}\frac{ \alpha_{\hat{e}}}{\alpha^k _{\hat{e}}}\leq 1$ since $\gamma_k \leq \gamma_k^{\hat{e}}\frac{\alpha^k _{\hat{e}}}{\alpha_{\hat{e}}}$. \qedhere

\end{proof}
We are now ready to prove a result that will in particular imply stability for varying parameter $\alpha$ and  convergence for vanishing noise for \eqref{eq:linear_inverse_general}.
\begin{teo}\label{thm:gammaconvergence} Let $R_\alpha = R(\Gc_\alpha)$ with $\Gc_\alpha$ be a regularization graph with weight $\alpha$ and root node $\hat{n}$ such that $X_{\hat{n}}$ is reflexive, and let $Y$ be a Banach space, $K: X_{\hat{n}}\to Y$ be linear and continuous and 
  $S_f,S_{f^k}: Y \to {[{0,\infty}]}$ for $k \in \N$ be proper, convex, lower semi-continuous and coercive discrepancy functionals with $S_f(v)=0$ if and only if $v=f$. Further, assume that $S_{f^k}$ converges to $S_f$ and that $(S_{f^k})_k$ is equi-coercive. Choosing $\delta_k:= S_{f^k}(f)$ (such that $\delta_k \rightarrow 0$ by convergence of $S_{f^k}$), let $(\alpha^k)_k$ in $(0,\infty)^{E}$ and $(\beta_k)_k$ in $(0,\infty)$ be such that
\begin{itemize}
\item[i)] $\beta_k \rightarrow \beta, \, \frac{\delta_k}{\beta_k} \rightarrow 0$, \text{and} \quad  $\alpha^k_e \rightarrow \alpha_e$, $\alpha^k_e \geq  \alpha_e$ for all $e \in E$.
\end{itemize}
In case $\beta=0$, assume additionally that
\begin{itemize}
\item[ii)] there exists $u_0 \in X_{\hat{n}}$ with $\hat{R}_\alpha(u_0) < +\infty$ such that $Ku_0 = f$.
\end{itemize}
Then, for $(u^k)_k$ a sequence of minimizers of \eqref{eq:linear_inverse_general} with parameters $(\beta_k)_k$ and $(\alpha^k)_k$, up to shifts in $\ker(K) \cap L$, with $L$ being the invariant subspace of $R_{\alpha^k}$ (which does not depend on $k$), $(u^k)_k$ has a subsequence weakly converging in $X_{\hat{n}}$. Further, any limit $\hat{u}$ of a subsequence $(u^{k_i})_i $ converging weakly in $X_{\hat{n}}$ solves
\begin{equation} \label{eq:converged_problem}
 \min_{u \in X_{\hat{n}}} S_f(Ku) + \beta \hat{R}_\alpha(u)
  \end{equation}
 in case $\beta>0$ and
\begin{equation} \label{eq:r_minimizing_solution}
 \min_{u \in X_{\hat{n}}} \hat{R}_\alpha(u) \quad \text{s.t. } Ku=f 
 \end{equation}
in case $\beta=0$. Also, in both cases, $\lim_i R_{\alpha^{k_i}} (u^{k_i}) = \hat{R}_\alpha(\hat{u})$.
\end{teo}
\begin{proof} Given the properties we have obtained for $R_\alpha$ and the assumptions on $S_{f^k},S_f$, the proof is now rather direct and we only provide a sketch for the sake of completeness.

At first note that, in case $\beta=0$, existence of a solution $\hat{u}$ to \eqref{eq:r_minimizing_solution} follows using
Theorem \ref{thm:general_reg_existence_linear} with $S_f = \mI_{\{f\}}$, and assumption $ii)$ ensures a finite minimum.
Further, since $\alpha^k_e \geq  \alpha_e$ for all $e \in E$, which yields $\gamma_k=1$ for $\gamma_k$ according to \eqref{eq:gamma_definition}, Theorem \ref{thm:l.s.c} implies that $R_{\alpha^k}(\hat{u}) \rightarrow \hat{R}_{\alpha}(\hat{u})$ and we get
\begin{equation} \label{eq:convergence_central_boundedness_beta_0}
 S_{f^k}(Ku^k) + \beta_k R_{\alpha^k} (u^k) \leq \delta_k + \beta_k R_{\alpha^k}(\hat{u})  \rightarrow 0 \quad \text{ as } \quad k \rightarrow \infty
 \end{equation} 
 using assumption $ii)$. Consequently, using hypothesis $i)$, it also holds that
 \begin{equation} \label{eq:convergence_boundedness_R_beta_0}
R_{\alpha^k}(u^k) \leq  \delta_k/\beta_k +  R_{\alpha^k}(\hat{u})  \rightarrow \hat{R}_\alpha(\hat{u}) \quad \text{ as } \quad n \rightarrow \infty.
 \end{equation}
This implies in particular boundedness of $S_{f^k}(Ku^k)$ and $R_{\alpha^k}(u^k)$.

In case $\beta>0$, we can select $\hat{u}$ to be a solution to \eqref{eq:converged_problem} and by Theorem  \ref{thm:l.s.c} and convergence of $S_{f^k}$ to $S_f$ estimate according to
\begin{equation} \label{eq:convergence_central_boundedness_beta_neq_0}
 S_{f^k}(Ku^k) + \beta_k R_{\alpha^k} (u^k) \leq S_{f^k}(K\hat{u}) + \beta_k R_{\alpha^k}(\hat{u})  \rightarrow S_{f}(K\hat{u}) + \beta \hat{R}_{\alpha}(\hat{u}) \ \text{ as } \ k \rightarrow \infty.
 \end{equation}
In particular, also in this case, both  $S_{f^k}(Ku^k)$ and $R_{\alpha^k}(u^k)$ are bounded. Choosing $Z$ as a complement of $\ker(K) \cap L$ in $L$, such that the projection $P_Z:L \rightarrow Z$ satisfies $\rg(\Id - P_Z) = \ker(K) \cap L$,
 and $\tilde{u}^k := u^k - P_{L}u^k + P_{Z}P_{L} u^k$, we observe that $\tilde{u}^k - u^k \in \ker(K) \cap L$ and, using equi-coercivity of $(S_{f^k})_k$ and Proposition \ref{prop:coercivity_dependence_on_weights}, we can obtain, as in the proof of Lemma \ref{lem:general_existence}, that $(\tilde{u}^k)_k$ is bounded and hence admits a subsequence weakly converging in $X_{\hat{n}}$.

Now take $u \in X_{\hat{n}} $ to be the limit of a subsequence $(u^{k_i})_i$ of $(u^k)_k$ weakly converging in $X_{\hat{n}}$. %
In case $\beta=0$, using weak lower semi-continuity, convergence of $S_{f^k}$ to $S_f$, and that $S_f(v) = 0$ only if $v=f$, it follows from  \eqref{eq:convergence_central_boundedness_beta_0} and \eqref{eq:convergence_boundedness_R_beta_0} that $Ku=f$ and $\hat{R}_\alpha(u) \leq \hat{R}_\alpha(\hat{u})$ as claimed in \eqref{eq:r_minimizing_solution}, and consequently, also that $\lim_i R_{\alpha^{k_i}}(u^{k_i}) = \hat{R}_{\alpha}(u)$.
 In case $\beta>0$, again using weak lower semi-continuity and convergence of $S_{f^k}$ to $S_f$, it follows from \eqref{eq:convergence_central_boundedness_beta_neq_0} and $\hat{R}_{\alpha}(u) \leq \liminf_i R_{\alpha^{k_i}}(u^{k_i})$
 that $u$ solves \eqref{eq:converged_problem}, and that 
 $ E_{k_i}:= S_{f^{k_i}}(Ku^{k_i}) + \beta_{k_i} R_{\alpha ^{k_i}}(u^{k_i}) 
 \rightarrow E_0:= S_{f}(Ku) + \beta \hat{R}_{\alpha }(u) $. 
 If $\limsup_i R_{\alpha^{k_i}}(u^{k_i}) > \hat{R}_\alpha(u)$, then the estimate
 \[ S_f(Ku) \leq \liminf_i \left( E_{k_i} - \beta_{k_i} R_{\alpha^{k_i}}(u^{k_i}) \right) = E_0 - \limsup_i R_{\alpha^{k_i}}(u^{k_i}) < S_f(Ku)
\] 
yields a contradiction, hence also $\lim_i R_{\alpha^{k_i}}(u^{k_i}) =  \hat{R}_{\alpha}(u)$ and the proof is complete. 
\end{proof}
\begin{rem} Theorem \ref{thm:gammaconvergence} is valid for several particular cases which are worth mentioning:
\begin{itemize}
\item If $\alpha>0$ component-wise, then the above results hold for $\hat{R}_\alpha = R_\alpha$.
\item If we fix $\alpha^k=\alpha$ and have $\beta=0$, this is a classical convergence-for-vanishing-noise result for a fixed regularization functional.
\item Regarding both $\beta$ and $\alpha$ as regularization parameters, this is a rather general convergence result for multi-parameter regularization and we refer to \cite{brediesholler2014,naumova2013multi,ito2011multi} for related work.%
\item If we fix $f^k=f$, this is a stability result for varying the parameters $\alpha,\beta$, which is in particular relevant in the context of bilevel optimization, see Section \ref{sec:bilevel} below.
\item Note that $\alpha^k_e \geq \alpha _e$ was only used in combination with Theorem \ref{thm:l.s.c} to ensure that $\lim_{k\rightarrow \infty} R_{\alpha^k} (u) = \hat R_{\alpha} (u)$. In case $R_\alpha$ is positively one-homogeneous, following Remark \ref{rem:parameter_convergence_pos_hom}, this assumption can be dropped. Also, in case $f^k=f$ and $\beta>0$, the assumption can be dropped in case $S_f$ is continuous in the sense that $\lim_{\lambda \rightarrow 1} S_f(\lambda v) = S_f( v)$ for all $v \in \domain(S_f)$.
\item Again, as described in Remark \ref{rem:well_posed_non_reflexive},
the result can be modified to hold without assuming reflexivity of $X_{\hat{n}}$. 
\end{itemize}
\end{rem}

\section{Bilevel optimization} \label{sec:bilevel}
The goal of this section is to show well-posedness of a bilevel optimization problem for learning the weights $\alpha$ in a regularization graph. In order to allow for an arbitrary removal of different subtrees of the graph via setting $\alpha_e =0$, we will need to include an additional penalty on the edge variables $(w_e)_{e \in E}$. To formulate this, we use the notation 
\[
R_\alpha(u,(w_e)_{ e \in E}):= \sum_{n \in V} \Psi_{n} 
\Big( \Theta_{(n^-,n)} w_{(n^-,n)} - \sum_{(n,m) \in E} \alpha_{(n,m)} \Phi_{(n,m)} w_{(n,m)}
\Big)
\]
where again $w_{(\hat{n}^-,\hat{n})} = u$ and $\Theta_{(\hat{n}^-,\hat{n})} = \Id $. Also, we need an assumption based on the invariant subspaces of regularization graph functionals. To formulate this,
first recall the recursive representation of a regularization graph functional $R_\alpha = R(\Gc_\alpha)$ from Lemma \ref{lem:primal_graph_recursion} as
\begin{align}  \label{eq:rec_bilevel}
  R_\alpha(u) = \inf\Big
  \{ \Psi_{\hat{n}} \big(u-  \sum_{\hat{e} \in \hat{E}}\alpha_{\hat{e}} \Phi_{\hat{e}} w_{\hat{e}} \big) + \sum_{\hat{e} \in \hat{E}} R_{\alpha^{\hat{e}}}^{\hat{e}}(\Theta_{\hat{e}} w_{\hat{e}}) \Bigst {w_{\hat{e}} \in \domain(\Theta_{\hat{e}})} \text{ for all } \hat{e} \in \hat{E}
 \Big\}.
\end{align}
Based on this, for $e \in E$, we henceforth denote $M^e:= \Theta^{-1}_e(L^e)$, where  $\Theta_{e}^{-1}$ is the inverse of $\Theta_{e}:\ker(\Theta_{e})^\perp \rightarrow \rg(\Theta_{e})$ (recall that $\ker(\Theta_{e})^\perp:= \rg (\Id - P_{\ker(\Theta_{e})}) \cap \domain(\Theta_{e})$ with $P_{\ker(\Theta_{e})}$ according to Assumption \ref{ass:coerc_fwd_operator}) and $L^e$ is the invariant subspace of the regularization graph functional $R_{\alpha^e}^e = R(\Gc^e_{\alpha^e})$ with $\Gc^e_{\alpha^e}$ the regularization graph corresponding to the subtree of $G$ starting at edge $e \in E$ with functionals, spaces, operators and weights inherited from $\Gc_\alpha$. Note that $M^e$ is finite dimensional by finite dimensionality of $L^e$ and of $\ker(\Theta_e)$ for every $e\in E$. Finally, we denote the projection $P^e : \domain(\Theta_e) \rightarrow M^e$ as
\begin{equation}\label{eq:Pieta}
 P^{e}w := \Theta_{e}^{-1} P_{\rg(\Theta_{e}) \cap L^{e}} \Theta_{e} w + P^{e}_{\ker(\Theta_{e})} w, 
 \end{equation}
 where $P^{e}_{\ker(\Theta_{e})}$ is a projection onto $\ker(\Theta_{e})$, noting that $P^e$ is indeed a projection.

Using these notations, we now provide a lower semi-continuity result that includes vanishing weights as follows.

\begin{lemma}\label{lem:lsc_graph_weights_edge_variables}
Let $\Gc_\alpha$ be a regularization graph with root node $\hat{n}$ and weights $\alpha \in [0,\infty)^{E}$, $R_\alpha = R(\Gc_\alpha)$ and $(\alpha^k)_k $ be a sequence of weights in $(0,\infty)^{E}$ such that $(\alpha^k)_k \rightarrow \alpha$.

Then, with $(u^k)_k$ weak* converging to some $u \in X_{\hat{n}}$ and $((w^k_e)_{e \in E})_k$ a sequence realizing the minimum in \eqref{defregularizer} with $u^k$ for $u$ and $\alpha^k$ for $\alpha$ such that $(P^e w^k_e)_k$ and $(R_{\alpha^k}(u^k,((w^k_e)_{e})_k))_k$ are bounded, $((w^k_e)_{e \in E})_k$ is bounded and admits a subsequence converging weak* to some $(w_e)_{e \in E}$ such that
\begin{equation}\label{eq:lscweights_joint}
R_{\alpha}(u,(w_e)_e) \leq \liminf_{k\rightarrow \infty} R_{\alpha^k}(u^k,((w^k_e)_{e})_k).
\end{equation}
\end{lemma}
Note that, in addition to explicitly including the variables $(w_e)_e$, this lower semi-continuity result differs from the one of Theorem \ref{thm:l.s.c} in the fact that in the limit, only the weights change (possibly to zero), but not the original regularization graph.  
This can be achieved thanks to the boundedness assumption on the sequences $(P^e w^k_e)_k$ that does not always hold true as clarified in the following remark.
\begin{rem}
Consider  the regularization graph functional for $\text{TGV}^2$  (see Section \ref{sec:examples}) according to
\begin{align}\label{eq:remassneeded}
R_{\alpha^k}(u^k)
 &= \inf_{\substack{w_1 \in \BV(\Omega), \\ w_2 \in \mM(\Omega, \R^d), \\w_3 \in \BD(\Omega, \R^d)}} 
\mI_{\{0\}}(u^k-w_1) + \mI_{\{0\}}(\nabla w_1 -  w_2- \alpha_k w_3) + \|w_2\|_{\mM} + \|\symgrad w_3\|_{\mM} 
\end{align}
where $\alpha^k \rightarrow 0$ and $u^k = u$ for every $k$ with $\nabla u \in \ker( \symgrad) \setminus \{0\}$. Then, the sequence $(w^k_1,w^k_2,w^k_3)_k = (u, 0, \nabla u/\alpha^k)_k$ realizes the minimum in \eqref{eq:remassneeded} with $R_{\alpha^k}(u^k) = 0$ for every $k$. However, for edge $3$, we have $M^3 = 
\ker(\symgrad)$ and it holds that $\|P_{M^3} (\nabla u/\alpha^k)\|_{L^{d'}(\Omega, \R^d)}  = (\alpha^k)^{-1} \|\nabla u\|_{L^{d'}(\Omega, \R^d)} \rightarrow +\infty$, showing that in this case, the assumptions of Lemma \ref{lem:lsc_graph_weights_edge_variables} do not hold.
\end{rem}

\begin{proof}[Proof of Lemma \ref{lem:lsc_graph_weights_edge_variables}]
Again we proceed by induction, assuming the result holds true for all regularization graphs of height less than $h$ and assume that the height of $\Gc_\alpha$ is $h$. The case $h=0$ is again immediate and we assume $h\geq 1$. 
Writing
\[
R_{\alpha^k}(u^k,(w^k_e)_{e}) = \Psi_{\hat{n}} \bigg(u^k - \sum_{\hat{e}\in \hat{E}} \alpha^k_{\hat{e}} \Phi_{\hat{e}}w^k_{\hat{e}} \bigg) + \sum_{\hat{e} \in \hat{E}} R^{\hat{e}} _{(\alpha^k)^{\hat{e}}} (\Theta _{\hat{e}}w^k_{\hat{e}},(w^k_{e})_{e \in E^{\hat e}} )
\]
with $\hat{E}\subset E$ the set all edges connected to the root node $\hat{n}$, $R_{ (\alpha^k)^{\hat{e}}} = R(\Gc_{(\alpha^k)^{\hat{e}}}^{\hat{e}})$,  and $(\Gc_{(\alpha^k)^{\hat{e}}}^{\hat{e}}) $ regularization graphs of height less than $h$ (see Lemma \ref{lem:primal_graph_recursion})  and with root node $\hat{n}^{\hat{e}}$ we observe, estimating as in \eqref{eq:theta_rl_coercivity_parameter_convergence_proof} and using boundedness of the $(\alpha^k_{\hat{e}})_k$, for generic constants $C,D,\tilde{C}>0$ independent of $k$, that
\[ \|w^k_{\hat{e}} - P^{\hat{e}}w^k _{\hat{e}} \|_{X_{\hat{e}}} \leq C R_{ (\alpha^k)^{\hat{e}}}^{\hat{e}} (\Theta _{\hat{e}}w^k_{\hat{e}})  + D < \tilde C < \infty ,\]
hence boundedness of $(P^e w^k_e)_k$ for every $e \in E$ implies that $(w^k_{\hat{e}})_k$ is bounded for $\hat{e} \in \hat{E}$. Further, again using the coercivity estimates for $ R^{\hat{e}} _{(\alpha^k)^{\hat{e}}}$, the definition of $P^{\hat e}$, the estimate in \eqref{eq:coercivityinductionnoise} and the continuity of $\Theta_{\hat e}P^{\hat e}$ we obtain
\begin{align*}
 \|\Theta_{\hat{e}}  w^k_{\hat{e}} \| _{X_{\hat{n}^{\hat{e}}}} & \leq  \|\Theta_{\hat{e}}  w^k_{\hat{e}} - \Theta_{\hat{e}} P^{\hat{e}} w_{\hat{e}}^k\|_{X_{\hat{n}^{\hat{e}}}} + \|\Theta_{\hat{e}} P^{\hat{e}} w_{\hat{e}}^k \|_{X_{\hat{n}^{\hat{e}}}}\\
& \leq C  R^{\hat{e}}_{(\alpha^k)^{\hat{e}}}(\Theta_{\hat{e}}  w^k_{\hat{e}})  + D + \|\Theta_{\hat{e}} P^{\hat{e}}w_{\hat{e}}^k \|_{X_{\hat{n}^{\hat{e}}}}  \leq \tilde C < +\infty,
\end{align*}
where again $C,D,\tilde{C}>0$ denote generic constants independent of $k$.
Hence, by weak* compactness and weak* closedness of the $\Theta_{\hat e}$ we obtain that that $w^k_{\hat{e}}\overset{*}{\rightharpoonup} w_{\hat{e}} \in \domain(\Theta_{\hat{e}})$ as well as $\Theta_{\hat{e}} w^k_{\hat{e}} \overset{*}{\rightharpoonup} \Theta_{\hat{e}}w_{\hat{e}}$. The induction hypothesis together with the weak* lower semicontinuity of $\Psi_{\hat n}$ and the weak*-to-weak* continuity of $\Phi_{\hat e}$ implies then the result.
\end{proof}

Consider now a regularization graph $\Gc_\alpha$ with root node $\hat{n}$ and let $R_\alpha= R(\Gc_\alpha) : X_{\hat n} \rightarrow [0,\infty]$ be the associated regularization functional. Let $Z$ be a Banach space such that $Z\hookrightarrow X_{\hat{n}} $ and let $H_1$, $H_2$ be two functionals that penalize the weights $\alpha$ and auxiliary variables $(w_e)_{e \in E}$, respectively. We consider the bilevel optimization problem 
\begin{equation} \label{eq:bilevel} 
\begin{aligned}
 \mathop{\inf_{\alpha \in [0,\infty)^{E},\, \beta\in (0,\infty)\  }}_{u_{\alpha,\beta}\in X_{\hat{n}},\, (w^{\alpha,\beta}_e)_{e \in E}} 
 & \|u_{\alpha,\beta} - \hat u\|_Z  + H_1(\alpha) + H_2((w^{\alpha,\beta}_e)_{e \in E}) \\
&\text{s.t. }  (u_{\alpha,\beta}, (w^{\alpha,\beta}_e)_{e \in E}) \in \argmin_{u \in X_{\hat{n}},\, (w_e)_{e \in E}} S_f(Ku) + \beta R_\alpha(u,(w_e)_{e \in E}),
\end{aligned}
\end{equation}
where $\hat{u}$ is some ground truth datum and $f \approx K\hat u$ a corrupted measurement.
\begin{rem}
Note that this single-datum bilevel setting is a generic model problem for learning parameters from a larger training set $(\hat{u}_m,f_m)_m$. Indeed, the single-datum bilevel setting can be extended to a larger training set by simply vectorizing all involved quantities, for instance.
\end{rem}
We now provide an existence result for the bilevel problem, where we use the convention that for $\beta=\infty$, we have $\beta R_\alpha(u,(w_e)_e) = 0$ if $R_\alpha(u,(w_e)_e)=0$ and $\beta R_\alpha(u,(w_e)_e) = \infty$ else, and for which a concrete example and its assumptions are discussed after its proof below. 

In this existence result, regarding the existence of an optimal parameter $\beta$, it is important to note that in \eqref{eq:bilevel}, the parameter $\beta$ is taken in the open interval $(0,\infty)$. This is necessary as otherwise, existence of a solution to the lower level problem cannot be guaranteed. The following theorem takes this into account by allowing the optimal parameter also to attain the value $0$, in which case it states that existence to the lower level problem with $\beta=0$ also holds, see Remark \ref{rem:bilevel_details} for details.

\begin{teo}\label{thm:bilevel_existence} 
Let $Z$ be a Banach space, $\Gc_\alpha$ be a regularization graph with root node $\hat{n}$ and assume that $X_{\hat{n}}$ is reflexive with $Z\hookrightarrow X_{\hat{n}} $.
Further, let $K:X_{\hat{n}} \rightarrow Y$ with $Y$ a Banach space be linear and continuous, $S_f:Y \rightarrow [0,\infty]$ be a proper, convex, lower semi-continuous and coercive and $H_1:[0,\infty)^{E} \rightarrow [0,\infty]$, $H_2:\bigtimes _{e \in E}X_e \rightarrow [0,\infty]$ be proper and strongly and weak* lower semi-continuous functionals, respectively, such that $H_2(0) = 0$.

Let $\hat{u} \in Z$ be a given ground truth variable such that $S_f$ is the discrepancy with respect to a noisy version $f \in Y$ of $K\hat{u}$. Defining $\tilde{\alpha} \in [0,\infty)^E$ by $\tilde{\alpha}_e = \inf \{ \alpha_e \st \alpha \in \domain(H_1)\}$, further
assume that
\begin{enumerate}
\item[i)] $\domain(R_{\tilde{\alpha}})$ is dense in $X_{\hat{n}}$,
\item[ii)] $S_f$ is continuous and $\domain(S_f)$ open,
\item[iii)] $\domain(S_f \circ K) \cap L_{\tilde{\alpha}} \neq \emptyset$, where $L_{\tilde{\alpha}}$ is the invariant subspace of $R_{\tilde{\alpha}}$ as in Remark \ref{rem:domain_invariant_subspace},
\item[iv)] $\|\alpha\| \leq C H_1(\alpha) + D$ for $C,D>0$,
\item[v)] $ \sum_{e \in E} \| P^e w_e \| \leq C H_2( (w_e)_e) + D$ for $C,D>0$ and with $P^e$ as in \eqref{eq:Pieta}.
\end{enumerate}
Then, there exist $\hat{\alpha} \in [0,\infty)^E$, $\hat{\beta}\in [0,\infty]$, $u_{\hat{\alpha},\hat{\beta}}$ and $(w^{\hat \alpha,\hat \beta}_e)_{e \in E}$ such that $(u_{\hat{\alpha},\hat{\beta}},(w^{\hat \alpha,\hat \beta}_e)_{e \in E})$ solves the lower level problem in \eqref{eq:bilevel} with parameters $(\hat{\alpha},\hat{\beta})$ and such that
\begin{equation} \label{eq:bilevel_extended_existence} 
\begin{aligned}
\|u_{\hat{\alpha},\hat{\beta}} - \hat u\|_Z +  H_1(\hat{\alpha}) + H_2((w^{\hat \alpha,\hat \beta}_e)_{e}) = 
 \mathop{\inf_{ \alpha \in [0,\infty)^E,\,\beta\in (0,\infty),\  }}_{u_{\alpha,\beta}\in X,\,(w^{\alpha,\beta}_e)_{e}} 
  \|u_{\alpha,\beta} - \hat u\|_Z  + H_1(\alpha) + H_2((w^{\alpha,\beta}_e)_{e}) \\
\text{s.t. }  (u_{\alpha,\beta},(w^{\alpha,\beta}_e)_{e }) \in \argmin_{u \in X_{\hat{n}},\, (w_e)_{e}} S_f(Ku) + \beta R_{\alpha}(u,(w_e)_{e }).
\end{aligned}
\end{equation}
\end{teo}

\begin{proof} %
 In case the infimum in the bilevel problem \eqref{eq:bilevel} is infinite, any parameter combination together with a corresponding solution of the lower level problem will be a solution, hence we assume from now on that the infinum in \eqref{eq:bilevel} is finite.
Take $(\alpha^k,\beta_k)_k$ to be a minimizing sequence in $ [0,\infty)^E \times (0,\infty)$ for $\eqref{eq:bilevel}$ with 
$(u^k)_k = (u_{\alpha^k,\beta_k})_k$ and $((w_e^k)_e)_k = ((w^{\alpha^k,\beta_k}_e)_e)_k$
 corresponding sequences of solutions to the lower level problem. Then, obviously $(u^k)_k$ is bounded in
 $Z$ and by $\|\cdot \|_{X_{\hat{n}}} \leq C \|\cdot \|_{Z}$ we obtain a (non-relabeled) subsequence weakly converging to some $u$ in $X_{\hat{n}}$. By the coercivity of $H_1$ (hypothesis $iv)$) we can also assume that, up to a subsequence, $\alpha^k \rightarrow 
 \hat{\alpha} \in \domain(H_1)$. By possibly considering another (non-relabeled) subsequence, we can further achieve that, for each $e \in E$, either $\alpha^k_e>0$ for all $k$ or $\alpha^k _e = 0$ for all $k$. Noting that in the latter case we can remove the subgraphs of $G=(V,E)$ after $e \in E$ with $\alpha^k_e = 0$ for all $k$ without changing the value of $R_\alpha$ or any of the $R_{\alpha^k}$, we can further assume that $\alpha^k_e>0$ for all $k$ and $e \in E$. 
 
At first assume that there exists a subsequence of $(\beta_k)_k$ converging to zero. Then, moving to this subsequence, we obtain for any $z \in  \domain(R_{\tilde{\alpha}}) \subset \domain(R_{\hat{\alpha}}) $ (where $\domain(R_{\tilde{\alpha}}) \subset \domain(R_{\hat{\alpha}}) $ follows from Lemma \ref{lem:weight_equivalence} and the definition of $\tilde{\alpha}$) and $z \in \domain(S_f \circ K)$ that 
\begin{align*}
S_f(Ku) 
&\leq \liminf _k S_f(Ku^k) + \beta_k R_{\alpha^k } (u^k) \leq\liminf _k  S_f(\gamma_k Kz) + \beta_k R_{\alpha^k }(\gamma_k z) \\
& \leq \liminf _k  S_f(\gamma_k Kz) + \beta_k R_{\hat{\alpha} }(z) =  S_f(Kz),
\end{align*}
where we have used that $R_{\alpha^k}(\gamma_k z) \leq  \hat{R}_{\hat{\alpha}}(z) \leq R_{\hat{\alpha}}(z)$ for $\gamma_k $ according to \eqref{eq:gamma_definition} by Theorem \ref{thm:l.s.c} and Lemma \ref{lem:rhat_r_estimate}, and that $S_f$ is continuous on $\domain(S_f)$ with $\domain(S_f)$ open (hypothesis $ii)$).
Density of $\domain(R_{\tilde{\alpha}})$ and continuity of $S_f$ then implies that $(u,(0)_{e \in E})$ is a solution to the lower level problem in \eqref{eq:bilevel_extended_existence} for $\hat{\beta}=0$. Lower semi-continuity of $\|\cdot \|_Z$ and $H_1$, and the fact that $0 = H_2((0)_{e \in E}) \leq \liminf_k H_2 ((w_e^k)_e)$ then yields the claimed optimality of $(\hat{\alpha},\hat{\beta})$ with $\hat{\beta}=0$.

Assume now that $(\beta_k)_k$ is unbounded such that, again by using a non-relabeled subsequence, we can assume that $\beta_k\rightarrow \infty$. Optimality and the estimate \eqref{eq:convergenceweights} then give for any $z \in  \domain(S_f \circ K)$ with $R_{\hat{\alpha}}(z) = 0$ (such a $z$ exists by hypothesis $iii)$ since $L_{\tilde{\alpha}} \subset L_{\hat{\alpha}}$, with $L_{\tilde{\alpha}}$ and $L_{\hat{\alpha}}$ being the invariant subspaces of $R_{\tilde{\alpha}}$ and $R_{\hat{\alpha}}$, respectively) that
\[ S_f(Ku^k) + \beta_k R_{\alpha^k } (u^k,(w_e^k)_e) \leq S_f(\gamma_k Kz) \rightarrow S_f(Kz)  < \infty.
\]
This implies in particular that $(R_{\alpha^k } (u^k,(w_e^k)_e))_k$ is bounded such that, using that $(P^e w_e^k)_k$ is bounded for each $e \in E$ due to coercivity of $H_2$ as in assumption $v)$, by Lemma \ref{lem:lsc_graph_weights_edge_variables}, the sequence $((w_e^k)_e)_k$ admits a subsequence weak* converging to some $(w_e)_e$. Weak* lower semi-continuity then yields
\[ R_{ \hat \alpha } (u,(w_e)_e) \leq \liminf_k R_{\alpha^k } (u^k,(w_e^k)_e) = 0 .\]
Also, from weak lower semi-continuity of $S_f \circ K$, we obtain that
$ S_f(Ku) \leq S_f(Kz) $ and, consequently, that $(u,(w_e)_e)$ solves the lower level problem in \eqref{eq:bilevel_extended_existence} for $(\hat{\alpha},\hat{\beta})$ with $\hat{\beta}=\infty$.
Lower semi-continuity of $\|\cdot \|_Z$ and  $H_1$, and weak* lower semi-continuity of $H_2$  finally implies that $(\hat{\alpha},\hat{\beta})$ is optimal as claimed.

At last assume that, again up to a non-relabeled subsequence, $\beta_k \rightarrow \hat \beta \in (0,\infty)$. 
Then, we get for any $z \in \domain(R_{\hat{\alpha}}) \cap \domain(S_f\circ K)$ (which again exists by hypothesis $iii)$ since $L_{\tilde{\alpha}} \subset \domain(R_{\tilde{\alpha}}) \subset \domain(R_{\hat{\alpha}}) $)  that 
\begin{align*}
 \liminf _k S_f(Ku^k) + \beta_k R_{\alpha^k } (u^k,(w_e^k)_e) \leq\liminf _k S_f(\gamma_k Kz) + \beta_k R_{\alpha^k }(\gamma_k z)   \leq S_f(Kz) +  \beta R_{\hat{\alpha} }(z),
\end{align*}
such that again, $(R_{\alpha^k } (u^k,(w_e^k)_e))_k$ and $(P^e w_e^k)_k$ are bounded and by Lemma \ref{lem:lsc_graph_weights_edge_variables}, we can assume that $((w_e^k)_e)_k$ admits a subsequence weak* converging to some $(w_e)_e$. Lower semi-continuity then yields  
\[
 S_f(Ku) + \beta R_{\hat \alpha } (u,(w_e)_e)  \leq S_f(Kz) +  \beta R_{\hat{\alpha} }(z),
 \]
which shows that $(u,(w_e)_e)$ solves the lower level problem in \eqref{eq:bilevel_extended_existence}. Finally, again lower semi-continuity of $\|\cdot \|_Z$, $H_1$ and weak* lower semi-continuity of $H_2$ imply optimality of $(\hat{\alpha},\hat{\beta})$ as claimed.
\end{proof}
Before discussing the assumptions and results of Theorem \ref{thm:bilevel_existence} in detail, we provide an example.

\medskip

\noindent \textbf{$\TGV^2$-shearlet infimal convolution}. With the notation of Section \ref{sec:examples} we can define a regularization graph $\Gc_\alpha$ to be the one in Section \ref{sec:examples} such that  $R_\alpha=R(\Gc_\alpha):L^2(\Omega) \rightarrow [0,\infty]$ with $\Omega \subset \R^2$ a bounded Lipschitz domain is given as
\[ R_\alpha(u) = \inf_{\substack{ w_1 \in \BD(\Omega), \\ w_2 \in L^2(\R^2)}} \|\nabla (u-\alpha_0  r_\Omega w_2) - \alpha_1 w_1\|_\M + \|\symgrad w_1 \|_\M  + \|\mathcal{SH}w_2\|_1,
\]
where $\alpha \in [0,\infty)^2$ is chosen accordingly. 
Then, for $K :L^2(\Omega) \rightarrow Y$ linear and continuous and some Banach space $Y$, we can consider the bilevel problem 
\begin{equation} \label{eq:bilevel_example}
\begin{aligned}
 \inf_{ \substack{ \beta \in (0,\infty), \\ (\alpha_0,\alpha_1) \in [0,c]^2}} & \|u_{\alpha,\beta} - \hat{u}\|_{L_2(\Omega)} + \mI_{[0,d]} \left(\left\|P_{\ker(\symgrad)} w_1^{\alpha,\beta}\right\|_{L^2(\Omega)}\right)  \\ 
 &  \text{s.t. } (u_{\alpha,\beta},w_1^{\alpha,\beta},w_2^{\alpha,\beta}) \in \argmin _{u,w_1,w_2} \frac{1}{2}\|Ku - f\|_Y ^2 +\beta R_\alpha(u,w_1,w_2),
\end{aligned}
\end{equation}
where $c,d>0$ and $P_{\ker(\symgrad)}:\BD(\Omega) \rightarrow \ker(\symgrad)$ is a projection to the finite dimensional space $\ker(\symgrad) = \{ x \mapsto Ax+b \st A \in \R^{2 \times 2}, b \in \R^2, A^T = -A \}$.
Then, with $e_0$ and $e_1$ being the edges such that $\alpha_{e_0} = \alpha_0$ and $\alpha_{e_1} = \alpha_1$, respectively,
 $H_1(\alpha) = \Ic_{[0,c]^2}((\alpha_{e_0},\alpha_{e_1})) + \Ic_{\{\alpha \st \alpha_e = 1 \text{ for } e \notin \{e_0,e_1\}\}}(\alpha)$ and  $H_2(w_1,w_2) = \mI_{[0,d]}  (\|P_{\ker(\symgrad)} w_1\|_{L^2(\Omega)})$, we have with $\tilde{\alpha}_e = \inf\{\alpha_e \st$ $ \alpha \in \domain(H_1)\}$ that $R_{\tilde{\alpha}}  = \TV $. Hence, $\domain(R_{\tilde{\alpha}}) \supset C_{c}^\infty(\Omega)$ is dense in $X_{\hat{n}} = L^2(\Omega)$ and all assumptions of Theorem \ref{thm:bilevel_existence} hold.

Consequently there exist parameters $(\hat{\alpha},\hat{\beta})$ which are optimal for \eqref{eq:bilevel_example} as stated in Theorem \ref{thm:bilevel_existence}. Denoting by $\infconv$ the infimal convolution operator, we observe for the corresponding parameters $\hat{\alpha}_0$, $\hat{\alpha}_1$ that:
\begin{itemize}
\item If $\hat \alpha_0=0, \hat\alpha_1 >0$, then $R_{\hat{\alpha}} = \TGV^2_{(1,1/\hat{\alpha}_1)}$.
\item If $\hat\alpha_0=\hat\alpha_1 = 0$, then $R_{\hat{\alpha}} = \TV$.
\item If $\hat\alpha_1=0,\hat\alpha_0>0$, then $R_{\hat{\alpha}} = \TV \infconv (\|\cdot\|_1 \circ \mathcal{SH}) $.
\item If $\hat\alpha_1>0,\hat\alpha_0>0$, then $R_{\hat{\alpha}} = \TGV^2_{(1,1/\hat{\alpha}_1)} \infconv (\|\cdot\|_1 \circ \mathcal{SH}) $.
\end{itemize}
Thus, the model is able to learn different functionals by modifying the graph accordingly. This extends directly, e.g., to learning the order of $\TGV$ or the infimal convolution of $\TGV$ with other regularization functionals. The term $\mI_{[0,d]} \left(\|P_{\ker(\symgrad)} w_1^{\alpha,\beta}\|_{L^2(\Omega)}\right)$ puts a constraint on the norm of the projection of the auxiliary variable $w_1^{\alpha,\beta}$ to $\ker(\symgrad)$. Avoiding such a term is also possible, but would lead to different limit functionals in case of vanishing $\alpha$: Without a bound on the elements of $\ker(\symgrad)$, the limit graph in case $\hat \alpha_0=\hat \alpha_1=0 $ would in this example for instance be
\[\hat{R}_{\hat{\alpha}}(u) = \inf_{w \in \ker(\symgrad)} \| \nabla u - w\|_\M\]
instead of $R_{\hat{\alpha}} (u) = \|\nabla u \|_\M$. Hence, in case of using the infimal convolution of functionals with non-trivial invariant subspace, the limit functional still allows to subtract an arbitrary element of this subspace.

\begin{rem} \label{rem:bilevel_details}
We now discuss necessity of the additional density and continuity assumptions of the theorem and the obtained result in more detail.
\begin{itemize}
\item If $\hat{\beta}=0$, the theorem states that $u_{\hat{\alpha},\hat{\beta}}$ is a solution to
\[ \min_{u \in X_{\hat{n}}} S_f(Ku) ,\]
and in particular that a best approximation of the noisy data exists. Note that this is not true in general. In a classical Hilbert space setting with $S_f(v) = \|u-f\|_2 ^2$ for instance, existence of a best approximation for every $f \in Y$ is in fact equivalent to $K$ having closed range \cite{Engl96}.
\item If $\hat{\beta}=\infty$, we see that $u_{\hat{\alpha},\hat{\beta}}$  solves
\[ \min_{u  \in X_{\hat{n}}} S_f(Ku) \quad \text{subject to }R_{\hat{\alpha}}(u) =0.\]
Here,
it can be shown as in Theorem \ref{thm:general_reg_existence_linear} that solution always exists and we could have alternatively used $\beta \in (0,\infty]$ in the bilevel problem \eqref{eq:bilevel}.
\item Density of $\domain(R_{\tilde{\alpha}})$ is only required in case $\hat{\beta}=0$ to ensure optimality over the entire space instead of $\domain(R_{\hat{\alpha}})$. 
In particular, this assumption can be dropped by bounding the admissible $\beta$ away from zero.
\item The assumption $ S_f$ being continuous and $\domain(S_f)$ open is always fulfilled if, for instance, $S_f(u) = \|u-f\|_Y^q$. It can be replaced by the weaker assumption that $S_f(\gamma ^k Kz) \rightarrow S_f(Kz)$ for all $z \in \domain(S_f \circ K)$ and $\gamma^k \in (0,1]$ converging to $1$ by either bounding $\beta$ away from zero or reducing the optimality of $u_{\hat{\alpha},\hat{\beta}}$ in case $\hat{\beta} = 0$ to optimality with respect to all functions in $\domain(R_{\hat{\alpha}})$ instead of the entire space.
\item The assumption $\domain(S_f \circ K) \cap L_{\tilde{\alpha}} \neq \emptyset$ is always true if, for instance, $0 \in \domain (S_f)$. It can be weakened to $\domain(S_f \circ K) \cap  \domain(R_{\tilde{\alpha}}) \neq \emptyset$ if the set of admissible $\beta$ is bounded above.
\item Typical examples for $H_1$ fulfilling the assumption of Theorem \ref{thm:bilevel_existence} would be $H_1 $ that constrains $\alpha_e  \in [0,c]$ for all $e \in E_0$ or penalizes $\sum_{e \in E_0}|\alpha_e|$, with some $ E_0 \subset E$, and fixes $\alpha_e = 1$ for all remaining $e \in E \setminus E_0$. Here, a penalization of $\sum_{e \in E_0}|\alpha_e|$ is expected to promote sparsity of $\alpha$ and hence, a reduced complexity of the optimal regularization graph. The purpose of the constraints $\alpha_e = 1$ for $e \in E \setminus E_0$ is to avoid overparametrization, i.e., the usage of unnecessary parameters. This happens, for instance, in case of splitting nodes, i.e., if $\Psi_n = \mathcal{I}_{\{0\}}$ for some $n$. Further, the constraint $\alpha_e = 1$ for $e \in E \setminus E_0$ can be used to avoid $R_{\tilde{\alpha}} = \mathcal{I}_{\{0\}}$, which is the case if all weights are set to zero and $\Psi_{\hat n} =\mathcal{I}_{\{0\}}$ (see Definition \ref{def:reg_graph}). 

\item The coercivity of $H_2$ is only required on the finite dimensional spaces $M_e$ for all  $e \in E$, and is used to %
allow for the bilevel framework to cut edges of the graph by setting weights to zero. Without this assumption, a similar existence result with $R_\alpha$ being replaced by $\hat{R}_\alpha$ can be obtained.

\end{itemize}
\end{rem}

\section{Conclusions}

In this work, we have introduced \emph{regularization graphs} as a flexible framework for designing regularization functionals for the variational regularization of inverse problems. The proposed framework thoroughly covers existing regularization approaches and allows to define new ones in a simple and constructive way, essentially by drawing corresponding regularization graphs. We have provided a comprehensive analysis of the class of functionals derived from regularization graphs, which in particular includes well-posedness and convergence results for applying this class of functionals in a general inverse problems setting. Furthermore, we have developed and analyzed a bilevel optimization approach that allows to learn an optimal structure and complexity of a regularization graph, and hence of the corresponding regularization functional, from training data.

Future goals are to develop an equally flexible numerical framework for the application of regulariazation graphs to general inverse problems, as well as the numerical realization of the proposed bilevel approach.

\appendix

\section{Appendix}

Here we provide a list extending the examples of Section \ref{sec:examples}, that outlines the representation of different, existing regularization functionals as regularization graphs. Note that, as discussed in Section \ref{sec:algebraic}, also a finite combination of any of those functionals via summation or infimal convolution can again be represented as regularization graph.

\medskip 

\small

\noindent \begin{minipage}{0.45\linewidth}
\noindent \textbf{$\TV-L^q$ infimal convolution \cite{burger2015infimal_infty, burger2016infimal_finite_p}.}
\begin{align*}
 R_\alpha(u) 
  =&  \mathop{\inf_{w_1 \in \BV(\Omega),\, w_2 \in \mathcal{M}(\Omega, \R^d),}}_{w_3\in L^q(\Omega, \R^d) } \mI_{\{0\}}(u- w_1) \\
  & + \mI_{\{0\}}(\nabla w_1  - w_2 - \alpha w_3) + \|w_2\|_{\mM} + \|w_3\|_{L^q}\\
  =& \mathop{\inf_{w\in L^q(\Omega, \R^d) }}\|\nabla u - \alpha w\|_{\mM} + \|w\|_{L^q}
\end{align*}
with $1 < p \leq d'$ and $q \in (1,\infty)$.
\end{minipage}
\small
\hspace*{-0.5cm}\begin{minipage}{0.69\linewidth}
\begin{tikzpicture}[rgraph]

\draw (0,0) pic (p1) {node={$L^p(\Omega)$,$\mI_{\{0\}}$}};
\draw (0.8,0) pic (p2) {node={$\mM(\Omega\comma\R^d)$}};

\draw (1.7,0) pic (p3) {node={$\mM(\Omega\comma\R^d)$,$\|\cdot \|$}};
\draw (1.7,-1) pic (p4) {node={$L^q(\Omega\comma\R^d)$,$\|\cdot \|$}};

\draw[edge] (p1) -- pic{edge={$\emb$,$L^{d'}(\Omega)$,$\nabla$}} (p2);
\draw[edge] (p2) -- pic{edge={$\Id$,$\mM(\Omega\comma\R^d)$,$\Id$}} (p3);
\draw[edge] (p2) -- (0.8,-1) -- pic{edge={$\alpha \emb$,$L^{q}(\Omega\comma\R^d)$,$\Id$}} (p4);
\end{tikzpicture}
\end{minipage}

\medskip
\noindent\rule{\linewidth }{0.4pt}
\medskip

\begin{minipage}{0.65\linewidth}
\noindent \textbf{Infimal convolution of spatio-temporal TV \cite{hollerkunisch2014,holler17ictgvmri_mh}}.
\begin{align*}
R_\alpha(u) & =  \inf_{w_1,w_2 \in \BV(\Omega)} \mI_{\{0\}}(u-w_1-\alpha w_2) + \|\nabla w_1\|_{\mM, \beta_1} + \|\nabla w_2\|_{\mM, \beta_2}\\
& = \inf_{w\in \BV(\Omega)} \|\nabla u - \alpha \nabla w \|_{\mM, \beta_1} + \|\nabla w\|_{\mM, \beta_2}
\end{align*}
where $1 < p \leq d'$  and $\|\cdot \|_{\mM, \beta_i} = \|\cdot \|_{\beta_i}$ are anisotropic norms.
\end{minipage}
\begin{minipage}{0.39\linewidth}
\begin{tikzpicture}[rgraph]
\draw (0,0) pic (p1) {node={$L^p(\Omega)$,$\mI_{\{0\}}$}};
\draw (1,0) pic (p2) {node={$\mM(\Omega\comma\R^d)$,$\|\cdot \|_{\beta_1}$}};

\draw (1,-1) pic (p3) {node={$\mM(\Omega\comma \R^d)$,$\|\cdot \|_{\beta_2}$}};

\draw[edge] (p1) -- pic{edge={$\emb$,$L^{d'}(\Omega)$,$\nabla$}} (p2);
\draw[edge] (p1) -- (0,-1) --  pic{edge={$\alpha \emb$,$L^{d'}(\Omega)$,$\nabla$}} (p3);
\end{tikzpicture}
\end{minipage}

\medskip
\noindent\rule{\linewidth }{0.4pt}
\medskip

\begin{minipage}{0.65\linewidth}
\noindent \textbf{Sum of convex functions of $\TV$ and $\TV^2$ \cite{papafitsoros2014firstandsecondorder_mh}}.
\begin{align*}
R_\alpha(u)  
= &  \inf_{w \in \BV^2(\Omega)} \mI_{\{0\}}(u-w) + \Psi_{f,g}\big((\nabla w,\nabla^2w)\big) \\
= &   \|\nabla u\|_{f} +\|\nabla^2 u\|_{g} 
\end{align*}
\end{minipage}
\begin{minipage}{0.39\linewidth}
\vspace*{-0cm}
\begin{tikzpicture}[rgraph]
\draw (0,0) pic (p1) {node={$L^p(\Omega)$,$ \mI_{\{0\}}$}};
\draw (1,0) pic (p2) {node={$\mM(\ldots)  \times \mM(\ldots) $,$\Psi_{f,g}$}};

\draw[edge] (p1) -- pic{edge={$\emb$,$L^{d'}(\Omega)$,$(\nabla\comma\nabla^2)$}} (p2);

\end{tikzpicture}
\end{minipage}
\vspace{0.1cm}

where $1 < p \leq d'$, $\Psi_{f,g}(z_1,z_2)  = \|f ( z_1)\|_\M+\|g (z_2)\|_\M$ with $\|f(\cdot) \|_\M$, $\|g(\cdot )\|_\M$ appropriate convex functions of measures generalizing $\|\cdot \|_\M$ such that coercivity and weak* lower semicontinuity holds.

\medskip
\noindent\rule{\linewidth }{0.4pt}
\medskip

\begin{minipage}{0.46\linewidth}
\noindent \textbf{General second-order model \cite{brinkmann2019unified,chan2010}.}
\begin{align*}
R_\alpha(u) 
 =&  \inf_{\substack{ w_1 \in \BV(\Omega),\,w_2 \in \M(\Omega, \R^d), \\  w_3 \in \BV(\Omega, \R^d)} }\mI_{\{0\}}(u - w_1)\\
 & + \mI_{\{0\}}( \nabla  w_1 -w_2 - \alpha w_3) +  \|w_2\|_\mathcal{M} + \|\mathcal{A}\nabla w_3\|_\mathcal{M} \\ 
 = &\inf_{w\in\BV(\Omega, \R^d)}   \|\nabla u - w\|_\mathcal{M} + \frac{1}{\alpha}\|\mathcal{A}\nabla w\|_\mathcal{M} 
\end{align*}
\end{minipage}
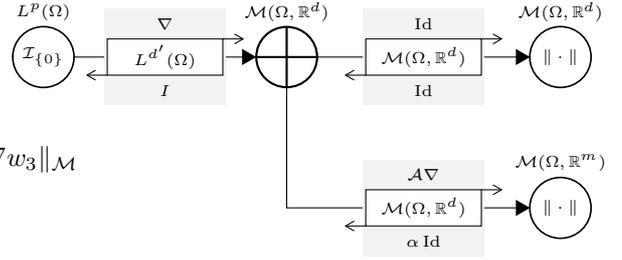
\begin{minipage}{0.6\linewidth}
\hspace*{-1cm}\begin{tikzpicture}[rgraph]

\draw (0,0) pic (p1) {node={$L^p(\Omega)$,$\mI_{\{0\}}$}};
\draw (0.8,0) pic (p2) {node={$\mM(\Omega\comma\R^d)$}};

\draw (1.7,0) pic (p3) {node={$\mM(\Omega\comma\R^d)$,$\|\cdot \|$}};
\draw (1.7,-1) pic (p4) {node={$\mM(\Omega\comma\R^m)$,$\|\cdot \|$}};

\draw[edge] (p1) -- pic{edge={$\emb$,$L^{d'}(\Omega)$,$\nabla$}} (p2);
\draw[edge] (p2) -- pic{edge={$\Id$,$\mM(\Omega\comma\R^d)$,$\Id$}} (p3);
\draw[edge] (p2) -- (0.8,-1) -- pic{edge={$\alpha\Id$,$\mM(\Omega\comma\R^d)$,$\mathcal{A}\nabla$}} (p4);
\end{tikzpicture}
\end{minipage}
\vspace{0.1cm}

where $1 < p \leq d'$, $m \in \N$ and the linear operator $\mathcal{A}:\R^{d \times d} \rightarrow \R^m$ is defined pointwise on $\nabla w$ such that suitable lower semicontinuity and coercivity assumptions hold.

\medskip
\noindent\rule{\linewidth }{0.4pt}
\medskip

\begin{minipage}{0.65\linewidth}
\noindent \textbf{Infimal convolution  of tight frames \cite{Kutyniok13_mh}}.
\begin{align*}
R_\alpha(u) & =  \inf_{w_1,w_2 \in L^2(\Omega)} \mI_{\{0\}}(u-w_1-\alpha w_2) + \|\Phi_1^* w_1 \|_{1} + \| \Phi_2^* w_2 \|_{1}\\
& = \inf_{u = w_1 + w_2}  \|\Phi_1^*  w_1\|_{1} + \frac{1}{\alpha}\| \Phi_2^* w_2 \|_{1}
\end{align*}
where $\Phi^*_i$ are associated with tight frames such as curvelets or Gabor frames \cite{Kutyniok13_mh} and $\|\cdot \|_{1}$ is the extension to $+\infty$ of the $\ell^1$-norm to $\ell ^2$.
\end{minipage}
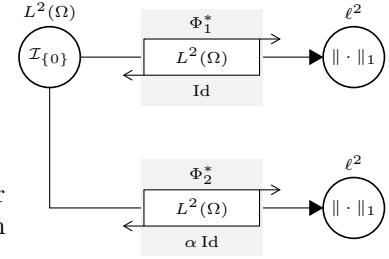
\begin{minipage}{0.34\linewidth}
\begin{tikzpicture}[rgraph]
\draw (0,0) pic (p1) {node={$L^2(\Omega)$,$ \mI_{\{0\}}$}};
\draw (1,0) pic (p2) {node={$\ell^2$,$\|\cdot \|_{1}$}};

\draw (1,-1) pic (p3) {node={$\ell^2$,$\|\cdot \|_{1}$}};

\draw[edge] (p1) -- pic{edge={$\Id$,$L^2(\Omega)$,$\Phi_1^*$}} (p2);
\draw[edge] (p1) -- (0,-1) --  pic{edge={$\alpha\Id$,$L^2(\Omega)$,$\Phi_2^*$}} (p3);
\end{tikzpicture}
\end{minipage}

\medskip
\noindent\rule{\linewidth }{0.4pt}
\medskip

\begin{minipage}{0.44\linewidth}
\noindent \textbf{$\TV_{pwL}$ regularization \cite{burger2019total}}.
\begin{align*}
R_\alpha(u) 
 = & \inf_{\substack{ w_1 \in \BV(\Omega), \\  w_2,w_3 \in \mathcal{M}(\Omega, \R^d)} } \mI_{\{0\}}(u-w_1) \\
 &  + \mI_{\{0\}}( \nabla  w_1- w_2 - \alpha w_3) +  \|w_2\|_\mathcal{M} +\mI_{\{|\cdot | \leq \gamma\}}(w_3) \\ 
 =&  \inf_{w\in\mathcal{M}(\Omega, \R^d)}   \|\nabla u - \alpha w\|_\mathcal{M} \qquad \text{s.t.} \qquad  |w| \leq \gamma
\end{align*}
\end{minipage}
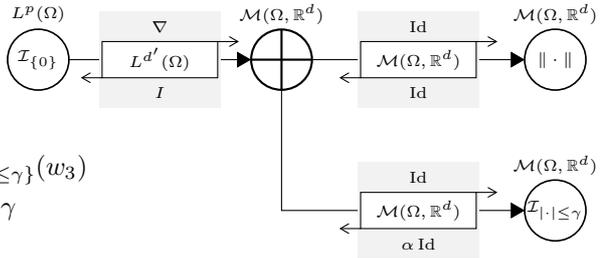
\begin{minipage}{0.65\linewidth}%
\hspace*{-1.2cm}\begin{tikzpicture}[rgraph]

\draw (0,0) pic (p1) {node={$L^p(\Omega)$,$ \mI_{\{0\}}$}};
\draw (0.8,0) pic (p2) {node={$\mM(\Omega\comma\R^d)$}};

\draw (1.7,0) pic (p3) {node={$\mM(\Omega\comma\R^d)$,$\|\cdot \|$}};
\draw (1.7,-1) pic (p4) {node={$\mM(\Omega\comma\R^d)$,$\mI_{|\cdot | \leq \gamma}$}};

\draw[edge] (p1) -- pic{edge={$\emb$,$L^{d'}(\Omega)$,$\nabla$}} (p2);
\draw[edge] (p2) -- pic{edge={$\Id$,$\mM(\Omega\comma\R^d)$,$\Id$}} (p3);
\draw[edge] (p2) -- (0.8,-1) -- pic{edge={$\alpha\Id$,$\mM(\Omega\comma\R^d)$,$\Id$}} (p4);
\end{tikzpicture}
\end{minipage}
\vspace*{0.1cm}

where $|w| \in \M^+(\Omega)$ is the variation of the measure $w \in \M(\Omega,\R^d)$, $\gamma \in \M^+(\Omega)$ is a given positive measure and $|w| \leq \gamma$ means that $\gamma  - |w|$ is a positive measure.

\bigskip

\normalsize

For the sake of completeness, we also provide the proof of the equivalence of a coercivity- and and closed-range assertion for the operators considered in this paper.
\begin{lemma}\label{lem:equivalence_coercivity_closed_range}

 Let $\Theta:\domain(\Theta) \subset X_e \rightarrow X_m$ be a linear operator between Banach spaces $X_e$ and $X_m$ that both admit a predual space and such that bounded sequences in $X_e$ admit weak* convergent subsequences. Further, assume that $\Theta$ is weak* closed and has finite dimensional kernel. Then, there exists $C>0$ and $P_{\ker(\Theta)}:X_e  \rightarrow\ker(\Theta)$ a linear, continuous projection such that
\[ \|w - P_{\ker(\Theta)}w \|_{X_e} \leq C \|\Theta w \|_{X_m} \]
for all $w \in \domain(\Theta)$ if and only if $\Theta $ has closed range.
\begin{proof}
Assuming that the coercivity assertion holds, the closedness of $\rg(\Theta)$ can be proven directly using the weak* closedness of $\Theta$. On other hand, if $\rg(\Theta)$ is closed, then from  \cite[Remark 2.18]{brezisfunctionalanalysis} we  deduce that there exists $\tilde C>0$ such that 
\begin{equation}
\inf_{z \in \ker(\Theta)} \|w  - z\|_{X_e} \leq \tilde C \|\Theta w\|_{X_m} \quad \forall w \in \domain(\Theta).
\end{equation}
In particular, due to finite dimensionality of $\ker(\Theta)$, we deduce the existence of a map $G  : \domain(\Theta) \rightarrow \ker(\Theta)$ such that
\begin{equation}
 \|w  - G(w) \|_{X_e} \leq \tilde C \|\Theta w\|_{X_m} \quad \forall w \in \domain(\Theta).
\end{equation}
Defining now 
\begin{equation}
R(w) := \left\{\begin{array}{ll}
 \|\Theta w\|_{X_m} & w \in \domain(\Theta), \\
 + \infty & \text{otherwise,}
\end{array}
\right.
\end{equation}
and arbitrarily extending $G$ outside $\domain(\Theta)$ to a function $G: X_e \rightarrow \ker(\Theta)$ we obtain  that 
\begin{equation}
 \|w  - G(w) \|_{X_e} \leq \tilde C R(w) \quad \forall w \in X_e\,.
\end{equation}
Finally, applying Lemma \ref{lem:projection_transfer} with $D = 0$ (which yields $\tilde{D}=0$) and $K = X_e$, we obtain the existence of a bounded, linear projection $P_{\ker(\Theta)}$ and a constant $C>0$ such that the claimed coercivity holds.
\end{proof}
\end{lemma}

\small

	\bibliographystyle{plain}
	\bibliography{biblio}

\end{document}